%% file: fano_manifold_topology.tex
\documentclass[leqno,11pt,a4paper]{amsart}
\usepackage{amsthm}
\usepackage{cite}
\usepackage{enumitem}
\usepackage{hyperref}
\usepackage{graphicx}
\usepackage{amssymb}

\usepackage{xcolor}
\hypersetup{
  colorlinks,
  linkcolor={blue!50!black},
  citecolor={blue!50!black},
  urlcolor={blue!80!black}
}
\usepackage{arydshln}
\usepackage{multirow}
\usepackage[all]{xy}
\usepackage{tikz}
\usepackage{tikz-cd}
\usepackage{booktabs}
\usepackage{longtable}
\usepackage{pdflscape}
\usepackage{colortbl}
\usepackage{subfigure}

\setlist[enumerate]{labelsep=*, leftmargin=1.5pc}
\setlist[enumerate]{label=\normalfont(\roman*), ref=\roman*}
\usepackage[margin=2.7cm]{geometry}
\graphicspath{{images/}}
\numberwithin{figure}{section}
\theoremstyle{plain}
\newtheorem{thm}{Theorem}[section]
\newtheorem{pro}[thm]{Proposition}
\newtheorem{lem}[thm]{Lemma}
\newtheorem{cor}[thm]{Corollary}

\theoremstyle{definition}
\newtheorem{dfn}[thm]{Definition}
\newtheorem{rem}[thm]{Remark}
\newtheorem{eg}[thm]{Example}

\newtheorem{cons}[thm]{Construction}
\DeclareMathOperator{\Ima}{Im}

\DeclareMathOperator{\coker}{coker}

\DeclareMathOperator{\PD}{PD}
\DeclareMathOperator{\Pic}{Pic}

\DeclareMathOperator{\Div}{Div}

\DeclareMathOperator{\GL}{{GL}}

\DeclareMathOperator{\Vect}{Vect}
\DeclareMathOperator{\Int}{Int}
\DeclareMathOperator{\convhull}{conv}

\newcommand{\fs}{\mathfrak{s}}

\newcommand{\cA}{\mathcal{A}}
\newcommand{\fC}{\mathfrak{C}}

\newcommand{\cG}{\mathcal{G}}
\newcommand{\cO}{\mathcal{O}}
\newcommand{\QQ}{{\mathbb{Q}}}
\newcommand{\RR}{{\mathbb{R}}}
\newcommand{\PP}{{\mathbb{P}}}

\newcommand{\ZZ}{{\mathbb{Z}}}

\newcommand{\FF}{\mathbb{F}}
\newcommand{\CC}{\mathbb{C}}

\newcommand{\TT}{\mathbb{T}}

\newcommand{\fS}{\mathfrak{S}}

\newcommand{\bM}{\mathbb{M}}
\newcommand{\cP}{\mathcal{P}}

\newcommand{\cF}{\mathcal{F}}

\newcommand{\V}[1]{\operatorname{verts}\left({#1}\right)}
\newcommand{\Edges}[1]{\operatorname{edges}\left({#1}\right)}
\newcommand{\Cone}[1]{\operatorname{Cone}\left({#1}\right)}
\newcommand{\Objects}[1]{\operatorname{Objects}\left({#1}\right)}
\newcommand{\Area}[1]{\operatorname{Area}\left({#1}\right)}

\renewcommand{\tilde}{\widetilde}

\renewcommand{\bar}{\overline}
\newcommand{\conv}[1]{\operatorname{conv}\left({#1}\right)}

\newcommand{\MM}[2]{\mathrm{MM}_{#1\text{--}#2}} % The Mori-Mukai number

\begin{document}
%-------------------------------------------------------------------------------
\author[T.~Prince]{Thomas Prince}
\address{Mathematical Institute\\University of Oxford\\Woodstock Road\\Oxford\\OX2 6GG\\UK}
\email{thomas.prince@magd.ox.ac.uk}% Prince
%-------------------------------------------------------------------------------
%\keywords{}
%\subjclass[2000]{}
%-------------------------------------------------------------------------------
\title[Topology of Fano Manifolds]{Lagrangian torus fibration models of Fano threefolds}
\maketitle
%-------------------------------------------------------------------------------
\begin{abstract}
	Inspired by the work of Gross on topological Mirror Symmetry, we construct candidate Lagrangian torus fibration models for the $105$ families of smooth Fano threefolds. We prove, in the case the second Betti number is one, that the total space of each fibration is homeomorphic to the expected Fano threefold, and show that the numerical invariants coincide for all $105$. Our construction relies on a notion of toric degeneration for affine manifolds with singularities, and the correspondence we obtain between polytopes and Fano manifolds is compatible with that appearing in the work of Coates--Corti--Kasprzyk \textit{et al.}\ on Mirror Symmetry for Fano manifolds. 
\end{abstract}


\input{introduction}
\input{affine}
\input{smoothing_polytope}
\input{degeneration_data}
\input{euler_and_degree}

\input{betti_numbers}
\input{top_classification}
\input{examples}
\input{outstanding}

\appendix
%\begin{landscape}

\input{appendices}

%----------------------------------------------------------------------
\section{Tables of Invariants}
\label{sec:tables}
In this appendix we compile tables summarising the $105$ topological constructions of Fano $3$-folds. Unless stated otherwise we apply the method described in \S\ref{sec:smooth_decompositions}, applied to the reflexive polytope with the indicated PALP ID. We indicate those cases for which there is a choice of smooth Minkowksi decomposition, many of which are treated separately in \S\ref{sec:examples}. Note that since the polytopes associated to the toric degenerations of Fano threefolds with $-K_X$ not very ample are not reflexive we do not specify a PALP ID in these cases.

\input{tables_of_invariants}
%\end{landscape}

%----------------------------------------------------------------------

%-------------------------------------------------------------------------------
\bibliographystyle{plain}
\bibliography{bibliography}
%-------------------------------------------------------------------------------
\end{document}

%% file: introduction.tex
% !TEX root = fano_manifold_topology.tex

%----------------------------------------------------------------------
\section{Introduction}
%----------------------------------------------------------------------

The classification of three-dimensional Fano manifolds, that is, of  smooth projective varieties with ample anti-canonical class, is one of the most famous results in modern Algebraic Geometry. There are $105$ deformation families of Fano manifolds in dimension three, of these families $98$ have very ample anti-canonical bundle. The classification was completed by Mori--Mukai~\cite{Mori--Mukai:Turin,Mori--Mukai:Erratum,Mori--Mukai:Kinosaki,Mori--Mukai:Tokyo} building on work of Fano and Iskovskikh~\cite{Iskovskih:1,Iskovskih:2}.

In this article we describe a topological model for each three dimensional Fano manifold. Each model $X$ is a topological manifold together with a continuous map $X \rightarrow B^3$ to a three-dimensional ball, giving $X$ the structure of a \emph{torus fibration with simple singularities}, defined by Gross~\cite{G01,DBranes09}, and described in \S\ref{sec:affine_manifold}. Moreover, following work of Casta\~{n}o-Bernard--Matessi~\cite{CBM09}, we see that after making suitable local adjustments the fibration can be given the structure of a Lagrangian fibration on a symplectic manifold.

The constructions of these models are inspired by the construction of Gross~\cite{G01} of a topological torus fibration on a (Calabi--Yau) quintic threefold and its mirror--dual manifold. In~\cite{G01} Gross establishes a topological version of the famous Mirror symmetry conjecture of Strominger--Yau--Zaslow~\cite{SYZ} (the SYZ conjecture) for the quintic threefold: demonstrating that the quintic threefold and its mirror mirror manifold carry dual (topological) torus fibrations which interchange cohomological data as expected under Mirror Symmetry. The primary goal of the current work is to obtain a suitable extension of this construction of a torus fibration on a quintic threefold to the Fano threefolds.

Our first main result is the identification, up to homeomorphism, of each of the rank one Fano threefolds with its topological model.

\begin{thm}
	\label{thm:rank_one_models}
	Let $X$ be a Fano threefold with Picard rank one, there is an affine manifold with simple singularities $B$ such that the total space of the torus fibration
	\[
	\pi \colon \breve{X}(B) \rightarrow B
	\]
	is homeomorphic to $X$. Moreover the cycle $D := [\pi^{-1}(\partial B)]  \in H^2(\breve{X}(B),\ZZ)$ has triple self-intersection $D^3 = -K_X^3$, and the index of $D$ in $H^2(\breve{X}(B),\ZZ)$ is equal to the Fano index of $X$.
\end{thm}

The definition of affine manifolds $B$ with simple singularities, as well as the definition of the torus fibration $\breve{X}(B) \rightarrow B$ determined by $B$, is given in \S\ref{sec:affine_manifold}, and is central to all the constructions we consider in this article. Indeed, following the treatment given in~\cite{DBranes09}, a torus fibration with singularities can be reconstructed from such an affine manifold. As we shall see, the manifold $\breve{X}(B)$ is closely related to the cotangent bundle of the affine manifold $B$ and, via the results of \cite{CBM09}, the canonical symplectic structure on the cotangent bundle of $B$ extends to endow $\breve{X}(B)$ with a symplectic structure.

\begin{cor}
	\label{cor:Lagrangian}
	Given a rank one Fano threefold $X$ there is a symplectic manifold $Y$ homeomorphic to $X$ such that $Y$ has a (piecewise smooth) Lagrangian fibration with base $B'$, obtained from the $B$ determined by Theorem~\ref{thm:rank_one_models} by a \emph{localised thickening} of the discriminant locus of $B$.
\end{cor}

The definition of localized thickening is given in \cite{CBM09}, and the fibration we obtain enjoys the properties listed in the main theorem of \cite{CBM09}.

\begin{rem}
	Note that since, in our setting, the affine manifold $B$ has boundary, the map $\pi \colon \breve{X}(B) \rightarrow B$ can only be Lagrangian away from $\partial B$. However there is a symplectic stratification of the boundary such that on each stratum $\pi$ is Lagrangian. Note that this is completely analogous to the moment map of a toric variety, which also ceases to be Lagrangian at fibres over the boundary of the moment polytope.
\end{rem}

Our second main result is that for Fano threefolds of rank $\geq 2$ the topological models we provide are fake Fano threefolds: their numerical invariants coincide with those of the Fano threefolds.

\begin{thm}
	\label{thm:invariants_match_up}
	Let $X$ be a Fano threefold, there is an affine manifold with simple singularities $B$ such that the total space of the compactified torus fibration
	\[
	\pi \colon \breve{X}(B) \rightarrow B
	\]
	has $b_k(\breve{X}(B)) = b_k(X)$ for all $k$, and $\pi_1(\breve{X}(B)) = 0$. Moreover the cycle $D := [\pi^{-1}(\partial B)]  \in H^2(\breve{X}(B),\ZZ)$ has triple self-intersection $D^3 = -K_X^3$.
\end{thm}

There are Lagrangian models of these torus fibrations, applying the results of~\cite{CBM09}, in analogy with Corollary~\ref{cor:Lagrangian}.

\begin{rem}
	The important distinction for us between the rank one case and the higher rank cases is that the class $D := [\pi^{-1}(\partial B)]$ generates the second rational cohomology group in precisely the rank one case. Since our computation of the intersection form and characteristic classes $w_2(\breve{X}(B))$,~$p_1(\breve{X}(B))$ relies on the identification of explicit cycles (as does the analogous computation in \cite{G01}) we would need to construct additional cohomology classes for the cases $b_2(\breve{X}(B)) \geq 2$, and we do not attempt this here. 
\end{rem}

\begin{rem}
	We also comment on an important connection with the Gross--Siebert program~\cite{Gross--Siebert,GS06}. In the context considered by Gross--Siebert the affine manifold with singularities $B$ is determined by a choice of log structure on the central fibre of a toric degeneration. The algorithm explained in \cite{Gross--Siebert} describes how, under certain hypotheses, to pass from this input data to a formal family deforming this central fibre. A topological model for the general fibre of this family is given by the \emph{Kato--Nakayama space}~\cite{KN99}, constructed from the log structure on the central fibre. It is expected that in this context the corresponding Kato--Nakayama space (with fixed phase) is homeomorphic to $\breve{X}(B)$. Were these remarks made into theorems in this context the current work would become a topological analysis of the general fibre of a toric degeneration of a Fano threefold from logarithmic degeneration data associated to the central fibre.
\end{rem}
\begin{rem}
	The Kato--Nakayama space is also studied in the context of the Gross--Siebert program in the recent work of Arg\"{u}z--Siebert~\cite{AS16}, which studies certain real structures in these spaces. It would be interesting know whether the approach taken in~\cite{AS16} yields interesting orientable real Lagrangians in any of the Fano threefolds.
\end{rem}

The manifolds we construct are closely related to the work of Coates--Corti--Galkin--Golyshev--Kasprzyk on Mirror Symmetry for Fano manifolds. In the paper~\cite{CCGGK} the authors identify candidates for mirror K$3$ fibrations for three-dimensional Fano manifolds, and in \cite{CCGK} the authors find explicit examples of mirror fibrations for each of the Fano threefolds. Each such fibration is determined by a regular function $f$ on a (three-dimensional) complex torus and the authors of \cite{CCGK} compare the Picard--Fuchs equations of $f$ with the Quantum Differential Equations of each of the Fano threefolds. It is conjectured in \cite{CCGGK} that the toric variety defined by the Newton polytope of $f$ is the central fibre of a degeneration of the corresponding Fano manifold. In this article we construct a candidate torus fibration models for a given Fano threefold via a topological smoothing of a toric variety the Fano threefold is expected to degenerate and computing its invariants. Thus we have an automatic compatibility between our results and the conjecture of \cite{CCGGK}.

\begin{thm}
	\label{thm:toric_degenerations}
	Given a Fano threefold $X$ with very ample anti-canonical bundle the affine manifold $B$ we consider admits a \emph{polyhedral degeneration} (see~\S\ref{sec:smoothing_polytope}) to a reflexive polytope $P$, and determines a Minkowski decomposition of the facets of $P^\circ$. The induced correspondence between polytopes and Fano manifolds is compatible with the correspondence of~\cite{CCGK,CCGGK} predicted by Mirror Symmetry.
\end{thm}
 
\begin{rem}
	\label{rem:quantum_periods}
	The mirror correspondence in \cite{CCGK,CCGGK} uses the notion of a Minkowski polynomial $f$ associated to a reflexive polytope $P$ and a collection of Minkowski decompositions of its facets. In the notation used in this article this mirror correspondence relates a reflexive polytope $P^\circ$ to a Fano manifold $X$ if and only if the regularised quantum differential operator of $X$ is equal to the Picard--Fuchs operator of a \emph{Minkowski polynomial} with Newton polytope $P^\circ$.
\end{rem}

The majority of this article is devoted to constructing models for the $105$ Fano threefolds, and proving Theorems~\ref{thm:rank_one_models} and~\ref{thm:invariants_match_up}. In \S\ref{sec:smoothing_polytope} we describe how to obtain a candidate $B$ for a given family of Fano manifolds. In general, we fix a polytope $P$ from the lists appearing in~\cite{CCGK} and construct an affine manifold admitting a polyhedral degeneration (a concept introduced in \S\ref{sec:smoothing_polytope}) to $P^\circ$, the polar polytope to $P$. We describe three techniques for producing such a degeneration, depending on the structure of the polytope $P$ we are attempting to smooth in \S\ref{sec:smooth_decompositions}, \S\ref{sec:complete_intersections}, and \S\ref{sec:product_constructions} respectively.

The first step in proving Theorem~\ref{thm:invariants_match_up} is to compute the Euler number of $\breve{X}(B)$ for a given affine manifold $B$. We present a simple formula for $e(\breve{X}(B))$ in \S\ref{sec:euler_characteristic} in terms of data attached to a polytope to which $B$ degenerates, and give a topological proof of a combinatorial identity for reflexive polytopes involving the number $24$. In \S\ref{sec:betti_numbers} we express the second Betti number of the torus fibration $\breve{X}(B)$ in terms of combinatorial data attached to the degeneration of $B$. This data involves the computation of a limit of a system of vector spaces closely related to the one-skeleton of $P$. In many cases this system of vector spaces can be interpreted as a constructible sheaf on the one-skeleton of $P$, related to a sheaf appearing in the work of Itenberg--Katzarkov--Mikhalkin--Zharkov~\cite{IKMZ16} on Tropical Homology.

Given formulas for the Betti numbers of $\breve{X}(B)$ (Proposition~\ref{pro:euler_number} and Theorem~\ref{thm:picard_rank}), the proof of Theorem~\ref{thm:invariants_match_up} is reduced to a case-by-case computation. We present a number of sample calculations in \S\ref{sec:examples} and a table of all $105$ Fano manifolds is given in Appendix~\ref{sec:tables}. To complete the proof of Theorem~\ref{thm:rank_one_models} we need to compute further topological invariants to apply the classification result of Jupp~\cite{J73}, which provides the classification of simply connected $6$-manifolds with torsion free homology. This result is the extension of the result of Wall~\cite{Wall66}, of spin $6$-manifolds under the same hypotheses. These additional invariants are computed in \S\ref{sec:top_classification}.

We also wish to highlight another connection with polyhedral combinatorics. In dimension two there is a well understood theory of \emph{mutation} of polygons~\cite{A+,ACGK,AK14}, capturing the $\QQ$-Gorenstein toric degenerations of log del~Pezzo surfaces. A similar theory of mutations exists in higher dimensions, although currently without such a precise geometric interpretation. The formulae we provide to compute numerical invariants of Fano threefolds provide \emph{mutation invariants} of the polytope in dimension three. If we could suitably generalise these formulas these would directly generalise the notion of singularity content in dimension two.

%----------------------------------------------------------------------
\subsection*{Acknowledgements}
%----------------------------------------------------------------------

We thank Tom Coates, Alessio Corti, Alexander Kasprzyk, Mark Gross, and the members of the Fanosearch group at Imperial College London for many useful conversations. We also thank Bal\'azs Szendr\"oi for suggesting a number of corrections. TP was supported by an EPSRC Doctoral Prize Fellowship, Tom Coates' ERC Grant 682603, and a Fellowship by Examination at Magdalen College, Oxford.

%% file: affine.tex
% !TEX root = fano_manifold_topology.tex

%----------------------------------------------------------------------
\section{Affine manifolds with singularities}
\label{sec:affine_manifold}
%----------------------------------------------------------------------

In this section we review the necessary material on affine manifolds, and introduce local models of the affine manifolds we use throughout this article. While (to our knowledge) the definition of affine manifold with corners and singularities does not appear elsewhere, none of this section is original and follows the treatments appearing in~\cite{CBM09,DBranes09}.

\begin{rem}
	The use of affine manifolds is motivated by, and closely linked to, the study of topological and Lagrangian torus fibrations. While we do not recall the explicit constructions of torus fibrations from affine manifolds in this section, they are fundamental to the proofs of our main results, and are described in Appendix~\ref{sec:torus_fibrations}.
\end{rem}

\begin{dfn}
	An \emph{(integral) affine manifold} $B$ is an $n$-dimensional topological manifold equipped with a maximal atlas $\cA$ whose transition functions are contained in $\ZZ^n \ltimes \GL(n,\ZZ)$. We refer to $\cA$ as an \emph{affine structure} on $B$.
\end{dfn}

\begin{rem}
	Since all affine manifolds we consider are integral we will suppress this adjective throughout this article. We note however that the term \emph{affine manifold} typically refers to a manifold with transition functions contained in $\RR^n \ltimes \GL(n,\RR)$, introduced and developed by Bishop--Goldman~\cite{BG68}, Auslander~\cite{A64}, and Hirsch--Thurston~\cite{HT75}. Note that our notion of integral affine manifold agrees with that of~\cite{G13}, but differs from that used in~\cite{CBM09}. The notion of integral affine manifold used in~\cite{CBM09} coincides with the notion of \emph{tropical affine manifold} appearing in~\cite{G13}. We note that many (though not all) of our results only rely on the tropical affine structure.
\end{rem}

For the remainder of this article we will be interested in the cases $n=2$ or $3$. We also need to extend the definition to take two important phenomena into account: first we need to allow the affine manifold to have a boundary and corners, second we need to allow certain \emph{singularities} to appear in the affine structure. Recall that a rational cone in $\RR^n$ is said to be smooth if it is mapped to $\RR^{n-k}\times \RR^k_{\geq 0}$ for some $k \in \{0, \ldots, n\}$ by an integral linear isomorphism.

\begin{dfn}
	An \emph{affine manifold with corners} is an $n$ dimensional topological manifold with boundary with a maximal atlas $\cA$ whose transition functions are contained in $\RR^n \ltimes \GL(n,\ZZ)$. Moreover for each point $b \in \partial B$ there is a chart in $\cA$ which sends a neighbourhood of $b$ to a neighbourhood of the origin in a smooth cone in $\RR^n$.
\end{dfn}

\begin{rem}
	Given an affine manifold with corners there is a stratification of $\partial B$:
	\[
	\varnothing = (\partial B)_{-1} \subset (\partial B)_0 \subset (\partial B)_1 \subset (\partial B)_2 = \partial B,
	\]
	such that neighbourhoods of points in $(\partial B)_i$ are identified with neighbourhoods of the origin in $\RR^i\times\RR^{3-i}_{\geq 0}$. If $(\partial B)_0 = (\partial B)_1 = \varnothing$ we say that $B$ has a \emph{smooth boundary}, and in this case $\partial B$ is itself an affine manifold. Note that it is possible that $(\partial B)_0 = \varnothing$ while $(\partial B)_1 \neq \varnothing$, see Example~\ref{eg:first_affine_manifold}.
\end{rem}

\begin{dfn}
	An \emph{affine manifold with corners and singularities} is a triple $(B,\cA,\Delta)$ where
	\begin{itemize}
		\item $B$ is a topological manifold with boundary.
		\item $\cA$ is an affine structure on $B \setminus \Delta$.
		\item $\Delta$ is a finite union of locally closed submanifolds of codimension at least two.
		%\item $\cS = \{S_{-1} \subset S_0 \subset S_1 \subset S_2 = \partial B\}$ is a stratification of $\partial B$. We insist that $\Delta \cap S_1 = \varnothing$.
	\end{itemize}
	We insist that $(\partial B)_1 \cap \Delta = \varnothing$. We will refer to the components of $(\partial B)_0$ as \emph{vertices of $B$} and to the components of $(\partial B)_1$ as \emph{edges of $B$}.
\end{dfn}

\begin{rem}
	One can drop the assumption that $(\partial B)_1 \cap \Delta = \varnothing$, although we never consider affine manifolds of this form, and to do so would require developing the appropriate local model for a torus fibration over a neighbourhood of such a point. 
\end{rem}

We will use the term `affine manifold' from now on as shorthand for `integral affine manifold with corners and singularities'. All the affine manifolds we consider in this article are of a particularly simple form: $\Delta$ is always the image (under a regular embedding) of a graph $\Gamma$ whose vertices are either trivalent and map to $B \setminus \partial B$ or univalent and map into $\partial B$. We will define $B_0$, the smooth locus to be the complement of $\Delta$ in $B$.

\begin{rem}
	\label{rem:monodromy}
	Given a point $b$ of $\Delta$ not contained in $\partial B$, the affine structure in a sufficiently small neighbourhood of $b$ is determined by the \emph{monodromy of the lattice of integral vectors, $\Lambda \subset TB_0$}. In fact a (tropical) affine structure on a smooth manifold $M$ is equivalent to the data of a flat, torsion free connection on $TM$, and a covariant lattice $\Lambda \subset TM$.
\end{rem}

\begin{eg}
	The fundamental example for all the constructions we use is the focus-focus singularity in dimension two, see~\cite{KS,S02}. This is an affine structure on $B := \RR^2$ (with co-ordinates $x$,$y$) defined by the charts:
	
	\begin{align*}
	U_1 := \RR^2 \setminus \{y = 0, x \leq 0\}, &&  U_2 := \RR^2 \setminus \{ y = 0, x \geq 0 \}
	\end{align*}
	on $B_0 := \RR^2 \setminus \{0\}$ (in other words, $\Delta = \{0\}$). Let $\phi_i \colon U_i \rightarrow \RR^2$, $i = 1,2$ be maps such that the transition function $\phi_2 \circ \phi^{-1}_1$ restricted to the image of the connected component $\{y > 0\}$ of $U_1 \cap U_2$ is given by the matrix
	\[
	\begin{pmatrix}
	1 & -1 \\
	0 & 1
	\end{pmatrix},
	\]
	and the transition function on $\{y < 0\}$ is the identity map.
\end{eg}

In light of Remark~\ref{rem:monodromy}, and the detailed descriptions of the local models given in \cite[\S{3}]{CBM09}, we identify the affine structures near a point of $\Delta$ by giving the local monodromy of $\Lambda$ in loops around $\Delta$ in suitable co-ordinates. While we use the descriptions given in \cite{CBM09} analogous fibrations have appeared under various names in the literature; as positive and negative fibres in \cite{Gross:Egs}; as $(2,1)$ or $(1,2)$ fibres in earlier work of Gross \cite{G01}; and as type II and III fibres in the work of W.-D. Ruan \cite{Ruan:hypersurfaces}.
 
\begin{enumerate}
	\item $b \in B$ is not contained in $\Delta$, then the affine structure identifies a neighbourhood of $b$ with a neighbourhood of the origin in $\RR^{n-k} \times \RR_{\geq 0}^k$ for some $k \in \{0,1,2,3\}$.
	\item $b \in \Delta$ is the image of a point on an edge of $\Gamma$, the monodromy of $\Lambda$ about such an edge in a suitable basis is equal to
	\[
	\begin{pmatrix}
	1 & 1 & 0 \\
	0 & 1 & 0 \\
	0 & 0 & 1
	\end{pmatrix}
	\]
	\item $b \in \Delta$ is a \emph{negative} trivalent node. Let $b' \in B_0$ be a point near $b$ and $\gamma_i$, $i \in \{1,2,3\}$ be simple loops around each leg of $\Delta$ near $b$ such that $\gamma_1\gamma_2\gamma_3 = 1 \in \pi_1(B_0,b')$, then there is a basis of $T_{b'}B$ such that the monodromy matrices corresponding to $\gamma_i$ are:
	\begin{align*}
	\begin{pmatrix}
	1 & 0 & 1 \\
	0 & 1 & 0 \\
	0 & 0 & 1
	\end{pmatrix}
	&&
	\begin{pmatrix}
	1 & 0 & 0 \\
	0 & 1 & 1 \\
	0 & 0 & 1
	\end{pmatrix}
	&&
	\begin{pmatrix}
	1 & 0 & -1 \\
	0 & 1 & -1 \\
	0 & 0 & 1
	\end{pmatrix}.
	\end{align*}
	\item $b \in \Delta$ is a \emph{positive} trivalent node. Let $b'$ and $\gamma_i$ for $ i \in \{1,2,3\}$ be defined as in the case of the negative node, then there is a basis of $T_{b'}B$ such that the respective monodromy matrices are equal to:
	
	\begin{align*}
	\begin{pmatrix}
	1 & 0 & 1 \\
	0 & 1 & 0 \\
	0 & 0 & 1
	\end{pmatrix}
	&&
	\begin{pmatrix}
	1 & 1 & 0 \\
	0 & 1 & 0 \\
	0 & 0 & 1
	\end{pmatrix}
	&&
	\begin{pmatrix}
	1 & -1 & -1 \\
	0 & 1 & 0 \\
	0 & 0 & 1
	\end{pmatrix}.
	\end{align*}	
	
	\item $b \in \Delta$ is a univalent vertex, the affine structure is the product of a focus-focus singularity with a half open interval, see Example~\ref{eg:boundary_singularity}.
\end{enumerate}

The choice of the basis of $\Lambda$ in each of these cases, as well as a detailed description of the form $\Delta$ takes in each case is given in~\cite[\S{3}]{CBM09}. For example the affine structure around a general point in $\Delta$ is modelled in on the product $U \times I$ where $U$ is a neighbourhood of a focus-focus singularity and $I$ is a small open interval. This model may then be perturbed by making $\Delta$ the graph of a function $\tau \colon I \rightarrow U$ and keeping the monodromy matrix (with the same basis of $\Lambda_b$ for a fixed $b \notin \Delta$) the same.

\begin{rem}
	The most important qualitative difference between the affine structures near positive and negative node is the difference in their monodromy invariant subspaces at a nearby point $b$. Given a negative node, the monodromy action given by any of small loop based at $b$ leaves a plane invariant. Alternatively, given a positive node, the corresponding monodromy action leaves a line invariant.
\end{rem}

\begin{rem}
	We note that in \cite{CBM14} the authors' refer to the points we have designated as positive or negative \emph{nodes} as positive or negative \emph{vertices}, and reserve the word node for the points in the affine structure corresponding to ordinary double points of the total space. We wish to reserve the word vertex for the zero dimensional strata in the boundary (for example, the vertices of a polytope), as well as a general term for trivalent points in the $\Delta$, and accept the mild clash in terminology.	
\end{rem}

\begin{eg}
	\label{eg:boundary_singularity}
	Let $b \in \partial B$ be the image of a univalent node of $\Delta$ and let $U$ be a neighbourhood of $b$. The affine structure is a neighbourhood, containing the origin, of the product $\RR^2 \times \RR_{\geq 0}$, where the first factor is given the affine structure of a focus-focus singularity, with discriminant locus $\{0\}$ and the second factor is a ray with trivial affine structure. Following~\cite{CBM09} we also allow $\Delta$ to be perturbed to a curve given by the graph of a function $\tau \colon I \rightarrow U$ such that $\tau(0) = 0$, although we remark that we may always assume that $\Delta$ is \emph{straight} (equal to $\{0\}\times \RR_{\geq 0}$) sufficiently close to $\partial B$.
\end{eg}

From an affine manifold $B$ we can construct a topological (in fact a Lagrangian) torus fibration over $B_0 := B \setminus \Delta$ by setting
\[
\pi_0 \colon \breve{X}(B_0) := T^\star B_0 / \breve{\Lambda} \rightarrow B_0
\]
where $\breve{\Lambda}$ is the lattice of integral covectors. In fact this definition extends over the boundary of $B_0$, replacing $T_b^\star B$ with $T_b^\star (\partial B)_j$ for $j$ minimal such that $b \in (\partial B)_j$ for any $b \notin \Delta$. Note that over the boundary this map is not Lagrangian (as the fibres have the wrong dimension), but $\breve{X}(B_0)$ can still be endowed with a symplectic structure, for example using the technique of \emph{boundary reduction}, see \cite{S98,S02}. In fact it is straightforward to show that defining $\breve{X}(B_0)$ via boundary reduction the map $\pi_0 \colon \breve{X}(B_0) \rightarrow B_0$ is is isotropic on $\breve{X}(B_0)$ and Lagrangian on each stratum of $\partial B$.

\begin{rem}
	We remark that, by construction, there is a neighbourhood $U$ of every point in $\partial B \setminus \Delta$ such that $\pi_0^{-1}(U)$ is symplectomorphic to an open set in $\CC^\star{}^{n-k} \times\CC^k$ for some $k$. Moreover the map $\pi$ restricted to this open set coincides with the moment map for the usual Hamiltonian torus action on $\pi_0^{-1}(U)$. Of course, we will not assume or construct a global toric structure on $\breve{X}(B)$.
\end{rem}

In \cite[Chapter $6$]{DBranes09} Gross describes a topological compactification of the map $\pi_0$ to a map
\[
\pi \colon \breve{X}(B) \rightarrow B.
\]
We collect the local models used in this construction in Appendix~{\ref{sec:torus_fibrations}}. An important property of these torus fibrations is that they are \emph{simple} in the sense of \cite[Definition~$6.95$]{DBranes09}. This implies that they are $\QQ$-simple (\!\!\cite[Definition~$6.101$]{DBranes09}), that is, for all $p$ we have that,
\[
i_\star R^p{\pi_0}_\star\QQ = R^p\pi_\star\QQ,
\]
where $i$ is the inclusion $B_0 \hookrightarrow B$.

We present an example of an affine manifold with corners and singularities, representative of the examples we study for the remainder of this article. Later we will associate $\breve{X}(B)$ with the Fano threefold $B_3$.

\begin{eg}
	\label{eg:first_affine_manifold}
	\begin{figure}
		\includegraphics[scale=1.2]{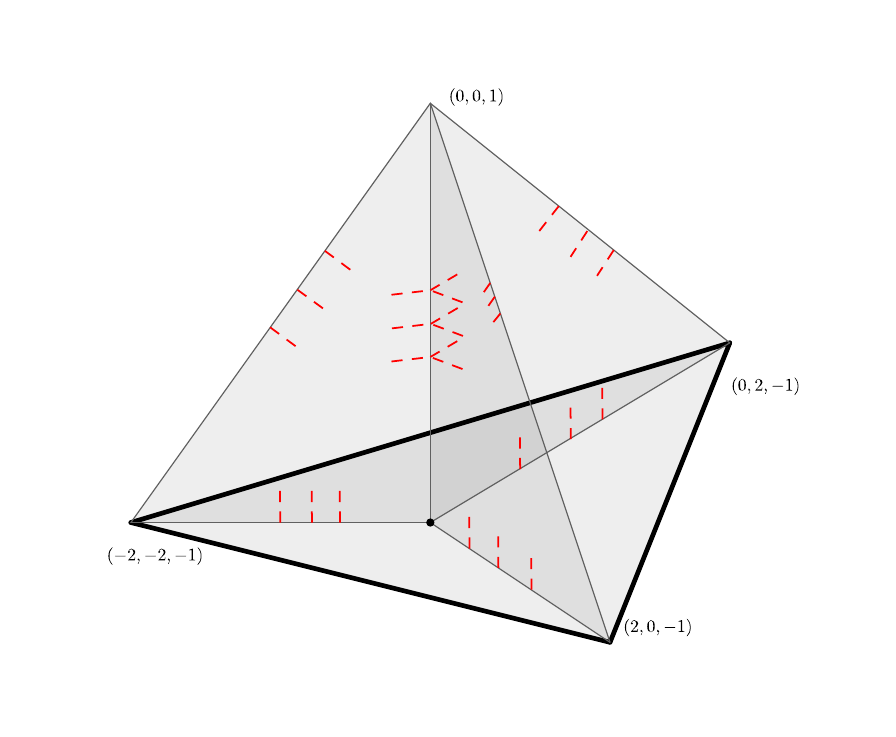}
		\caption{Diagram of an affine manifold with singularities}
		\label{fig:B3}
	\end{figure}
	There are a number of diagrams similar to Figure~\ref{fig:B3} in this article, and we use this example to explain how to interpret them. Figure~\ref{fig:B3} is a representation of an affine manifold $B$ on a polytope $P^\circ$; the convex hull of the vertices indicated in Figure~\ref{fig:B3}. The red dashed curve indicates the discriminant locus $\Delta$. For clarity we have not shown all the discriminant locus on Figure~\ref{fig:B3}, but in Figure~\ref{fig:B3_sail} we show how to complete the curve $\Delta$ over the three triangles $T_1$, $T_2$, and $T_3 \subset P^\circ$  on which it is supported.
	
	\begin{figure}
		\includegraphics[scale=1]{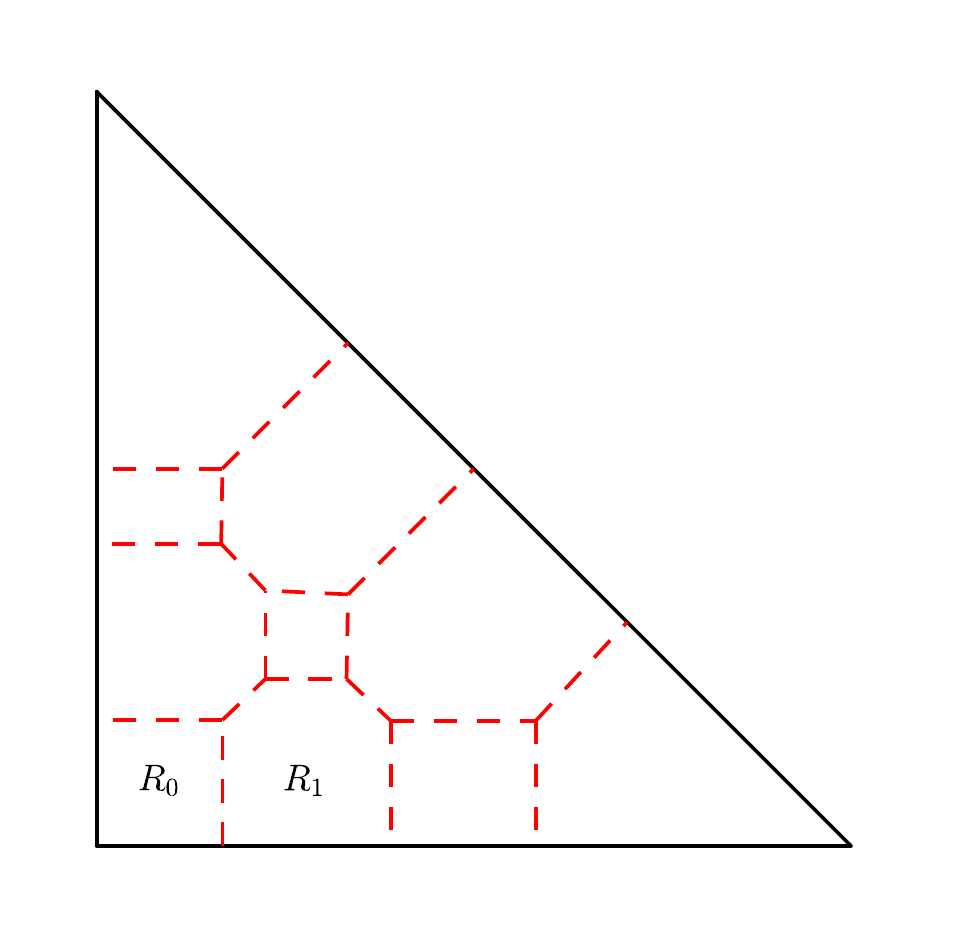}
		\caption{Completing $\Delta$ to a trivalent graph}
		\label{fig:B3_sail}
	\end{figure}
	Observe that the curve shown in Figure~\ref{fig:B3_sail} is formed by suitably triangulating $T_i$,~$i \in \{1,2,3\}$ and embedding the dual graph into $T_i$. Regarding $T_i \subset \RR^2$, each segment of $\Delta$ is associated with a direction in $\RR^2$: the unique (up to sign) primitive direction vector along the edge in the chosen triangulation of $T_i$ dual to the given segment of $\Delta$. For example, taking the segment $l$ between the regions $R_0$ and $R_1$ in the triangle $\{(0,0,1),(2,0,-1),(0,0,-1)\}$, the vector along the corresponding edge of the dual triangulation is $(1,0)$ -- as it is illustrated in Figure~\ref{fig:B3_sail} -- and $(1,0,0)$ when regarded as a vector in $\RR^3$.
	
	Having fixed a topological manifold $B = P^\circ$, and discriminant locus $\Delta$, we describe the affine structure on $B$. We do this by describing an affine atlas on $B_0$. First note that each of the three triangles supporting $\Delta$ is divided by $\Delta$ into $10$ connected components. Take one affine chart to be defined on the union of the connected components of $T_i \setminus \Delta$ which meet the point $(0,0,-1)$ -- this is labelled $R_0$ in Figure~\ref{fig:B3_sail} -- together with the complement of $\bigcup_{i \in \{1,2,3\}}{T_i} \subset P^\circ$. The affine chart on this open set $U$ is given by the identity map between $B$ and $P^\circ$. We define an open set for each connected component of each triangle. Fixing a connected component $R$ on $T_i$ for some $i \in \{1,2,3\}$, let $U_R$ be the open set 
	\[
	U_R := \left(R \times (-\epsilon,\epsilon)\right) \cap P^\circ
	\]
	for small epsilon. The intersection $U_R \cap U$ necessarily has two connected components. We determine the chart on each $U_R$ by insisting that on one component of $U \cap U_R$ this map is the identity while on the other it is a shear transformation
	\[
	x \mapsto x + \langle x,u_i \rangle \sum_l{v_l},
	\]
	where $u_i$ is a normal (co)vector to $T_i$, and the sum is taken over edges of the dual triangulation used to define $\Delta$ over any path in $T_i$ connecting $R_0$ and $R$ (now identified with vertices of a triangulation of $T_i$). Up to an overall sign, we fix signs in this sum by fixing a convention for the direction of $v_l$; for example that the direction of $v_l$ is compatible with the chosen path. We now have three binary choices: the sign of $u_i$, the sign of $\sum_l{v_l}$, and the choice of component on which the transition function is the identity. These choices result in two possible transition functions. We fix the transition function such that $u_i$ evaluates negatively on the component on which the transition function is the identity, and the vectors $v_l$ are oriented in a path from $R_0$ to $R$. Note that we have only define transition functions, rather than the charts of an atlas; in Construction~\ref{cons:slabs_to_affine_structure} we justify this, explaining that piecewise linear charts on $P^\circ \setminus \Delta$ are determined by the specified transition functions.

	We can now make various observations about the affine structure on $B$.
	\begin{enumerate}
		\item There are three positive nodes, along the edge $T_1 \cap T_2 \cap T_3$.
		\item There are $9 \times 3 = 27$ negative nodes, each contained in a unique triangle $T_i$.
		\item We have $(\partial B)_1 \cong S^1$, and is equal to the union of the three edges of $P^\circ$ which do not meet $\Delta$, while $(\partial B)_0 \cong \varnothing$.
		\item\label{it:two_discs} $\partial B$ consists of two discs meeting along their boundary. The affine structure on each disc is that induced by a Lagrangian fibration on a cubic surface.
	\end{enumerate}
	The curve $(\partial B)_1$ is marked in bold on Figure~\ref{fig:B3}. Point~\eqref{it:two_discs} is directly related to the fact that we may choose an anti-canonical divisor in $B_3$ comprised of a pair of cubic surfaces meeting in an elliptic curve. The ability to read important geometric information from these diagrams of affine manifolds is a central to their appeal. We generalise this construction in \S\ref{sec:smoothing_polytope}, and use this case as a running example.
\end{eg}

\subsection{Affine manifolds in dimension 2}

Affine structures on discs and spheres are both well-studied, and play an important role in this article. We summarize the most relevant examples in the following table.

\begin{center}
	{\renewcommand{\arraystretch}{2}
		\begin{tabular}{c | c | c}
			$B$ (topologically) & $\breve{X}(B)$ & Affine structure \\ \hline \hline
			Disc & polarised toric variety & $B$ is the image of the moment map \\ \hline
			$S^2$ & K3 surface & $24$ focus-focus singularities \\ \hline
			Disc & Del Pezzo surface $dP_d$ & $12-d$ focus-focus singularities \\ \hline
		\end{tabular}
	}
\end{center}
\bigskip

\begin{rem}
	In two dimensions it is straightforward to compactify the map $\pi_0 \colon \breve{X}(B_0) \rightarrow B_0$ to $\pi \colon \breve{X}(B) \rightarrow B$ as either a topological or symplectic manifold by adding pinched tori over the focus-focus singularities, this is described in a number of places, for example, by Gross in \cite[Chapter $6$]{DBranes09} and Auroux in \cite{Aur:1,Aur:2}, where it is shown that the local models of these compactifications form \emph{special} Lagrangian torus fibrations. The identification of $\breve{X}(B)$ with a $4$-manifold is a consequence of the classification of almost toric fibrations proved by Leung--Symington~\cite{LS10}.
\end{rem}

The case where $B$ is the moment polytope with its trivial affine structure is well known in toric and symplectic geometry. The case in which $B \cong S^2$ and $\Delta$ is a collection of $24$ focus-focus singularities is studied in detail in~\cite{KS}. The final case appears in the classification~\cite{LS10} and is also the subject of~\cite{Prince}.

The connection between the affine manifold obtained as the image of the moment map, and an affine structure on a disc with a number of focus-focus singularities was first described by Symington in \cite{S02}. In \cite{S02} the affine structure appears on the base of an \emph{almost toric fibration}, related to moment maps by the operation of \emph{nodal trade}. Interpreting a nodal trade as endowing a topological manifold with a family of affine structures produces a notion of degeneration of an affine manifold to a polygon. We make this operation precise in \S\ref{sec:product_constructions}, and refer to the operation as a \emph{polyhedral degeneration}. In the next section we define an analogous notion in three dimensions, which will be the central tool used to construct affine manifolds in this article.

%% file: smoothing_polytope.tex
% !TEX root = fano_manifold_topology.tex

%----------------------------------------------------------------------
\section{Smoothing a polytope}
\label{sec:smoothing_polytope}
%----------------------------------------------------------------------

The affine manifolds we use to construct models of Fano manifolds are closely related to \emph{Fano polytopes}. We recall that a Fano polytope $P$ is an integral polytope with primitive vertices such that the origin is contained in the interior of the polytope. The \emph{spanning fan} of $P$ is the fan defined by taking cones over the faces of $P$, and we let $X_P$ denote the corresponding toric variety. We will often use the following simple lemma concerning faces of a polytope and the polar polytope.

\begin{lem}
	There is a canonical bijection between the faces of $P$ and the faces of $P^\circ$. This bijection sends faces of dimension $k$ to faces of codimension $k+1$.
\end{lem}

Given a face $F$ of $P$ we define the corresponding face of $P^\circ$ by $F^\star$, and refer to this as the face \emph{dual} to $F$. In the three dimensional case this means that each the dual face of an edge is an edge, and the dual face to a vertex is a facet. We now introduce the combinatorial framework we will use to construct affine manifolds with singularities, which we call \emph{degeneration data} for $P$. We recall that a \emph{generalised fan} is a collection of cones satisfying all the conditions of a fan, but whose cones may not be strictly convex. Since we make heavy use of this notion, all fans in this article are assumed to be generalised fans unless otherwise stated.

We will assume throughout that $P$ is a Fano polytope contained in a vector space $N_\RR := N \otimes_\ZZ \RR$ for a lattice $N \cong \ZZ^3$. We let $M := \hom(N,\ZZ)$ denote the lattice dual to $N$ and define $M_\RR := M \otimes_\ZZ \RR$.

\begin{dfn}
	Given a polyhedral decomposition of a polytope $P^\circ$ a \emph{slab}\footnote{It would be closer to the terminology of Gross--Siebert to call these \emph{naked slabs}, since they do not yet carry sections.} $\fs$ is a pair $(c,D)$ consisting of a codimension-one cell $c$ of the decomposition and an element $D$ of the class group of the toric variety determined by the normal fan of $c$.
\end{dfn}

We will generally work with polyhedral decompositions of $P^\circ$ obtained by intersecting $P^\circ$ with a rational fan $\Sigma$ in $M_\RR$. Given such a fan $\Sigma$ we introduce a notion of labelling the one-skeleton of $P^\circ$ adapted to $\Sigma$; this will be an essential component in our notion of degeneration data.  We first recall that, given an integral polytope $Q$ in $\RR^n$ -- for any $n \in \ZZ_{> 0}$ -- such that $\sigma := \Cone{Q} \subset \RR^n$ is a strictly convex cone, the \emph{Gorenstein index} $r(Q)$ of the toric singularity associated to $\Cone{Q}$ is equal to the value $-u(Q)$, where $u$ is a primitive inner normal vector to $E$ in the saturated sublattice $L$ of $\ZZ^n$ such that $Q \subset L \otimes_\ZZ \RR$ and $\dim L = \dim Q+1$.

\begin{dfn}
	\label{dfn:edge_data}
	Given a Fano polytope $P \subset N_\RR$ and a fan $\Sigma$ contained in $M_\RR$ we define \emph{edge data} to be a choice of one-dimensional torus invariant cycle $C$ on the toric variety $X_P$ defined by the spanning fan of $P$. Moreover we demand that $C$ is supported on the collection of those torus invariant curves of $X_P$ whose images under the moment map $X_P \rightarrow P^\circ$ are contained in a two-dimensional cone of $\Sigma$.	Writing 
	\[
	C = \sum_{E \in Edges(P^\circ)}a_EC_E
	\]
	we insist that the coefficient $a_E$ is at most $\ell(E^\star)$, the lattice length of $E^\star \subset P$.
\end{dfn}

We assume throughout this article that if $E$ is an edge of $P^\circ$, $r(E^\star) = 1$ (although this need not be true for vertices of $P^\circ$). A more general definition is possible, and indeed required in \S\ref{sec:MM21} and \S\ref{sec:MM23}. However, since such definitions require separating various cases and depend on more complicated compatibility conditions, we present our construction with this additional assumption. We explain the (minor) modifications necessarily for the remaining two examples in \S\ref{sec:MM21} and \S\ref{sec:MM23}.

The bound on $a_E$ is a convexity condition, ensuring that the integral affine manifold we construct from this data has convex boundary. We describe a further condition, which characterises when this convex boundary is \emph{smooth} along edges.

\begin{dfn}
	\label{dfn:smooth_edge_data}
	We say that edge data is \emph{smooth} if, writing $C = \sum{a_E C_E}$, we have that $a_E \in \{\ell(E^\star)-1,\ell(E^\star)\}$.	
\end{dfn}

The affine manifold structure we obtain in Construction~\ref{cons:slabs_to_affine_structure} (partially) smooths the tangent cone along each edge of $P^\circ$ via the application of a piecewise linear transformation. This piecewise linear transformation acts on the quotient of the tangent cone of $P^\circ$ at $x$ a point in the interior of $E$, by the $T_xE$. This quotient is a two dimensional cone, and the piecewise linear function induced on the quotient `flattens' the boundary of the cone; as described in \cite[\S$2$]{Prince}. Smoothness of edge data corresponds to the smoothness of the cone obtained by applying such a piecewise linear transformation. We illustrate an example in Figure~\ref{fig:transverse}.

\begin{figure}
	\includegraphics{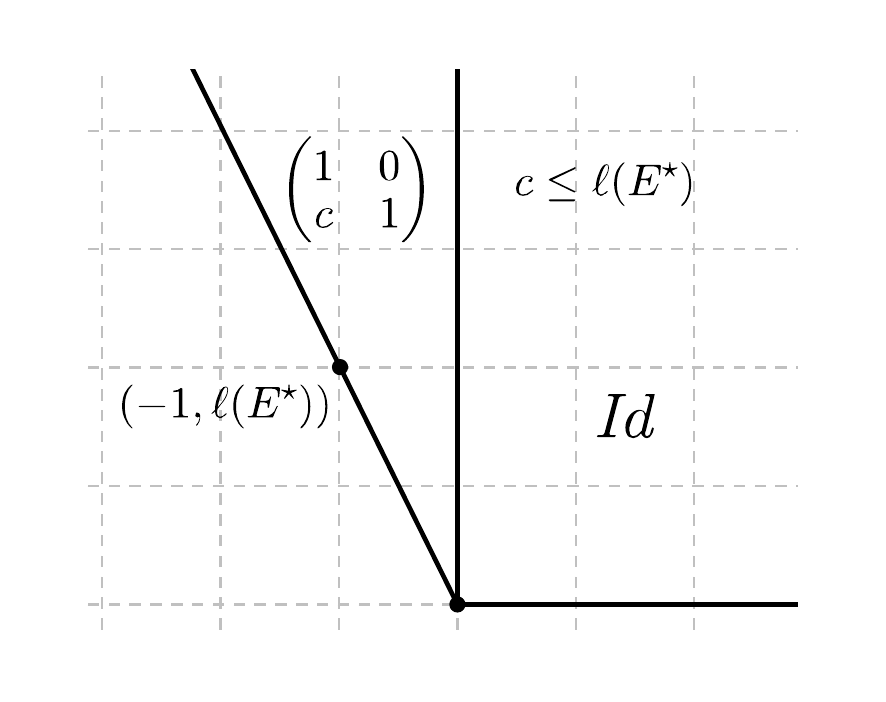}
	\caption{Flattening the boundary of the dual to a Gorenstein cone}
	\label{fig:transverse}
\end{figure}

\begin{rem}
	\label{rem:curve_to_slab}
	Let $C$ be edge data for a Fano polytope $P$ and let $\Sigma$ be a fan in $M_\QQ$. If the toric variety defined by $\Sigma$ is projective, $\Sigma$ defines a degeneration of $X_P$ in a standard way; such that the central fibre $X_{P,0}$ is a union of toric varieties whose moment polytopes form strata of the decomposition of $P^\circ$ by $\Sigma$. Clearly $C$ also defines a one-dimensional cycle of $X_{P,0}$.
\end{rem}

Following Remark~\ref{rem:curve_to_slab}, the cycle $C$ defines a collection of slabs, which we now describe. First, given a two-dimensional cone $\sigma$ of $\Sigma$, note that the toric variety defined by the normal fan of $\sigma \cap P^\circ$ contains a number of one-dimensional components of $C$. That is, $C$ defines a divisor $D_\sigma$ on $X_{\sigma \cap P^\circ}$ for each two dimensional cone $\sigma$ in $\Sigma$. Hence we may associate a slab $(c,D)$, where $c := \sigma \cap P^\circ$, and $D := D_\sigma$ for any $\sigma$.

The notion of degeneration data also depends on certain `gluing data', describing how slabs on neighbouring polygons are related. Let $\Sigma$ be a fan in the three-dimensional vector space $M_\QQ$, and let $\Sigma(k)$ denote the $k$-dimensional cones of $\Sigma$. For a cone $\tau \in \Sigma(k)$ let $X_\tau$ denote the torus invariant subvariety corresponding to $\tau$. Let $\Sigma^+(1)$ denote the set of rays contained in $\bigcup\{\rho : \rho \in \Sigma(1)\}$. If the minimal cone of $\Sigma$ has dimension different from one, we have that $\Sigma^+(1) = \Sigma(1)$; otherwise $\Sigma^+(1)$ contains a pair of elements: the pair of rays contained in the minimal cone of $\Sigma$.

\begin{dfn}
	\label{dfn:ray_data}
	Let $J := \{J(\rho) : \rho \in \Sigma^+(1)\}$ be a multiset of nef line bundle on each torus invariant hypersurface $X_\rho$. We refer to this as a choice of \emph{ray data}, and define the line bundle $L_\rho := \sum_{L \in J(\rho)}{L}$ on $X_\rho$. Moreover we say $J$ is \emph{smooth} if the image of the morphism from $X_\rho$ to a projective space defined by sections of $L$ is dominant (and hence has image $\PP^d$ for some $d \in \{0,1,2\}$) for every $\rho \in \Sigma^+(1)$ and $L \in J(\rho)$.
\end{dfn}

We can combinatorially interpret ray data $J$ using the following two facts from toric geometry, see~\cite{Fulton93}.

\begin{lem}
	Let $D$ be a nef Cartier toric divisor on a toric variety. The divisor $D$ determines and is determined by its polyhedron of sections.
\end{lem}

\begin{lem}
	Given $D_1$, $D_2$ globally generated Cartier divisors on a toric variety $Z$, the inclusion
	\[
	H^0(Z, \cO(D_1))\otimes H^0(Z, \cO(D_2)) \rightarrow H^0(Z,\cO(D_1+D_2)) 
	\]
	is an isomorphism.
\end{lem}

The data of $J$ is thus equivalent to the data of a Minkowski decomposition of the polyhedron of sections of $L_\rho$ (uniquely defined up to translation) for all $\rho \in \Sigma^+(1)$. Thus we also use $J$ to denote the corresponding set of Minkowski summands of the polyhedra of sections $P_{L_\rho}$. Note that smoothness of $J$ translates to the condition that all the Minkowski summands in $J$ are standard simplices of dimension $\leq 2$.

\begin{eg}
	\label{eg:more_B3}
	We describe edge data and ray data in the context of Example~\ref{eg:first_affine_manifold}. Let $P \subset N_\RR$ be dual to the polytope shown in Figure~\ref{fig:B3}, and let $\Sigma$ be the normal fan to the facet of $P$ dual to the vertex $(0,0,1)$ of $P^\circ$. The minimal cone of $\Sigma$ is the line $L$ generated by $(0,0,1) \in M_\RR$, and its two dimensional cones are generated by $L$ and $(1,0,0)$, $(0,1,0)$, and $(-1,-1,0)$ respectively -- see Figure~\ref{fig:B3}.
	
	We fix edge data by labelling of the edges of $P^\circ$ which are contained some two dimensional cone of $\Sigma$ with an integer. In this example we label the three edges of $P^\circ$ which contain the vertex $(0,0,1)$ with the integer $3$. The convexity condition is also easily verified: given an edge $E$ of $P^\circ$ which contains $(0,0,1)$, we have $\ell(E^\star) = a_E = 3$ for any such edge; note this edge data is also smooth.
	
	The set $\Sigma^+(1)$ contains a pair of rays $\rho_+$ and $\rho_-$, generated by $(0,0,1)$ and $(0,0,-1)$ respectively. For each element $\rho \in \Sigma^+(1)$, $X_\rho$ is isomorphic to $\PP^2$. We set $J(\rho_+) := \{\ell,\ell,\ell\}$, where $\ell$ is the line bundle $\cO_{\PP^2}(1)$ on $X_\rho \cong \PP^2$, and set $J(\rho_-) := \{0\}$. Note that this ray data is smooth: the morphism associated to the ample line bundle $\cO_{\PP^2}(1)$ is an isomorphism.
\end{eg}

In order to define an affine structure on $P^\circ$ a certain compatibility condition must be satisfied on slabs whose edges contain a common ray of $\Sigma$.

\begin{dfn}
	\label{dfn:compatible_data}
	Fix a Fano polytope $P$, a fan $\Sigma$, edge data $C$, ray data $J$ for $\Sigma$, choose a ray $\rho$ of $\Sigma$, and let $F$ denote the minimal face of $P^\circ$ intersecting $\rho$. Since $C$ defines a map from the edges of $P^\circ$ to $\ZZ_{\geq 0}$, $C$ defines a map from the torus invariant divisors of $X_\rho$ to $\ZZ_{\geq 0}$ taking the value given by $C$ along edges meeting $\rho$, and zero otherwise. Denote this map by $\ell_{C,\rho}$. We say that the ray data and edge data are \emph{compatible} if
	\[
	X_\tau \cdot L_\rho = \deg(\iota_\tau^\star L_\rho) = \frac{\ell_{C,\rho}(\tau)}{r(F^\star)}
	\]
	 for all $\rho \in \Sigma^+(1)$, $\tau \in \Sigma(2)$ such that $\rho \subset \tau$, and where $\iota_\tau$ denotes the canonical inclusion of $X_\tau \cong \PP^1$ into $X_\rho$. Recall that $L_\rho$ is defined to be the product of bundles in $J(\rho)$.
\end{dfn}

Combinatorially, the values $\ell_{C,\rho}/r(F^\star)$ are nothing but the lattice lengths of the edges of $P_{L_\rho}$, thus $C$ determines the polygons $P_{L_\rho}$, and $J$ records a Minkowksi decomposition of each of these polytopes. Note that $C$ determines a torus invariant $1$-cycle on $X_P$, but we use $C$ in Definition~\ref{dfn:compatible_data} to label divisors of $X_\rho$ -- itself a divisor of $X_\Sigma$ -- which contains the dual torus to that of $X_P$.

\begin{dfn}
	\label{dfn:degeneration_data}
	Fix a Fano polytope $P$ and a triple $(\Sigma,C,J)$ where $\Sigma$ is a fan contained in $M_\RR$, $C$ is edge data for $P$, and $J$ is ray data associated to $\Sigma$ and compatible with $C$. We say that  $(\Sigma,C,J)$ defines \emph{degeneration data for $P$} if the divisor $D$ is Cartier and nef, and $|D|$ is basepoint free for every slab $\fs = (c,D)$.
\end{dfn}

\begin{eg}
	\label{eg:B3_again}
	We now show that the choices of edge and ray data given in Example~\ref{eg:more_B3} form degeneration data. We first show that the ray data and edge data we have chosen are compatible. Indeed, observe that $L_{\rho_+}$ is $\cO(-K_{X_{\rho_+}})$, the anti-canonical bundle on $X_{\rho_+}$. The pullback of $L_{\rho_+}$ to any torus invariant divisor $X_\tau$ has degree $3$, which agrees with the labels assigned to the corresponding torus invariant curve by the given edge data.
	
	 We now check the two further conditions required to define degeneration data. Since the toric variety underlying each slab is isomorphic to $\PP^2$, positivity follows immediately from the fact the divisor classes associated to each slab have positive degree.
\end{eg}

We now make a small diversion to consider a category associated with $\Sigma$ and ray data $J$, related to the two skeleton of $\Sigma$.

\begin{dfn}
	\label{dfn:cone_category}
	Given a fan $\Sigma$ together with ray data $J$ we define a category $\fC(\Sigma,J)$ (or simply $\fC$ if $\Sigma$ and $J$ are unambiguous) as follows:
\begin{enumerate}
	\item The set of objects of $\fC$ is the disjoint union of the sets $J(\rho)$, for all $\rho \in \Sigma^+(1)$, and the set $\Sigma(2)$.
	\item The morphisms in $\fC$ are the identity morphisms together with a (single) morphism $\sigma \rightarrow P_D$ where $P_D$ is a summand of $P_{L_\rho}$ in $J(\rho)$, and $\rho \subset \sigma \in \Sigma(2)$, such that the ray $\sigma/\langle\rho\rangle$ appears in the normal fan of $P_D$.
\end{enumerate}
We call $\fC$ the \emph{diagram} of the ray data $J$ on $\Sigma$, and note that its objects are partially ordered by the dimension of the corresponding cone in $\Sigma^+(1)$ or $\Sigma(2)$.
\end{dfn}

We also make use the forgetful functor $\fC \rightarrow \Sigma[1,2]$, where $\Sigma[1,2]$ denotes the poset of rays and two dimensional cones of $\Sigma$, sending an object of $\fC$ to its underlying cone. We denote this on objects by setting $\sigma \mapsto \bar{\sigma}$.

\begin{rem}
	Note that if $\Sigma$ is the normal fan of $P$, and each $J(\rho)$ contains one element, then $\fC$ is the usual category associated to the $2$-skeleton of $\Sigma$. If $\Sigma$ is the normal fan of $P$, but $J$ is more general, the category differs from the usual $2$-skeleton by replacing each ray with a number of copies, corresponding to the summands appearing in $J(\rho)$. Clearly the category $\fC$ determines, and is determined by, a partial order of its  set of objects.
\end{rem}

We now consider the notion of \emph{smooth degeneration data}; we will construct affine structures from polytopes together with a choice of smooth degeneration data. 

\begin{dfn}
	\label{dfn:smooth_data}
	We say that degeneration data $(\Sigma,C,J)$ is \emph{smooth} if the ray and edge data are smooth, and -- fixing a vertex $v$ of $P^\circ$ and letting $d$ denote the dimension of the minimal cone $\tau$ of $\Sigma$ containing $v$ (if $d=1$ we take the unique $\tau \in \Sigma^+(1)$ containing $v$) -- the following conditions hold.
	\begin{enumerate}
		\item If $d=3$ the cone over $v^\star \subset P$ is a smooth cone in $N_\RR$.
		\item If $d=2$ the cone over $v^\star$ is Gorenstein, and $v^\star$ is the Cayley sum of two line segments $F_1$ and $F_2$ (possibly of length zero) contained in the annihilator of $\tau$, such that $|a(F_1^\star) - a(F_2^\star)| \leq 1$, where $a(F_i^\star) = 0$ if $\dim F_i^\star \neq 1$. Moreover, if $\dim F_1^\star = \dim F_2^\star = 1$ we insist that $a(F_i^\star) = \ell(F_i^\star)$ for some $i \in \{1,2\}$.
		\item\label{it:smooth_cone}  If $d=1$, $v^\star$ satisfies
		\[
		v^\star = r(v^\star) P_{L_\tau} + S_v,
		\]
		where $S_v$ is a standard (affine) simplex, and we recall that $\tau \in \Sigma^+(1)$. Moreover we insist that either that $\dim(S_v) = 0$, or the cone over $v^\star$ is Gorenstein. 
	\end{enumerate}
\end{dfn}

The conditions given in Definition~\ref{dfn:smooth_data} ensure that the affine structure we construct below from smooth degeneration data has smooth boundary. In particular, given a vertex $v$ in a ray of $\Sigma$, the tangent cone at $v$ in the affine manifold $B$ will be isomorphic to the dual of the cone over $S_v$. This cone is smooth if and only if  the cone over $S_v$ is a smooth cone.

\begin{dfn}
	\label{dfn:polyhedral_degeneration}
	Given a Fano polytope $P$ a \emph{polyhedral degeneration} is a vector space determined by smooth degeneration data $(\Sigma, C, J)$. We define a functor
	\[
	S \colon \fC \rightarrow \Vect
	\]
	as follows. Given a cone $\sigma \in \Sigma(2)$, $S(\sigma) := \Gamma(Z_c,\cO(D))$, the space of sections of $D$ where $c := \sigma \cap P^\circ$, $Z_c$ is the toric variety defined by the normal fan of $c$, and $D$ is the divisor on the slab with polygon $c$. Given an element $L \in J(\rho)$, $\rho \in \Sigma(1)$, we set $S(L) := \Gamma(Z_{\rho},\cO(1))$ where $Z_\rho \cong \PP^1$ is the toric variety defined by the normal fan of $\rho \cap P^\circ$. The image of the morphisms is defined by restriction, noting that since the ray data is smooth, each polyhedron of sections $P_E$ for $E \in J(\rho)$ is a standard simplex and the divisor class $E$ pulls back to $\cO(1)$ on $Z_\rho$.
	
	The polyhedral degeneration associated to degeneration data $(\Sigma,C,J)$ is the inverse limit of $S$ over the diagram of $J$, or the space of `global sections' of $S$. 
\end{dfn}

In other words, the space defined in Definition~\ref{dfn:polyhedral_degeneration} is the space of sections of the linear systems on slabs $(c,D) \in \fS$ such that the sections chosen agree along the torus invariant curves of the slabs in a way encoded in $J$.

\begin{rem}
	The space appearing in Definition~\ref{dfn:polyhedral_degeneration} is the base of a (topological) degeneration. While we do not describe it in detail here, it is possible to define a family of affine manifolds over a polyhedral degeneration such that the special fibre is $P^\circ$ and the general fibre is a simple affine manifold with singularities and boundary. Making this family algebraic in dimension $2$ using the Gross--Siebert algorithm was pursued in \cite{Prince}.
\end{rem}

\begin{cons}
	\label{cons:slabs_to_affine_structure}
	Given smooth degeneration data $(\Sigma,C,J)$ on a Fano polytope $P$ we will determine the affine structure of a general fibre of the family over the corresponding polyhedral degeneration.
	\begin{enumerate}
		\item Decompose the polytope $P^\circ$ (with its usual affine structure) into polyhedra formed by intersecting $P^\circ$ with the cones of the fan $\Sigma$.
		\item Define a collection of slabs $\fS$ in bijection with $\Sigma(2)$, where $\fs = (c,D) \in \fS$ consists of a polygon $c := \sigma \cap P^\circ$ for a cone $\sigma \in \Sigma(2)$, and $D$ is defined using the torus invariant cycle $C$ via Remark~\ref{rem:curve_to_slab}. 
		\item Given a slab $\fs = (c,D) \in \fS$, let $\Gamma_\fs$ be the trivalent curve formed by the one-skeleton of the dual graph of a maximal triangulation of $P_D$, the polyhedron of sections of $D$. Since $D$ is a nef divisor on $X_c$ there is a canonical map $\varphi$ from the edges of $P_D$ to faces of $c$.
		\item For each $\rho \in \Sigma^+(1)$ such that $E_\rho := \rho \cap P^\circ$ is an edge of $c$, choose a set of distinct points $\{p_{\rho,L} | L \in J(\rho), L \neq 0\}$ contained in the interior of $E_\rho$.
		\item For each $\fs = (c,D) \in \fS$, embed $\Gamma_\fs$ into the polygon $c$ such that if $E$ is an edge of $P_D$ and $\varphi(E) = E_\rho$ for some $\rho \in \Sigma(1)$, the end points of $\Gamma_\fs$ dual to line segments contained in $E$ map bijectively to points 
		\[
		\{p_{\rho',L} : \rho' \in \Sigma^+(1), \rho' \subset \rho, L \in J(\rho')\}.
		\]
		Note that this construction makes use of the assumed compatibility between ray and edge data.
	\end{enumerate}

	We now make $P^\circ$ into an affine manifold $B$, with boundary equal to $\partial P^\circ$ (regarding $P^\circ$ as a topological manifold in the obvious manner), and singular locus $\Delta$ defined by the union of the curves $\Gamma_\fs$ for $\fs \in \fS$. Note that given a slab $\fs = (c,D)$, the curve $\Gamma_\fs$ partitions $c \subset P^\circ$ into a number connected components in bijection with the torus invariant sections $t$ of $\cO(D)$.
	
	We cover $P^\circ \setminus \Delta$ by a number of charts. First define a chart $U_\sigma := P^\circ \cap \Int(\sigma)$ for each three-dimensional cone $\sigma$ of $\Sigma$. The affine structure on $U_\sigma$ is induced by the inclusion $P^\circ \subset M_\RR$. Note that $U_\sigma$ may inherit boundary strata from $P^\circ$, so this chart may already have corners. Let $I$ be the set of connected components of
	\[
	P^\circ \setminus \left(\bigcup_{\sigma \in \Sigma(3)}U_\sigma \cup \bigcup_{\fs \in \fS}\Gamma_\fs \right).
	\]
	We define a chart $U_R$ for each element $R$ of $I$ by choosing a connected neighbourhood of $R$ in $P^\circ \setminus \bigcup_{\fs \in \fS}\Gamma_\fs$ which retracts onto $R$. Recall from Example~\ref{eg:first_affine_manifold} that regions $R \in I$ such that $R \subset c$ for some $(c,D) \in \fS$ can be identified with integral points in the polygons $P_D$. Note that the open set $U$ which appears in Example~\ref{eg:first_affine_manifold} is -- in our current notation -- $U_R \cup \bigcup U_\sigma$, where $R$ contains the point $(0,0,-1)$; denote this open set $\tilde{U}_R$. The polygons $P_D$ for each $(c,D) \in \fS$ contain the origin; and hence a distinguished integral point. In fact there is a distinguished component $R_0 \in I$, identified with the origin in every polygon $P_D$; note that if $\Sigma$ contains a zero dimensional cone $R_0$ contains the origin in $M_\RR$.
	
	We identify $\tilde{U}_{R_0}$ with the open set of $P^\circ$ via the identity map. To define charts for each $U_R$ we describe piecewise linear maps $\phi_R\colon M_\RR \to M_\RR$ and define a chart on $U_R$ on $B$ by composing the canonical inclusion $U_R \hookrightarrow M_\RR$ with $\phi_R(U_R)$. These piecewise linear maps are integral affine functions on the intersections of these open sets and hence determine the transition functions between charts. Since $\phi_R$ is determined by its restriction to $\tilde{U}_{R_0} \cap U_R$, specifying the transition functions determines the integral affine manifold $B$.
	
	First note that we can assume that -- if $R_0 \neq R$ -- the intersection $\tilde{U}_{R_0} \cap U_R$ has two connected components. We determine the transition function on $U_R \cap \tilde{U}_{R_0}$ on each $U_R$ by insisting that on one component of $\tilde{U}_{R_0} \cap U_R$ this map is the identity while on the other it is a shear transformation
	\[
	x \mapsto x + \langle x,u \rangle \sum_l{v_l},
	\]
	where $u$ is a normal (co)vector to $c$, and the sum is over edges of the dual triangulation used to define $\Delta$ over any path in $P_D$ connecting integral points identified with $R_0$ and $R$. As in Example~\ref{eg:first_affine_manifold} we make choices of signs and normal vectors such that $u$ evaluates negatively on the component of $\tilde{U}_{R_0} \cap U_R$ on which the transition function is the identity, and the vectors $v_l$ are oriented in a path from $R_0$ to $R$.
\end{cons}

\begin{rem}
	\label{rem:assumptions}
	The above construction relies on the compatibility of the ray and edge data (allowing us to match end points of the trivalent graphs $\Gamma_\fs$). We also require positivity of the degeneration data to ensure we can match edges of $c$ with (certain) faces of $P_D$.
\end{rem}

The main result of this section is that this construction produces an affine manifold with singularities and corners.

\begin{thm}
	\label{thm:affine_from_poly}
	Given smooth degeneration data $(\Sigma,C,J)$, Construction~\ref{cons:slabs_to_affine_structure} defines an affine structure on $B_0 := P^\circ\setminus\Delta$ and endows $B := P^\circ$ with the structure of an affine manifold with singularities and corners if $(\partial B)_1 \cap \Delta = \varnothing$. %Moreover, if $v^\star = r(v^\star)P_{L_\rho}$ for any vertex $v$ of $P^\circ$ contained in a ray $\rho \in \Sigma(1)$ (see Definition~\ref{dfn:smooth_data}), then $B$ has smooth boundary.
\end{thm}
\begin{proof}
	The affine structure over the interior of $P^\circ \setminus \Delta$ is standard; neighbourhoods of segments of $\Delta$ are isomorphic to the product of a focus-focus singularity with an interval, while neighbourhoods of trivalent points are \emph{positive} and \emph{negative} nodes. Note that smoothness of ray data ensures that the trivalent points contained in rays are positive nodes, while the remaining trivalent points are negative nodes as the triangulation of $P_D$ for each $(c,D) \in \fS$ is unimodular.
	
	Let $x$ be a point in the (relative) interior of a two-dimensional face of $P^\circ$. Since $x$ is contained in some $U_\sigma$, a neighbourhood of $x$ is locally isomorphic to $\RR^2 \times \RR_{\geq 0}$. Next consider a point $x$ in the interior of an edge $E$ of $P^\circ$. If $x$ is not contained in a two dimensional cone $\tau$ of $\Sigma$, $x \in U_\sigma$ for some $\sigma \in \Sigma(3)$. Hence assume that $x \in \tau$ for some $\tau \in \Sigma(2)$ -- and therefore $x \in R$ for some $R \in I$. Let $V$ be a neighbourhood of $x$ and note that $V\setminus \tau \subset U_{\sigma_1}\cup U_{\sigma_2}$, where $\sigma_1$ and $\sigma_2$ are the three dimensional cones of $\Sigma$ which contain $\tau$. Taking the quotient $M_E := M_\RR/T_xE$, the faces meeting $x$ are shown in Figure~\ref{fig:flat_edge}. The tangent cone at $x$ defines a transverse singularity (the toric variety associated to the dual of the tangent cone at $x$). The transition function $x \mapsto x + \langle x,u \rangle \sum_l{v_l}$ induces a piecewise linear map
	\[
	\bar{x} \mapsto \bar{x} + \langle \bar{x},\bar{u} \rangle \bar{v}
	\]
	on $M_E$, where $\bar{x}$ is the image of $x$ under the projection $p \colon M_\RR \to M_E$, $\bar{u}$ is the unique element in $M_E^\star$ such that $p^\star\bar{u} = u$, and $\bar{v}$ is the projection of $\sum_l{v_l}$ to $M_E$. The integral vector $\bar{v}$ lies in the tangent space to the image of $\tau$ in $M_E$ (see Figure~\ref{fig:flat_edge}), and has index $a_E$. An example of this transition function in co-ordinates is illustrated in Figure~\ref{fig:flat_edge}. Hence convexity of the boundary of $B$ imposes a bound on $a_E$. Applying \cite[Lemma~$2.2$]{Prince} this bound is equal to the \emph{singularity content} $\left\lfloor \frac{\ell(E^\star)}{r(E^\star)} \right\rfloor =  \ell(E^\star)$ defined in \cite{AK14}.

	\begin{figure}
		\includegraphics{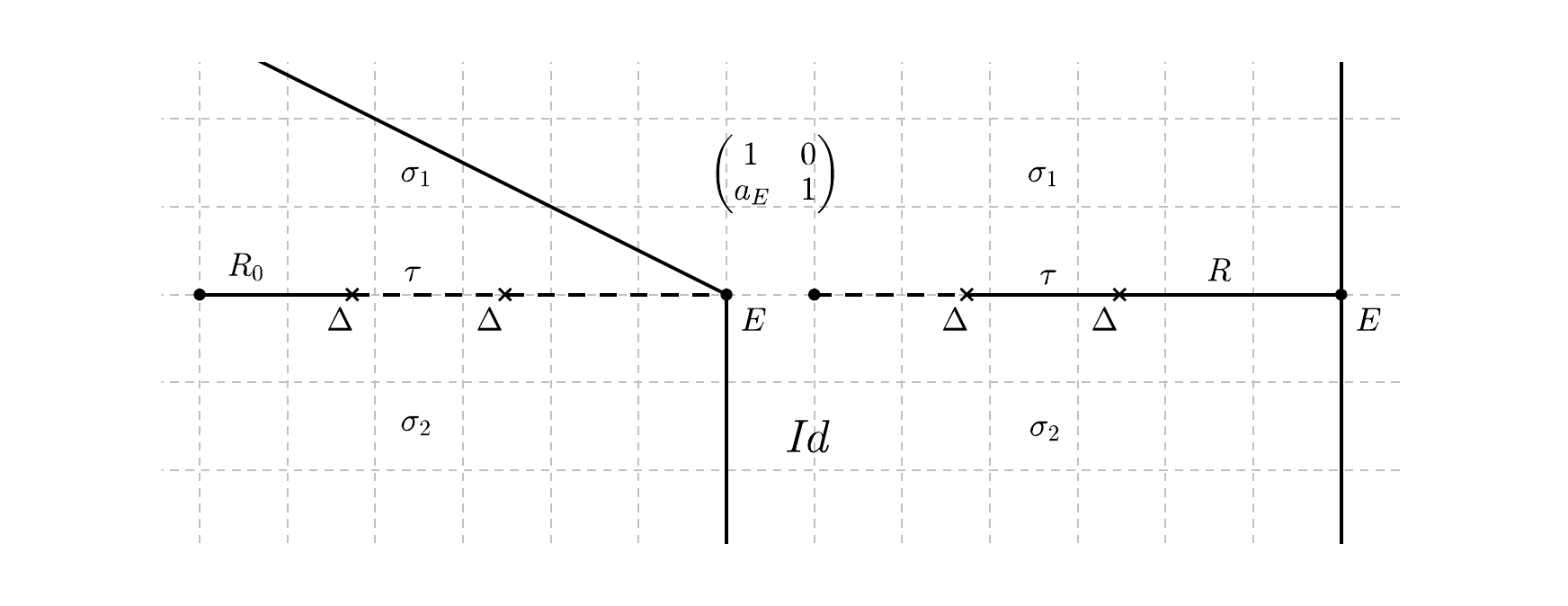}
		\caption{Cross section of $B$.}
		\label{fig:flat_edge}
	\end{figure}

	Let $v$ be a vertex of $P^\circ$ and let $d$ be the dimensional of the minimal cone $\tau$ in $\Sigma$ containing $v$. If $d=3$ the tangent cone at $x$ is necessarily a smooth cone (by Definition~\ref{dfn:smooth_data}). If $d=2$ the conditions given in Definition~\ref{dfn:smooth_data} mean that, up to a change of co-ordinates we can put $v^\star$ into the standard form
	\[
	v^\star = \conv{\begin{pmatrix}
		0\\0\\1
		\end{pmatrix},
		\begin{pmatrix}
		0\\\ell(F_1^\star)\\1
		\end{pmatrix},
		\begin{pmatrix}
		1\\\ell(F_2^\star)\\0
		\end{pmatrix},
		\begin{pmatrix}
		1\\0\\0
		\end{pmatrix}
	}
	\]
	The dual cone is generated by the rays illustrated in Figure~\ref{fig:vertex_in_cone}. The region $R \in I$ corresponds to an integral point in the polygon $P_D$, where $(c,D) \in \fS$ is such that $c = \tau \cap P^\circ$. Identifying the plane spanned by $(1,0,0)$ and $(0,0,1)$ with the plane containing $P_D$, this integral point has co-ordinates $(a(F_1^\star), a(F_1^\star))$. Hence the transition function from $\tilde{U}_{R_0}$ to $U_R$ sends the point $(-\ell(F_1^\star),1, -\ell(F_1^\star))$ to $(a(F_1^\star)-\ell(F_1^\star),1, a(F_2^\star)-\ell(F_1^\star))$. By our assumptions on $a(F_1^\star)$ and $a(F_2^\star)$ this cone is smooth.

	\begin{figure}
		\includegraphics[scale=1]{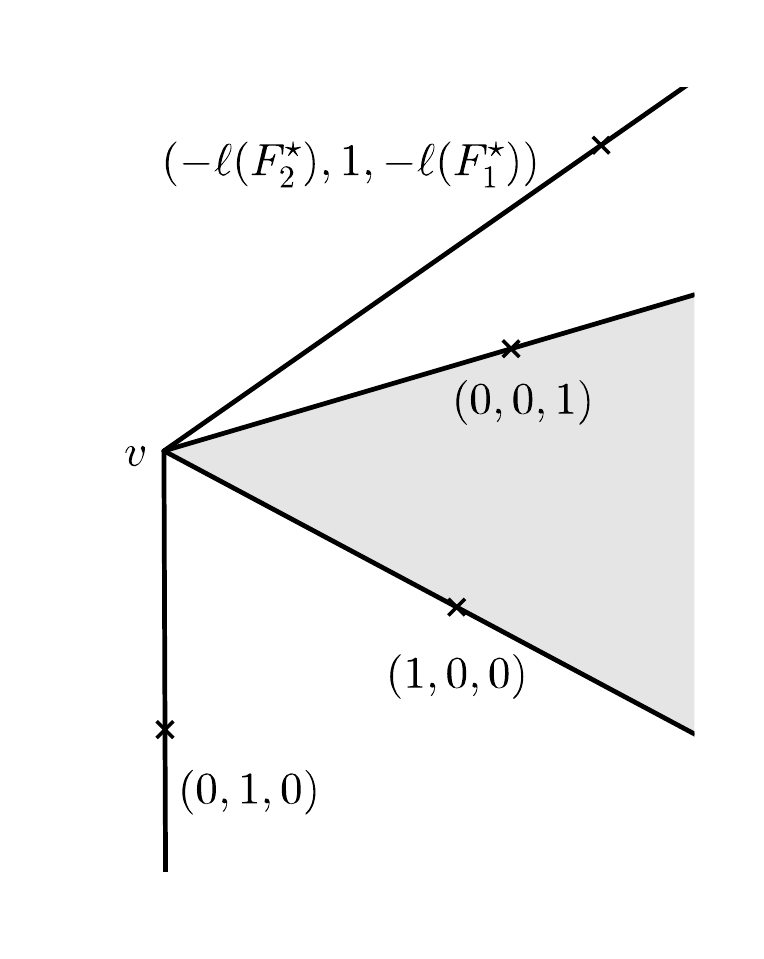}
		\caption{The tangent cone at a vertex of $P^\circ$ contained in $\tau$.}
		\label{fig:vertex_in_cone}
	\end{figure}

	\begin{figure}
	\centering
	\subfigure{%
		\includegraphics[scale=0.8]{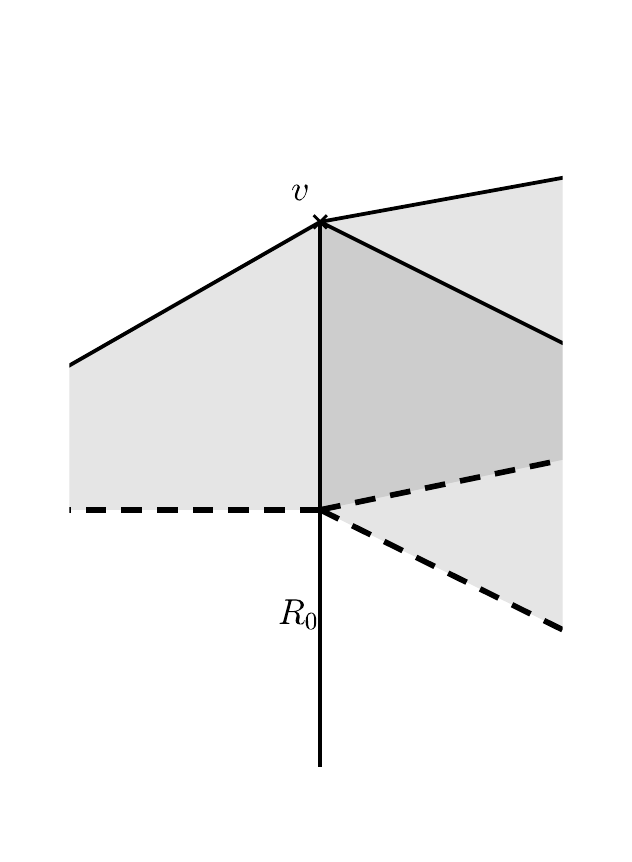}}
	\quad
	\subfigure{%
		\includegraphics[scale=0.8]{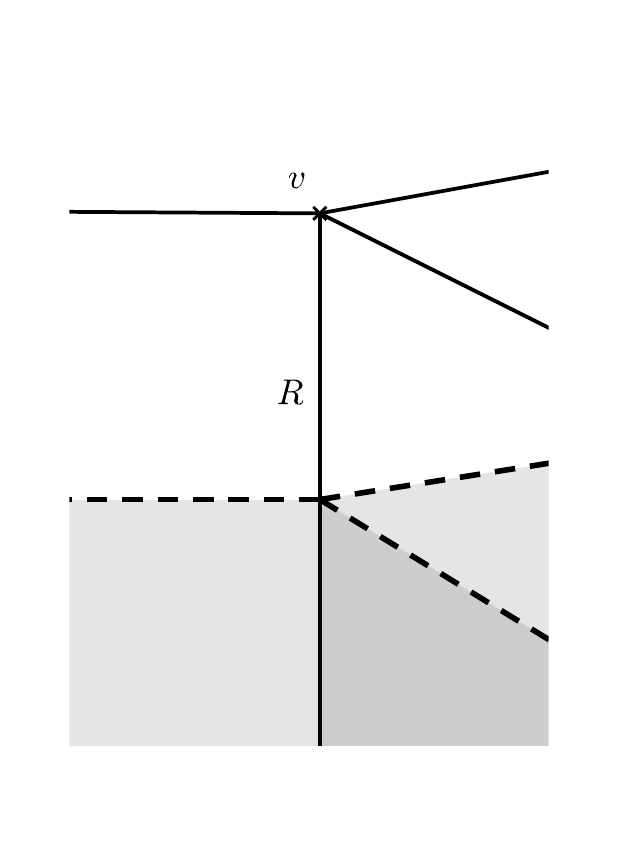}}
	
	\caption{Smoothing the boundary of $B$ near a vertex.}
	\label{fig:flat_vertex}
\end{figure}

	Now assume that $v \in \rho$ for some $\rho \in \Sigma^+(1)$. Smoothness of the tangent cone at $v$ follows from the fact that $v^\star = r(v^\star)P_{L_\rho} + S_v$, for a standard simplex $S_v$ such that $\dim S_v =0$ unless the cone over $v^\star$ is Gorenstein. The tangent cone at $v \in B$ is the image of the tangent cone at $v \in P^\circ$ under the piecewise linear map determined by the transition function from $U_{R_0}$ to $U_R$, where $R \in I$ is such that $v \in R$. An example of such a piecewise linear transformation is illustrated in Figure~\ref{fig:flat_vertex}. In general, the transition function from $\tilde{U}_{R_0}$ to $U_R$ is defined by the following formula,
	\begin{equation}
	\label{eq:ray_function}
	\phi_R \colon x \to x + \left(\min_{u \in \V{P_{L_\rho}}}{\langle x,u\rangle}\right)v.
	\end{equation}

	This follows from the fact that the affine structure around each trivalent point in $\rho$ is a \emph{positive node} and -- replacing $P_{L_\rho}$ with a standard simplex in \eqref{eq:ray_function} -- the map given in \eqref{eq:ray_function} describes the transition function from one affine chart near a positive node to the other (see Figure~\ref{fig:flat_vertex}). The transition function from $\tilde{U}_{R_0}$ to $U_R$ is the composition of such piecewise linear maps, which is easily verified to be given by \eqref{eq:ray_function}. Letting $C_v$ denote the tangent cone of $P^\circ$ at $v$, the cone $\phi_R(C_v)$ is dual to the cone over $S_v$, with the same Gorenstein index as the cone over $v^\star$.
\end{proof}

As indicated in the statement of Theorem~\ref{thm:affine_from_poly}, we need to check case by case that $(\partial B)_1 \cap \Delta = \varnothing$. This is indeed the case in every example described in Appendix~\ref{sec:tables}. It is obviously sufficient to show -- and usually the case -- that $(\partial B)_1 = \varnothing$.

\begin{eg}
	We describe the application of Construction~\ref{cons:slabs_to_affine_structure} in the prototypical example of $\PP^3$. First fix the polytope $P$ in $N_\QQ \cong \QQ^3$ defined to be the convex hull of the standard basis in $\ZZ^3$ together with the point $(-1,-1,-1)$. Fix degeneration data by choosing $\Sigma$ to be the normal fan of $P$, $C$ to be the sum of the one-dimensional toric strata of $\PP^3$ (the curve defined by labelling each edge of $P^\circ$ with $1$), and $J$ to be the trivial Minkowski decomposition of each facet of $P$. The polytope $P^\circ$ together with the labelling defining $C$ is shown in Figure~\ref{fig:P3}. 
	
	\begin{figure}[htbp]
		\includegraphics[scale = 0.8]{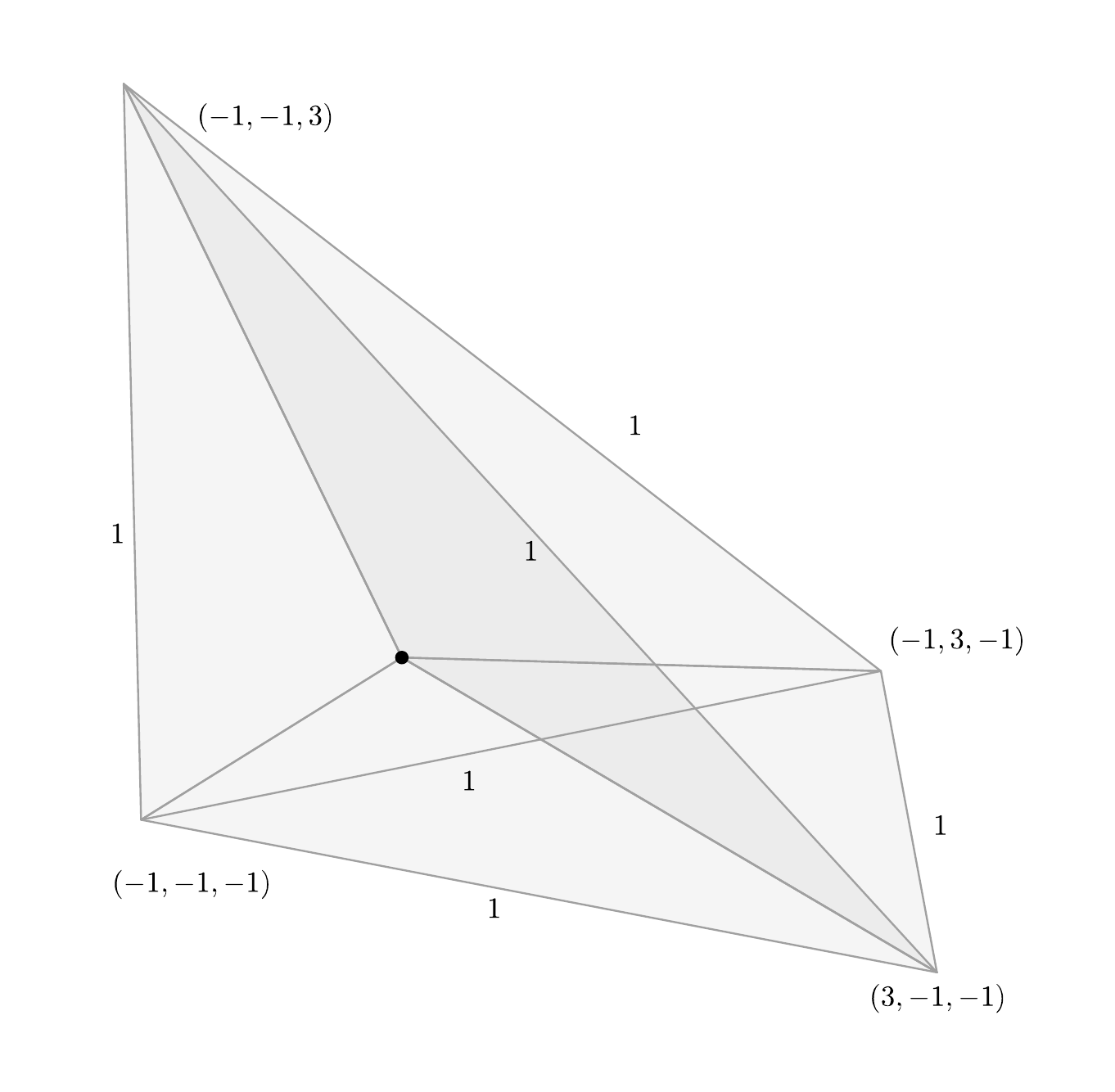}
		\caption{Degeneration data for $\PP^3$}
		\label{fig:P3}
	\end{figure}

	For each slab $\fs = (c,D)$, we have that $Z_c \cong \PP(1,1,4)$. Giving the surface $Z_c$ co-ordinates $x_0,x_1,y$ of weights $1$, $1$, and $4$ respectively, $D$ is the divisor $\{y=0\}$, determined by a section of $\cO(4)$. The curve $\Gamma_\fs$ is shown in Figure~\ref{fig:tropicalcurve}; note that this curve is the dual graph of the unique maximal triangulation of the polyhedron of sections of $\cO(4)$ on $\PP(1,1,4)$. We fix embeddings of each of these curves such that they meet in trivalent points (which will become the positive nodes); an example of such an embedded curve is shown in Figure~\ref{fig:embeddedcurve}.

	\begin{figure}
		\includegraphics{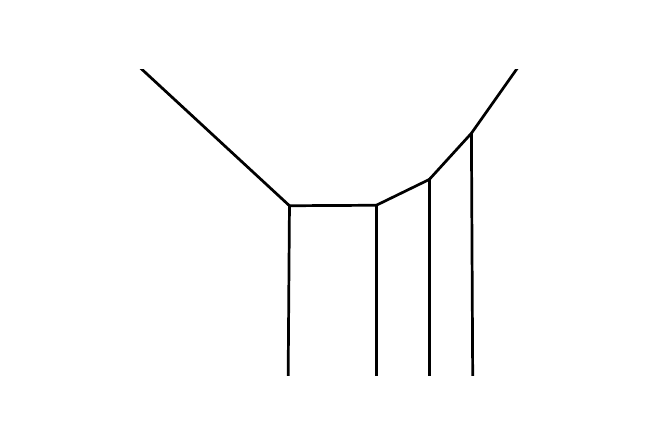}
		\caption{The curve $\Gamma_\fs$ for the slab $\fs = (\PP(1,1,4),\cO(4))$.}
		\label{fig:tropicalcurve}
	\end{figure}
	
	\begin{figure}
		\includegraphics[scale=1.8]{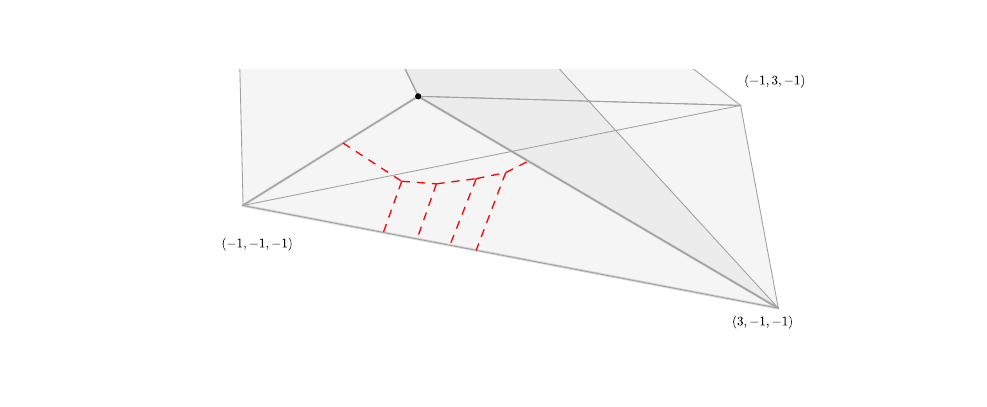}
		\caption{Embedding a curve $\Gamma_\fs$ into $P^\circ$.}
		\label{fig:embeddedcurve}
	\end{figure}
\end{eg}

\begin{rem}
	In images such as Figure~\ref{fig:embeddedcurve} we display the polytope $P^\circ$, and the singular locus $\Delta$. However the image cannot be an accurate description of the whole affine structure, but only of a single chart. We always display the chart which contains the origin in $P^\circ$, and hence it often appears that $(\partial B)_1 \cap \Delta \neq \varnothing$, while in fact there is no edge present in the affine structure of $B$.
\end{rem}

%% file: degeneration_data.tex
% !TEX root = fano_manifold_topology.tex

%----------------------------------------------------------------------
\section{Constructing degeneration data}
\label{sec:degeneration_methods}
%----------------------------------------------------------------------

In this section we present three constructions of degeneration data on a Fano polytope $P$. Given any Fano threefold $X$ there is a polytope such that one of these three methods give a topological model of $X$; these polytopes and constructions are enumerated in the tables in Appendix~\ref{sec:tables}.

\subsection{Smooth Minkowski Decompositions}
\label{sec:smooth_decompositions}

The first of the three constructions takes advantage of a special form of the facets of certain reflexive polytopes $P$ to construct an affine structure on $P^\circ$ with smooth boundary. This construction will be used to construct affine manifolds corresponding to $89$ of the $105$ families of Fano threefolds. Fix a lattice $N \cong \ZZ^3$ and let $P$ be a reflexive polytope $P \subset N_\RR$.

\begin{dfn}
 A \emph{smooth Minkowski decomposition} of $F$ is a Minkowski decomposition of $F$
	\[
	F = \sum_{i \in I}{F_i}
	\]
	such that all the polygons $F_i$ are standard simplices.
\end{dfn}

Given a reflexive polytope $P$, the input to our construction of degeneration data on $P^\circ$ is a set $\bM$ of smooth Minkowski decompositions of the facets of $P$. Recall that given an edge $E$ of any integral polytope we denote its lattice length by $\ell(E)$.

\begin{rem}
	Note that for most reflexive polytopes $P$ no choices of such Minkowski decompositions $\bM$ exist (for example if $P$ has a Minkowski irreducible facet which is not a standard simplex, no smooth Minkowski decomposition exists), and if one does exist it may not be unique.
\end{rem}

\begin{cons}
	\label{cons:smooth_decompositions}
	Given a reflexive polytope $P$ and a set of smooth Minkowski decompositions $\bM$ of its facets we fix degeneration data $(\Sigma,C,J)$ as follows.
	\begin{enumerate}
		\item Let $\Sigma$ be the normal fan of $P$.
		\item Let $C$ be defined by the map $E \mapsto \ell(E^\star)$ for each edge $E$ of $P^\circ$.
		\item Let $J$ be the collections of nef divisors determined by the Minkowski decompositions $\bM$.
	\end{enumerate}
\end{cons}

Given a set $\bM$ of smooth Minkowski decompositions of the facets of $P$, we let $B_{P,\bM}$ denote the affine manifold obtained by applying Construction~\ref{cons:slabs_to_affine_structure} to the choice of degeneration data given in Construction~\ref{cons:smooth_decompositions}. In \S\ref{sec:euler_characteristic}, \S\ref{sec:degree}, and \S\ref{sec:betti_numbers} we will compute the numerical invariants of the total space of the torus fibration with base $B_{P,\bM}$.

\begin{pro}
	\label{pro:smooth_boundary}
	Let $B_{P,\bM}$ be an affine manifold obtained via the application of Construction~\ref{cons:smooth_decompositions} to the pair $(P,\bM)$, then $\partial B_{P,\bM}$ is an integral affine sphere with $24$ focus-focus singularities.
\end{pro}
\begin{proof}
	We first verify that, given an edge $E$ of $P^\circ$ and a point $x \in E$, the integral affine structure $B_{P,\bM}$ identifies a neighbourhood of $x$ with a neighbourhood of the origin in $\RR_{\geq 0}\times\RR^2$. The transverse singularity associated to $E$ is Gorenstein as $P$ is reflexive; and hence the affine structure around $x$ is smooth if and only if $a_E = \ell(E^\star)$ (the `width' of the singularity).
	
	Fix a vertex $v \in P^\circ$, we verify that $\partial B_{P,\bM}$ is smooth in a neighbourhood of $v$. This follows from the assumption that $v^\star = P_{L_\rho}$, where $\rho$ is the ray of $\Sigma$ containing $v$. In particular affine structure along the boundary of $B$ near $v$ is equal to the image of a piecewise linear map applied to the tangent cone of $P^\circ$ at $v$. Following the description of this map in the Proof of Theorem~\ref{thm:affine_from_poly}, this piecewise linear map identifies a neighbourhood of $v$ with $\RR^2 \times \RR_{\geq 0}$ if and only if $S_v$ is a point, that is, if $v^\star = P_{L_\rho}$.
	
	Finally we observe that, by construction, the singular points $x$ in $\partial B_{P,\bM}$ are necessarily focus-focus singularities if the edge $E$ of $P^\circ$ containing $x$ is not contained in $(\partial B)_1$; however we have already observed that $(\partial B)_1 = \varnothing$.
\end{proof}

\begin{rem}
	We remark that the form of the polytope we use can be regarded as a special case of the \emph{Minkowski ansatz} considered in~\cite{CCGGK}. In particular there is always a candidate mirror family, closely related to the \emph{Minkowski Laurent polynomials} defined in~\cite{CCGGK}. In fact the additional restriction of Minkowski factors to \emph{standard} simplices is closely related to the condition of \emph{simplicity} or \emph{local rigidity} appearing in~\cite{Gross--Siebert}. In future work we hope to extend the topological local models we consider to analyse all cases considered in~\cite{CCGGK} and obtained by the Minkowski ansatz.  
\end{rem}

\subsection{Complete intersection constructions}
\label{sec:complete_intersections}

The second technique we use to specify degeneration data uses a connection between polyhedral decompositions of $P^\circ$ and complete intersection models of $X_P$. Indeed, given a description of $X_P$ as a complete intersection in a toric variety $Y$ via linear systems $D_1,\ldots,D_k$ which form a \emph{(Fano) nef partition}  (see \cite{CKP15,Przyjalkowski}, generalising the original notion for Calabi--Yau varieties due to Batyrev--Borisov \cite{BB96}) we can form a \emph{toric degeneration} by deforming the defining binomial equations of $X_P$. In addition, a nef partition defines a monomial degeneration, degenerating $X_P$ into a union of toric strata of $Y$. This further degeneration defines a polyhedral decomposition of $P^\circ$ via a fan $\Sigma$, the fan defined by a product of projective spaces. We do not explore this construction in more detail here, but refer the reader to~\cite{CKP17}, where it is carried out in detail.

The main tool used in~\cite{CKP17} to construct models of Fano varieties is that of a \emph{scaffolding}, the definition of which we briefly recall. Fix a Fano polytope $P \subset N_\RR$ and a smooth toric variety $Z$ -- the \emph{shape} -- whose dense torus has character lattice $\bar{N}$; and a complement $N_U$ to $\bar{N}$ in $N$.

\begin{dfn}
A \emph{scaffolding} $S$ of $P$ is a collection of pairs $(D,\chi)$ where $D$ is a torus invariant divisor of $Z$ and $\chi \in N_U$ is a lattice vector. We insist that the line bundle $\cO_Z(D) \in \Pic(Z)$ is nef for each divisor $D$ and that 
	\[
	P = \convhull\limits_{(D,\chi) \in S}(P_D+\chi).
	\]
We refer to the divisors $D$ as \emph{struts}.
\end{dfn}

%Note that this definition itself is a generalisation of the notion of nef partition (for the variety $Z$).

It is proved in~\cite{CKP17} that a scaffolding defines a torus invariant embedding of $X_P$ into a toric variety defined by a fan in $\Div_\TT(Z)_\RR$. An important case of this construction occurs when $Z = \prod_{i \in I}{\PP^{a_i}}$. In this case the embedding of $X_P$ (and its corresponding monomial degeneration) compactifies the family
\[
\left\{\prod_{i \in I_1}{x_i} = t, \ldots, \prod_{i \in I_k}{x_i} = t\right\} 
\]
where the sets $I_j$ are pairwise disjoint sets, for $j \in [k]$, and co-ordinates $x_i$ on a complex torus. The compactification lifts these binomials to binomials of Cox co-ordinates 
\[
\left\{ \prod_{i \in I_1}{X_i} = tZ^{m_1}, \ldots , \prod_{i \in I_k}{X_i} = tZ^{m_k} \right\}
\]
where $m_1,\ldots,m_k$ are lattice vectors. The reducible variety defined by setting $t=0$ contains a number of divisors obtained from the degeneration of the complex torus. These divisors are fixed by setting any two variables $X_i$ in the same index set $I_j$ to zero. In the three dimensional case, these divisors are toric surfaces, and the monomial $Z^{m_j}$ defines a torus invariant curve on this toric variety. We let the set $\fS$ of slabs be the set of such toric surfaces equipped the divisor classes determined by each monomial $Z^{m_j}$, moreover we denote by $C_\fs$ the torus invariant curve determined by $Z^{m_j}$. 

\begin{eg}
	\label{eg:cubic}
	 A simple example will help to clarify some of the preceding combinatorics. Let $N \cong \ZZ_3$ and fix the splitting $N = \bar{N} \oplus N_U$, where $N_U \cong \ZZ$ is generated by $e_3$, and $\bar{N} \cong \ZZ^2$ is generated by $e_1$ and $e_2$. Let $P$ be the polytope described in Example~\ref{eg:first_affine_manifold}, i.e. we let
	 \begin{align*}
	 P := \conv{(0,0,1),(-1,-1,-1),(-1,2,-1),(2,-1,-1)},\\
	 P^\circ := \conv{(0,0,1),(-2,-2,-1),(2,0,-1),(0,2,-1),(0,0,1)}.
	 \end{align*}
	 We write $P$ as the convex hull of the triangle $\conv{(-1,-1,-1),(-1,2,-1),(2,-1,-1)}$, and the single point $\{(0,0,1)\}$. We regard each of these polytopes as translates of polyhedra of sections associated to nef divisors (struts of a scaffolding) on $\PP^2$. The fan $\Sigma$ used to define an affine structure on $P^\circ$ is the product of the fan determined by $Z$ -- that is, the fan for $\PP^2$ -- together with $(N_U)\otimes_\ZZ \RR$. The intersection of $P^\circ$ with cones in $\Sigma$ is illustrated in Figure~\ref{fig:B3}.
	 
	 Geometrically $X_P$ is the hypersurface in $\PP^4$ defined by the binomial equation $X_1X_2X_3 = X_0^3$. This degenerates to the union of toric varieties defined by $\{X_1X_2X_3=0\}$. Each slab is a divisor of the form $X_i=X_j=0$ for $i,j \in \{1,2,3\}$ and $i \neq j$. Each of these divisors is isomorphic to $\PP^2$ and we assign to each the one-dimensional torus invariant cycle $3\cdot\{X_0=0\}$.
\end{eg}

\begin{rem}
	In fact, in the degeneration data specified for the quartic in Example~\ref{eg:cubic} coincides with the degeneration data associated to $P^\circ$ using the construction given in \S\ref{sec:smooth_decompositions}. This coincidence is not typical, and is related to the fact that the ambient space in this example has Picard rank one.
\end{rem}

\begin{cons}
	\label{cons:ci_constructions}
	Given a Fano polytope $P$ and a scaffolding of $P$ whose shape variety has fan $\Sigma$, we define degeneration data $(\Sigma,C,J)$ as follows.
	\begin{enumerate}
		\item Let $\Sigma$ be the fan fixed by the choice of shape variety $Z$.
		\item Let $C$ be a torus invariant curve given by the sum of the curves $C_\fs$, regarded as cycles in $X_P$.
		\item Let $J$ be the unique choice of smooth Minkowski decompositions determined by $C$.
	\end{enumerate}
	Note the choice of $J$ is unique since $\Sigma$ is the fan determined by a product of projective spaces.
\end{cons}

\begin{rem}
	This technique applies to a large number of reflexive (and Fano) polytopes to generate -- at least topologically -- many families of Fano threefolds. Indeed in~\cite{CCGK} the authors give complete intersection constructions of $93$ of the $105$ families of Fano threefolds. However since our analysis of these invariants is usually more involved we will only rely on these constructions where necessary, and where the computations are simple. We will recover the invariants of $11$ families of Fano threefolds using this construction. These are studied in \S\ref{sec:method_2_calculations}, and listed in Appendix~\ref{sec:tables}.
\end{rem}

\begin{rem}
	A more serious problem is that it is difficult, given a Fano polytope $P$, to see whether $P$ admits degeneration data of the form required for this construction to work. Indeed each of the examples we consider in \S\ref{sec:method_2_calculations} have been reverse-engineered from known complete intersection models of Fano threefolds.
\end{rem}

\subsection{Product constructions}
\label{sec:product_constructions}

The third technique we use to construct polyhedral degenerations exploits on the fact that there is a well known version of polyhedral degeneration in dimension two, the so-called \emph{nodal trades} used by Symington~\cite{S02}. There are $10$ families of Fano threefolds obtained by taking the product of a del~Pezzo surface and the projective line. Of these families $5$ are smooth toric varieties and of the remaining $5$, three have very ample anti-canonical bundle.

We briefly recall the notion of nodal trade and define the notion of degeneration data in dimension $2$.

\begin{dfn}
	Let $N$ be a two-dimensional lattice and let $P$ be a Fano polygon in $N_\QQ$. Degeneration data for $P$ is a pair $(\Sigma,f)$ where $\Sigma$ is a fan in the dual lattice $M$ and $f$ is a zero-dimensional torus invariant cycle on $X_P$. This data is required to satisfy analogues of the convexity and positivity conditions in dimension~$3$:
	\begin{enumerate}
%		\item \emph{(Positivity)} The cycle $f$ is effective.
		\item \emph{(Convexity and Positivity)} Writing
		\[
		f = \sum_{v \in \V{P^\circ}}{a_v v}
		\]
		we have that $0 \leq a_v \leq \left\lfloor \frac{\ell(v^\star)}{r(v^\star)} \right\rfloor$, where $r(v^\star)$ is the Gorenstein index of the cone over the edge $v^\star$.
		\item \emph{(Compatibility)} If $v \in \V{P^\circ}$ is not contained in a ray of $\Sigma$, $a_v = 0$.
	\end{enumerate}
	We say that degeneration data is \emph{smooth} if 
	\[
	\left\lfloor \frac{\ell(v^\star)}{r(v^\star)} \right\rfloor - a_v= 
	\begin{cases}
	 0 & \text{if $r(v^\star) > 1$} \\ 
	 0 \text{ or } 1 & \text{if $r(v^\star) = 1$} \\
	\end{cases}
	\]
\end{dfn}

For example, the trivial affine structure on a smooth polygon (a polygon such that the toric variety defined by its normal fan is smooth) defines smooth degeneration data using any fan $\Sigma$ and $f=0$.

Given degeneration data $(\Sigma,f)$ for a Fano polygon $P$ we form an affine manifold by a simplified version of Construction~\ref{cons:slabs_to_affine_structure}. A general fibre $B$ of a polyhedral degeneration in dimension two is determined by fixing $a_v$ points in the interior of the segment $[0,v]$, and putting the unique affine structure on $B$ such that each point is a focus-focus singularity, such that the direction $[0,v]$ is monodromy invariant.

\begin{cons}
	\label{cons:products}
	Let $P$ be a Fano polytope such that $P^\circ = P'{}^\circ \times [-1,1]$ and $P'$ is a Fano polygon. For each $v \in \V{P'{}^\circ}$ let $E_v$ be the edge of $P^\circ$ with vertices $(v,1)$ and $(v,-1)$. Let $B$ be the affine manifold determined by the degeneration data $(\Sigma,C,J)$ where:
	\begin{enumerate}
		\item $\Sigma$ is the product of the normal fan of $P'$ with the subspace spanned by $(\mathbf{0},1)$. Recall that -- as in \S\ref{sec:smoothing_polytope} -- we do not assume cones in a fan are strictly convex.
		\item $C$ is the cycle determined by the function $E_v \mapsto \ell(v^\star)$.
		\item $J$ is trivial, since there are only two rays of $\Sigma$ and neither ray meets a vertex of $P^\circ$.
	\end{enumerate}
	Applying Construction~\ref{cons:slabs_to_affine_structure} determines an affine structure on the topological manifold $P^\circ$.
\end{cons}

\begin{rem}
The affine manifold $B$ obtained by Construction~\ref{cons:products} is isomorphic to the product $B'\times [-1,1]$ where $B'$ is the affine manifold obtained from the degeneration data $(\Sigma,f)$ where $\Sigma$ is the normal fan of $P$ and $f$ sends $v \mapsto \ell(v^\star)$ for each vertex $v$ of $P^\circ$.
\end{rem}

\begin{eg}
	\label{dP4timesP1}
	Consider the affine manifold $B'$ formed by exchanging corners for focus-focus singularities in the square with vertices 
	\[
	\left\{(1,0),  (0,1),  (-1,0), (0,-1)\right\}.
	\]
	The torus fibration (with singularities) $\breve{X}(B')$ is homeomorphic to a del Pezzo surface of degree $4$ (in fact it can be made symplectomorphic to it). Taking a product with a closed line segment we obtain the affine manifold $B$, shown in Figure~\ref{fig:dP4xP1}. The resulting manifold $\breve{X}(B)$ is homeomorphic to $\breve{X}(B') \times S^2$, that is, to the product of a del~Pezzo of degree $4$ and the projective line.
		
	\begin{figure}
		\caption{An affine manifold model of $dP_4 \times \PP^1$}
		\label{fig:dP4xP1}
		\includegraphics[scale = 0.8]{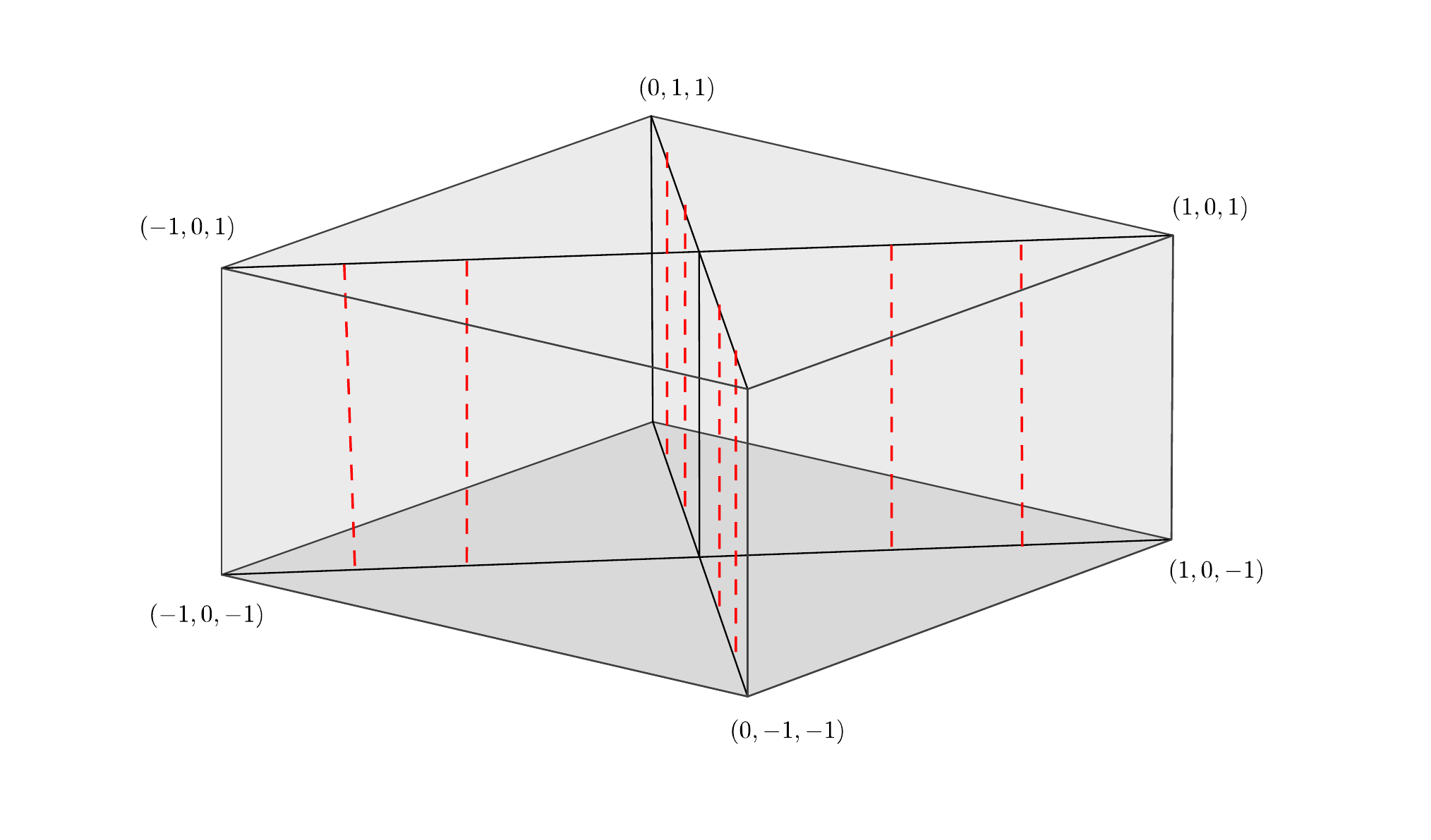}
	\end{figure}
\end{eg}

%% file: euler_and_degree.tex
% !TEX root = fano_manifold_topology.tex

%----------------------------------------------------------------------
\section{Euler Number}
\label{sec:euler_characteristic}
%----------------------------------------------------------------------

Given an affine manifold $B$ obtained from degeneration data $(\Sigma,C,J)$ by Construction~\ref{cons:slabs_to_affine_structure} we calculate the Euler number of the manifold $\breve{X}(B)$ in this section from the Euler numbers of the fibres of the map
\[
\pi \colon \breve{X}(B) \rightarrow B.
\]

As well as giving a general description of $e(\breve{X}(B))$ in terms of $B$ we give formulae in terms of the degeneration data obtained via each of the three constructions given in \S\ref{sec:degeneration_methods}.

\begin{rem}
	In the two dimensional case the Euler number of a smoothing coincides with the notion of \emph{singularity content}~\cite{AK14,A+} and this definition provides one possible generalisation of this notion to dimension three.
\end{rem}

\begin{pro}
	\label{pro:euler_number}
	Given degeneration data $(\Sigma,C,J)$ for a reflexive polytope $P$, let $B$ denote the affine manifold obtained via Construction~\ref{cons:slabs_to_affine_structure}, the Euler number of $\breve{X}(B)$ is computed by the following formula:
	\[
		e(\breve{X}(B)) = 2\sum_{\fs \in \fS}{(1-i_\fs)} - 2|J| + \V{B}.
	\]
	where, given a slab $\fs = (c,D) \in \fS$, $b_\fs$ and $i_\fs$ are the number of boundary and interior points of the polyhedron of sections $P_D$ respectively, and $|J|$ is the sum of the number of factors in $J(\rho)$ over all $\rho \in \Sigma(1)$.
\end{pro}

\begin{proof}[Proof of Proposition~\ref{pro:euler_number}]
	
We first compute the Euler number of the fibres of the torus fibration
\[
\pi \colon \breve{X}(B) \rightarrow B.
\]
\noindent Studying the descriptions of the fibres of $\pi$ given in Appendix~\ref{sec:torus_fibrations}, the only fibres of $\pi$ with non-zero Euler number are: the positive and negative nodes of $B$, points of intersections between $\Delta$ and $\partial B$, and vertices of $B$. We summarise these Euler numbers, see Lemmas~\ref{lem:euler_number_positive} and \ref{lem:euler_number_negative}, in the following table.

\begin{center}
	\begin{tabular}{c | c}
		Special fibre & Euler number  \\ \hline
		Positive node & $1$ \\
		Negative node & $-1$ \\
	    Point in $\Delta \cap \partial B$ & $1$ \\
		Vertex of $B$ & $1$ \\
	\end{tabular}
\end{center}
\medskip

Hence we have that
\[
e(\breve{X}(B)) = p - n + |\Delta \cap \partial B| + \V{B}.
\]
Recalling that $b_\fs$ denotes the number of boundary points of $P_D$, we have that
\[
\sum_{\fs \in \fS}b_\fs - |\Delta \cap \partial B| = 3p + 2d,
\]
where $d$ is the number of smooth points of $\Delta$ contained in a ray of $\Sigma$. However, by definition, $|J| = p+d$, and hence 
\[
\sum_{\fs \in \fS}b_\fs - 2|J| = p + |\Delta \cap \partial B|,
\]
and
\[
e(\breve{X}(B)) = \sum_{\fs \in \fS}{b_\fs} - 2|J| - n + \V{B}.
\]
The number of negative vertices in $P_D$ is equal to the number of standard simplices of a triangulation of $P_D$, which is equal to twice the area $A_\fs$ of $P_D$. By Pick's theorem, $2A_\fs = 2i_\fs + b_\fs - 2$, and hence $n = \sum_{\fs \in \fS}{(2i_\fs + b_\fs - 2)}$, and
\[
e(\breve{X}(B)) = 2\sum_{\fs \in \fS}{(1-i_\fs)} - 2|J| + \V{B}.
\]
\end{proof}

The formula given in Proposition~\ref{pro:euler_number} can be simplified considerably for the degeneration data used in the constructions given in \S\ref{sec:smooth_decompositions} and \S\ref{sec:product_constructions}.

\begin{cor}
	\label{cor:euler_number_smooth_mink}
	Given a reflexive polytope $P$ and a set of smooth Minkowski decompositions $\bM$ of its facets, let $B$ denote the affine manifold obtained in \S\ref{sec:smooth_decompositions} (deforming the standard affine structure on $P^\circ$) we have that,
	\[
	e(\breve{X}(B)) = 24 + T - \sum_{E \in \Edges{P^\circ}}{\ell(E)\ell(E^\star)^2},
	\]
	where $T$ is the total number of (standard) triangles appearing in $J$.
\end{cor}
\begin{proof}
Note that when $\partial B$ is itself smooth it is well known that $|\Delta \cap \partial B| = 24$, the number of focus-focus singularities on a flat $S^2$. Of course in this situation $B$ has no vertices. Moreover the total number of positive nodes is precisely $T$.
	
Finally the number of negative nodes is the sum of the area of $P_{D}$ (recall that this is equal to the number of triangles in a maximal triangulation of the polyhedron of sections $P_{D}$), however $c$ is a moment polytope of the weighted projective plane $\PP(1,1,\ell(E))$ and $D$ is the line bundle $\cO(\ell(E^\star)\ell(E))$. Thus the area of $P_{D}$ is precisely $\ell(E)\ell(E^\star)^2$.
\end{proof}

By way of a small digression, we remark that we can adapt this construction of an affine manifold to recover a famous combinatorial identity.

\begin{pro}[\!\!\cite{BCF05}]
	For a reflexive polytope $P$, we have that
	\[
	\sum_{E \in \Edges{P^\circ}}{\ell(E)\ell(E^\star)} = 24
	\]
\end{pro}
\begin{proof}
	We fix degeneration data as follows:
	\begin{enumerate}
		\item Let $\Sigma$ to be the normal fan of $P$;
		\item Let $C$ be the cycle defined by $E \rightarrow \ell(E^\star)$ for $E \in \Edges{P^\circ}$, and;
		\item Let $J$ be the divisor $X_\rho$, without further decomposition
	\end{enumerate}
	Although we did not describe Construction~\ref{cons:slabs_to_affine_structure} in precisely this context we may use a slight generalisation of it to define an affine structure on $P^\circ$ such that the boundary is a smooth $S^2$. Counting the number of focus-focus singularities appearing on the boundary we observe that for each edge $E$ of $P^\circ$ the corresponding slab $\fs = (c,D)$ where $D$ a section of $\cO(\ell(E)\ell(E^\star))$ on $\PP(1,1,\ell(E))$ and the number of singular points lying on the edge $E$ of $P^\circ$ contained in $c$ is the size of the zero set of a general section of this line bundle restricted to $\PP^1$. Summing over all edges of $P$ (equivalently over all slabs) we obtain the left hand side of the expected identity. However the total number of singular points is equal to $24$, the topological Euler number of a K$3$ surface.
\end{proof}

\begin{cor}
	\label{cor:euler_product}
	Given an affine manifold $B$ obtained by the construction given in \S\ref{sec:product_constructions} we have that
	\[
	e(\breve{X}(B)) = 2e(dP_d) = 2(12-d).
	\]
	where $d$ is the degree of the polygon $P'$ such that $P$ is the product of $P'$ and a line segment and $dP_d$ is any del~Pezzo surface of degree $d$.
\end{cor}
\begin{proof}
	Counting the number of special fibres, all such fibres appear over points contained in one of two faces of $\partial B$ and the affine manifold obtained by restricting to each of these faces is well known to have $12-d$ singularities.
\end{proof}

\begin{rem}
	The number of positive and negative nodes of affine manifolds $B$ describing models of each of the $105$ families of smooth Fano threefolds are displayed in the tables in Appendix~\ref{sec:tables}.
\end{rem}
%----------------------------------------------------------------------
\section{Anti-canonical degree}
\label{sec:degree}
%----------------------------------------------------------------------

In this section we compute (a topological analogue of) the anti-canonical degree of the compactified torus fibrations considered in \S\ref{sec:smoothing_polytope}. Despite the fact the families we consider are not algebraic, defining the degree of $\breve{X}(B)$ to be $[\pi^{-1}(\partial B)]^3$, the cube under the cup product of the class Poincar\'{e} dual to the pre-image of $\partial B$, we check that this coincides with the expected degree. This number is also the degree of the toric variety $X_P$, which agrees with our expectation that $X_P$ is a toric degeneration of a Fano manifold homeomorphic to $\breve{X}(B)$.

\begin{pro}
	\label{pro:degree}
	If $P$ is a reflexive polytope and $B$ an affine manifold obtained from degeneration data for $P$ the intersection number $[\pi^{-1}(\partial B)]^3$ is equal to $2|P^\circ \cap M|-6$.
\end{pro}
\begin{proof}
We make use of the contraction map $\breve{X}(B) \to X_0$ described in Appendix~\ref{sec:contraction} (and writing $X_0 := \breve{X}_0(B)$), based on the treatment given in \cite{G01}. The topological space $X_0$ is the reducible union of the toric varieties defined by the normal fans to $P^\circ \cap \sigma$ for each three dimensional cone $\sigma$ in $\Sigma$. Note that the (projective) toric variety $X$ associated to the normal fan $P^\circ$ is polarised by the line bundle $-K_X$ (here we assume that $X$ is Gorenstein, and $-K_X$ is very ample). Standard toric techniques -- see, for example, the description of the Mumford degeneration given in \cite{Gross--Siebert} -- provide an embedded degeneration of $X$ to $X_0$.

Let $Z$ be the union of the torus invariant boundary divisors of $X_0$ which are also torus invariant boundary divisors of $X$. That is, boundary divisors whose moment map image lies in $\partial P^\circ$, and let $Z_i$ be the irreducible toric components of $Z$. Observe that each toric stratum $V$ of $Z$ is contained in a unique toric stratum $\bar{V}$ of $X_0$ of equal codimension whose restriction to $Z$ is $V$. Choose an identification of a disc bundle $DV \subset N_{\bar{V}}V$ with a tubular neighbourhood $U_V$ of $V$ such that, if $V_1 \subset V_2$ are toric strata of $Z$, we have that $U_{V_2} \cap \bar{V}_1 = U_{V_1}$. Note the union of the tubular neighbourhoods $U_{Z_i}$ is a tubular neighbourhood $U_Z$ of $Z$ in $X_0$, and is identified with a disc bundle $DZ$ on $Z$.

We require that identifications of disc bundles $DZ_i$ the neighbourhoods $U_{Z_i}$ satisfy an additional compatibility condition with the surface $\tilde{\Delta}$ (the lift of $\Delta$ to $X_0$ described in Appendix~\ref{sec:contraction}). Noting that the surface $\tilde{\Delta}$ intersects $Z$ in a finite set contained in the union of torus invariant curves of $Z$, we insist that the fibre over $x \in Z \cap \tilde{\Delta}$ is a disc in $\tilde{\Delta}$.

Noting that $Z$ is a hyperplane section of $X_0$, we consider the intersection of $Z$ with a pair of sections $s_1$, $s_2$ of $DZ$; identified with the tubular neighbourhood $U_Z$. Choosing such sections generically, we can assume that $Z\cap s_1\cap s_2$ is contained in the smooth locus of $Z$ and that  $\deg(X_0) = \deg(X) = |Z\cap s_1\cap s_2|$.

We have that $\xi^{-1}(Z) = \pi^{-1}(\partial B)$; moreover, by the compatibility of $U_Z$ with the singular locus $\tilde{\Delta}$, we have that the pre-images $\xi^{-1}(s_1)$ and $\xi^{-1}(s_2)$ are homotopic to $\pi^{-1}(\partial B)$. Indeed, we consider the behaviour of $\xi$ on points $x \in Z$, letting $D_x \subset U_Z$ denote the image of the fibre of $DZ \to Z$ over $x$.
\begin{enumerate}
	\item If $x$ is contained in the smooth locus of $Z$, $\xi$ is a homeomorphism in a neighbourhood of $x$.
	\item If $x$ is a general point in the singular locus of $Z$, $\xi^{-1}(D_x) \cong D_x \times S^1$, and the map $\xi$ restricts to the composition of projection to the first factor and a homeomorphism.
	\item If $x \in \tilde{\Delta}$, $\xi^{-1}(D_x) \cong D_x$ and $\xi$ restricts to a homeomorphism.
	\item If $x$ is a torus invariant point in $Z$, $D_x$ is a disc in a torus invariant curve of $X_0$, and $\xi^{-1}(D_x) \cong D_x \times T^2$, and the map $\xi$ restricts to the composition of projection to the first factor and a homeomorphism.
	\item If the image of $x$ in $\partial B$ lies in $(\partial B)_1$ or $(\partial B)_0$ then, for either $i \in \{1,2\}$, $\xi^{-1}(s_i(x))$ is an $S^1$ or $T^2$ vanishing cycle respectively which disappears as $s_i(x)$ approaches $x$.
\end{enumerate}

Observing that we may assume (generically) that the intersection $\pi^{-1}(\partial B) \cap \xi^{-1}(s_1) \cap \xi^{-1}(s_2)$ occurs transversely in the smooth locus of $\pi^{-1}(\partial B)$, we obtain 
\[
[\pi^{-1}(\partial B)]^3 = |\pi^{-1}(\partial B) \cap \xi^{-1}(s_1) \cap \xi^{-1}(s_2)| =\deg(X_0) = \deg(X).
\]
It is a standard result of toric geometry that the anti-canonical degree of $X_P$ is the volume of $P^\circ$ (normalised so that the standard simplex has volume $1$); see, for example, \cite[\S$13.4$]{Cox--Little--Schenck}. Since $P$ is reflexive, this is equal to the normalised area $A$ of $\partial P^\circ$. Writing $P^\circ$ as a union of facets $F_i$ for $i \in I$, and -- using Pick's theorem -- we obtain that
\[
A = \sum_{i \in I}{2\Area{F_i}} = \sum_{i \in I}{2\iota_i + b_i - 2},
\]
where $\iota_i$ and $b_i$ denote the number of interior and boundary points of $F_i$ respectively. Writing $b_i = b'_i + v_i$, where $v_i$ is the number of vertices of $F_i$, we obtain that $A - 2|\partial P^\circ| = -2V + \sum_{i \in I}{v_i-2}$ -- where $V$ is the number of vertices of $P^\circ$. Letting $F$ denote the number of facets of $P^\circ$, $A - 2|\partial P^\circ| = -2V-2F -\sum_{i \in I}{v_i}$. However $\sum_{i \in I}{v_i} = 2E$ -- where $E$ is the number of edges of $P^\circ$ -- and hence $A - 2|\partial P^\circ| = -4$, as required.
\end{proof}

\begin{rem}
	\label{rem:non_reflexive}
	When $P$ is not reflexive (as occurs when we consider the seven examples of Fano varieties $X$ for which $-K_X$ is ample but not very ample) Proposition~\ref{pro:degree} is not true as stated. One way of generalising Proposition~\ref{pro:degree} to the non-reflexive case would be to consider dilates of $P^\circ$, and hence polarising the toric variety $X_P$ with a multiple of the anti-canonical class. We can then mimic the proof of Proposition~\ref{pro:degree} using the boundary of the dilated polytope.
\end{rem}

%% file: betti_numbers.tex
% !TEX root = fano_manifold_topology.tex

%----------------------------------------------------------------------
\section{Computing Betti numbers}
\label{sec:betti_numbers}
%----------------------------------------------------------------------

In this section we compute the Betti numbers of $\breve{X}(B)$ for $B$ obtained by the construction given in \S\ref{sec:smooth_decompositions}. This will provide the calculation of the Betti numbers for $89$ of the $105$ cases we consider, and similar techniques will be applied to the other $16$ examples. In particular our Betti number calculations are derived from studying the Leray spectral sequence associated to the contraction map $\xi$ described in Appendix~\ref{sec:contraction}.

Note that, by construction, $b_0(\breve{X}(B)) = 1$ as $B$ is connected. In fact, following the arguments used in~\cite{G01}, simply connectedness of $B$ also ensures that the first Betti number of $\breve{X}(B)$ vanishes.

\begin{lem}
	Given an affine manifold $B$ defined by Construction~\ref{cons:slabs_to_affine_structure} applied to an affine manifold using degeneration of the form defined in \S\ref{sec:smooth_decompositions} the manifold $\breve{X}(B)$ is simply connected.
\end{lem}
\begin{proof}
	This follows immediately from the proof of Theorem~2.12 of \cite{G01}. We briefly sketch this here. Denoting the universal cover by $\mu \colon \tilde{X} \rightarrow \breve{X}(B)$ we define the space $\tilde{B} := \tilde{X}/\sim$ : the quotient of $\tilde{X}$ equating points which lie over the same point of $B$ under the map $\pi \circ \mu$, and lie in the same connected component of the fibre of this map. The map $\pi \circ \mu$ then factors through a map to $\tilde{B}$, and let $\gamma$ denote the induced map $\tilde{B} \rightarrow B$. In \cite{G01} Gross then proves that $\gamma$ is a covering map. To see this we remark that for any point $b \in B$ the fibre of a neighbourhood $U$ of $b$ decomposes into connected components $V_1,\ldots, V_n$, each of which is quotient of the universal cover $\tilde{V}$ of $\pi^{-1}(U)$. Case by case analysis then confirms that $\tilde{V} \rightarrow B$ has connected fibres for any choice of $b \in B$, and hence, from the definition of $\tilde{B}$, $\gamma^{-1}(U)$ is a disjoint union of copies of $U$. Since $\gamma$ is a covering of (simply connected) $B$ it must be an isomorphism.

	%Note that not all the cases for $\tilde{V}$ which appear here also appear in \cite{G01}, given that in our context $B$ may have boundary. In particular we rely on the integral structure defining \emph{smooth cones} along $(\partial B)_0$ and $(\partial B)_1$, otherwise the groups which may appear as the fundamental groups of (singular) toric varieties may appear as fundamental groups of $\breve{X}(B)$.
	
	The proof of simply connectedness given in \cite{G01} then concludes by proving that $\pi_1(\breve{X}(B))$ is abelian, but that $H^1(\breve{X}(B),\ZZ_n) \neq 0$ would imply $H^0(B,R^1\pi_\star(\ZZ_n)) \neq 0$ by the Leray spectral sequence and simply connectedness of $B$.
	
	We then only need to check that $h^0(B,R^1\pi_\star(\ZZ_n)) = 0$ for all $n$. This follows directly from monodromy considerations, exactly as in the case of the quintic considered in \cite{G01}. That is, a section of $R^1\pi_\star(\ZZ_n)$ is required to be invariant under every monodromy transformation defined by $\Delta$, however this invariant subspace is necessarily trivial.
\end{proof}

\begin{rem}
While not all the cases enumerated in Appendix~\ref{sec:tables} use the method defined in \S\ref{sec:smooth_decompositions} we can nonetheless extend this argument to those additional cases. In the product cases we know that, by construction $\breve{X}(B)$ is the product of two simply connected spaces. In the remaining cases, described in \S\ref{sec:method_2_calculations} we only need to check that cycles invariant under monodromy transformations are collapsed to points by moving the cycle into the boundary. Given this calculation, we conclude that $b_1(\breve{X}(B)) = 0$ for every affine manifold described in Appendix~\ref{sec:tables}. 
\end{rem}

Since we have determined the Euler number $e(\breve{X}(B))$ in~\S\ref{sec:euler_characteristic} we only need to compute $b_2(\breve{X}(B))$ to determine all the Betti numbers of $\breve{X}(B)$.

\begin{rem}
	If we assume that $\breve{X}(B)$ is homotopy equivalent to a Fano manifold $X$ we have the identities:
	\begin{align*}
	b_3 = 2h^{1,2} && \text{and,} && b_2 = h^{1,1} = \rho_X
	\end{align*}
	where $\rho_X$ is the Picard rank of $X$. Thus we can generate lists of expected numerical invariants of Fano manifolds from the Betti numbers of $\breve{X}(B)$ and the degree calculation made in \S\ref{sec:degree}.
\end{rem}

We compute the second Betti number in terms of the limit of a functor $T^\bot\colon \fC^{op} \rightarrow \Vect$.

\begin{dfn}
	Given ray data $J$ for a fan $\Sigma$, we define the functor $T^\bot\colon \fC^{op} \rightarrow \Vect$, defined on objects by defining
	\[
	T^\bot(\tau) =
	\begin{cases}
	M_\QQ/\langle\tau\rangle & \text{ for $\tau \in \Sigma(2)$} \\
	M_\QQ/\langle\rho\rangle & \text{ for $\tau = P_D \in J(\rho)$, $\rho \in \Sigma(1)$ such that $\dim P_D = 2$} \\
	M_\QQ/\langle\sigma\rangle & \text{ for $\tau = P_D \in J(\rho)$, $\sigma \in \Sigma(2)$, $\dim P_D = 1$, and $\hom(\sigma, \tau) \neq \varnothing$}
	\end{cases}
	\]
	The morphisms are then sent to the projection maps induced by the canonical inclusion maps of the subspaces generated by the cones. Let $\Gamma(\Sigma,J)$ denote the inverse limit of $T^\bot$ in $\Vect$.
\end{dfn}

\begin{rem}
	\label{rem:global_sections}
	Note that, from the construction of an inverse limit of groups, 
	\[
	\Gamma(\Sigma,J) \subset \bigoplus_{\tau \in \Objects{\fC}} T^\bot(\tau).
	\]
	Moreover, an element of $\Gamma(\Sigma,J)$ is determined by its values on $\Sigma(2)$, and viewed in this way $\Gamma(\Sigma,J)$ is the set of integral 1-forms on $\sigma^\bot$ for $\sigma \in \Sigma(2)$ which satisfy certain gluing conditions over the rays of $\Sigma$. In particular, the composition
	\[
	\Gamma(\Sigma,J) \subset \bigoplus_{\tau \in \Objects{\fC}} T^\bot(\tau) \to \bigoplus_{\tau \in \Sigma[2]} T^\bot(\tau) 
	\]
	is injective, and we may regard $\Gamma(\Sigma,J)$ as a vector subspace of $\bigoplus_{\tau \in \Sigma[2]} T^\bot(\tau)$.
\end{rem}

\begin{thm}
	\label{thm:picard_rank}

	Given a reflexive polytope $P$ and a set $\bM$ of smooth Minkowski decompositions of its facets let $(\Sigma,C,J)$ be degeneration defined using the method described in \S\ref{sec:smooth_decompositions}, and let $B := B_{P,\bM}$ be the affine manifold constructed in~\ref{cons:slabs_to_affine_structure}. The second Betti number of $\breve{X}(B)$ is given by the following formula.
	\[
	b_2(\breve{X}(B)) = \dim \Gamma(\Sigma,J) - 2
	\]
\end{thm}

The remainder of this section is devoted to the computation of groups appearing in the Leray spectral sequence associated to a contraction map $\xi$, analogous to the map studied in Section~$4$ of \cite{G01}, see Appendix~\ref{sec:contraction}. For the remainder of this section we fix a reflexive polytope $P$ and a set $\bM$ of smooth Minkowski decompositions of the facets of $P$ and let $B := B_{P,\bM}$. Recall from~\S\ref{sec:smooth_decompositions} that given a choice of $P$ and $\bM$ we fix the degeneration data:

\begin{enumerate}
	\item $\Sigma$, the normal fan of $P$,
	\item $C$, the function $E \mapsto \ell(E^\star)$ for all $E \in \Edges{P^\circ}$, and,
	\item $J$ induced by the smooth Minkowski decompositions, $\bM$.
\end{enumerate}

\begin{dfn}
The fan $\Sigma$ induces a polyhedral decomposition of $P^\circ$, let $\breve{X}_0(B)$ be the union of polarised toric varieties with moment polytopes given by the maximal components of $\cP$, identified along the toric strata which are identified by $\Sigma$.
\end{dfn}

\begin{rem}
	The variety $\breve{X}_0(B)$ is the central fibre of the toric degeneration constructed by Gross--Siebert in \cite{Gross--Siebert} and the Gross--Siebert reconstruction algorithm constructs a formal deformation of $\breve{X}_0(B)$ from a choice of log structure on  $\breve{X}_0(B)$.
\end{rem}

Let $F_k$ denote the disjoint union of toric codimension $k$ strata of $\breve{X}_0(B)$ which do not project to boundary strata of $B$. Following the proof of~\cite[Theorem~$4.1$]{G01}, we define maps $i_k$ for $k \in \{0,\ldots,3\}$, the canonical inclusions of $F_k \setminus F_{k+1}$ into $\breve{X}_0(B)$. Note that each $F_k \setminus F_{k+1}$ contains points in the toric boundary of each $Z \in F_i$ which lie in boundary strata of $\breve{X}_0(B)$. We compute the Betti numbers of $\breve{X}(B)$ via the Leray spectral sequence associated to the map $\xi \colon \breve{X}(B) \rightarrow \breve{X}_0(B)$.

\begin{pro}
	\label{pro:Leray--Serre}
	Several of the ranks of the cohomology groups obtained by pushing forward the constant sheaf $\QQ$ along $\xi$ are as follows.
	\[	h^0(\breve{X}_0(B),R^i\xi_\star\QQ) = 
	\begin{cases}
	1 & \text{ if } i \in \{0,3\}  \\
	0 & \text{ if } i \in \{1,2\}
	\end{cases}
	\]
and
\[	h^j(\breve{X}_0(B),\xi_\star\QQ) = 
\begin{cases}
1 & \text{ if } j \in \{0,2\}  \\
0 & \text{ if } j \in \{1,3\}
\end{cases}
\]
\end{pro}

\begin{rem}
	The Leray spectral sequence for $\xi$ computes the cohomology of $\breve{X}(B)$:
	\[
	H^p(\breve{X}_0(B),R^q\xi_\star \QQ) \Rightarrow H^{p+q}(\breve{X}(B),\QQ)
	\]
	By Proposition~\ref{pro:Leray--Serre} the $E^2_{p,q}$ page of this spectral sequence has the following form:
	\[
	\begin{array}{cccc}
	\QQ & & & \\
	0 & \star & & \\
	0 & \QQ^{R-1} & \star & \\
	\QQ & 0 & \QQ & 0 \\
	\end{array}
	\]
	where $R := b_2(\breve{X}(B))$.	In particular $b_2$ is determined by the ranks of groups appearing on the $E^2$ page of this spectral sequence.
\end{rem}

\begin{proof}[Proof of Proposition~\ref{pro:Leray--Serre}]
	This proof follows the structure of the proof of Theorem~$4.1$  of \cite{G01}. First observe that $R^3\xi_\star\QQ = \QQ_{p}$, the skyscraper sheaf over the point $p$, which is the unique point of $B$ contained in the fibre over the origin of the map $\breve{X}_0(B) \rightarrow B$, and thus,
	\[
	H^0(\breve{X}_0(B),R^3\xi_\star\QQ) \cong \QQ.
	\]
	Second, we consider the map
	\[
	R^2\xi_\star\QQ \rightarrow {i_2}_\star{i_2}^\star R^2\xi_\star\QQ,
	\]
	 following the argument used in \cite{G01} we see that this map has zero kernel and by left exactness of global sections we have an inclusion
	 \[
	 H^0(R^2\xi_\star\QQ) \hookrightarrow H^0({i_2}_\star i_2^\star R^2\xi_\star\QQ).
	 \]
	 We can describe ${i_2}_\star i_2^\star R^2\xi_\star\QQ$ explicitly, since it is the direct sum of its restrictions to the one dimensional strata of the decomposition of $P^\circ$ induced by $\Sigma$. Each such stratum is isomorphic to $\PP^1$ and the restriction of $ {i_2}_\star{i_2}^\star R^2\xi_\star\QQ$ is isomorphic to the constant sheaf $\QQ$ away from a finite (and non-empty) set of points which have trivial stalks. Thus we have that
	 \[
	 \dim  H^0(\breve{X}_0(B), R^2\xi_\star\QQ)  = \dim H^0(\breve{X}_0(B), {i_2}_\star{i_2}^\star R^2\xi_\star\QQ) = 0.
	 \]
	
	Similarly, consider the map:
	\[
		R^1\xi_\star\QQ \rightarrow {i_1}_\star{i_1}^\star R^1\xi_\star\QQ.
	\]
	Again -- following the argument in \cite[p.$46$]{G01} -- we have that this map has zero kernel, and
	\[
	\dim  H^0(\breve{X}_0(B), R^1\xi_\star\QQ)  = \dim H^0(\breve{X}_0(B), {i_1}_\star{i_1}^\star R^1\xi_\star\QQ) = 0.
	\]
	Reasoning as in the proof of \cite[Theorem~$4.1$(c)]{G01}, we have the equality
	\[
	{i_1}_\star{i_1}^\star\xi_\star\QQ = \bigoplus_{F}\QQ_{F \setminus C},
	\]
	where the sum is taken over two dimensional non-boundary toric strata of $\breve{X}_0(B)$. Indeed, fixing a two dimensional non-boundary stratum, the stalks of ${i_1}_\star{i_1}^\star\xi_\star\QQ$ are isomorphic to $\QQ$ precisely when $x \notin C$, and trivial otherwise. Note that while the domain $i_1$ excludes some boundary components of each slab, stalks of ${i_1}_\star{i_1}^\star\xi_\star\QQ$ over points in these boundary components are not necessarily trivial. The difference from the analysis made in \cite{G01} comes along stalks at points $x$ in the (remaining) boundary strata of $F$; however -- since the boundary of $B$ is smooth -- stalks away from $C$ are also isomorphic to $\QQ$. Since, for each $k$, $H^k(F,\QQ_{F \setminus C}) = H^k_c(F\setminus C,\QQ)$, we have that $H^0(F,\QQ_{F \setminus C}) = H^1(F,\QQ_{F \setminus C}) = 0$; hence,
	\[
	\dim H^1(\breve{X}_0(B), {i_1}_\star{i_1}^\star R^1\xi_\star\QQ) = 0.
	\]
	
	We next consider the cohomology groups $H^j(\breve{X}_0(B),\xi_\star\QQ)$. Note that since all the fibres of $\xi$ are connected, we have that
	\[
	\xi_\star\QQ \cong \QQ,
	\]
	thus these cohomology groups are nothing other than the ordinary rational cohomology groups of $\breve{X}_0(B)$. Following the proof of Theorem~$4.1$ in \cite{G01}, we use the spectral sequence associated to the decomposition of $\breve{X}_0(B)$. Noting that the underlying complex of the decomposition of $B$ is homeomorphic to a ball (rather than a sphere), and that each toric variety $Y$ in the decomposition of $\breve{X}_0(B)$ has $H^0(Y,\QQ) \cong H^2(Y,\QQ) \cong \QQ$, we obtain the following (truncated) $E_2$ page.
	\[
	\xymatrix@R-2pc{
		\QQ & 0 & 0 \\
		0 & 0 & 0 \\
		\QQ & 0 & 0 & \\
	}
	\]
	This completes the calculation of the ranks of the cohomology groups we require.
\end{proof}

Having established the identity,
\[
b_2(\breve{X}(B)) = 1 + \dim H^1(\breve{X}_0(B),R^1\xi_\star\QQ),
\]
the purpose of the remainder of this section to compute the cohomology group $H^1(\breve{X}_0(B),R^1\xi_\star\QQ)$ in terms of the space $\Gamma(\Sigma,J)$. We proceed by attempting to continue to imitate the proof of~\cite[Theorem~$4.1$]{G01}. In particular we begin by defining the sheaf
\[
\cF := \coker(R^1\xi_\star\QQ \rightarrow {i_1}_\star i_1^\star R^1\xi_\star\QQ),
\]
and study the map $\cF \rightarrow {i_2}_\star i^\star_2 \cF$. From the short exact sequence
\[
0 \rightarrow R^1\xi_\star\QQ \rightarrow {i_1}_\star i_1^\star R^1\xi_\star\QQ \rightarrow \cF \rightarrow 0,
\]
the corresponding long exact sequence, and recalling from the proof of Proposition~\ref{pro:Leray--Serre} that both the zero and first cohomology groups of ${i_1}_\star i_1^\star R^1\xi_\star\QQ$ vanish, it is immediate that
\[
H^1(\breve{X}_0(B),R^1\xi_\star\QQ) \cong H^0(\breve{X}_0(B),\cF).
\]

In \cite{G01} the same argument we have employed in the proof of Proposition~\ref{pro:Leray--Serre} extends to show that this group vanishes: that is, the map  $\cF \rightarrow {i_2}_\star i^\star_2 \cF$ is monomorphic and the target sheaf has no non-trivial global sections. We observe that in the current context both of these properties may fail.

We begin with an analysis of the map
\[
\cF \rightarrow {i_2}_\star i^\star_2 \cF
\]
analogous to that in~\cite{G01}. We first note that the cokernel of this map is supported at the  zero stratum $p$ of $\breve{X}_0(B)$ which projects to the origin in $P^\circ$. Choose points $p_r$ for $r \in \Sigma^+(1)$ near $p$ such that $p_r$ is contained in the ray $r$, and points $p_s$ for each $s \in \Sigma(2)$ contained in the polygon $c$ such that $(c,D) \in \fS$ and $c \subset s$. Moreover choose the points $p_s \in B$ in a small neighbourhood of $p$. We then have the following commutative diagram, analogous to that appearing in \cite[p.46]{G01}.

\begin{equation}
\label{eq:cokernel_stalk}
\xymatrix{
	0 \ar[r] & H^1(\xi^{-1}(p),\QQ) \ar[r] \ar^{\phi_1}[d] & \bigoplus_s H^1(\xi^{-1}(p_s),\QQ) \ar[r] \ar^{\phi_2}[d] & \cF_p \ar[r] \ar^{\phi_3}[d] & 0\\
	0 \ar[r] & \bigoplus_r H^1(\xi^{-1}(p_r),\QQ) \ar^{\theta}[r] & \bigoplus_{r,s} H^1(\xi^{-1}(p_s),\QQ) \ar[r] & \bigoplus_r\cF_{p_r} \ar[r] & 0
}
\end{equation}
\medskip

\noindent where the sum $\bigoplus_{r,s} H^1(\xi^{-1}(p_s),\QQ)$ is taken over pairs $(r,s)$ such that the ray $r$ is contained in $s \in \Sigma(2)$. Note that the map $\phi_3$ is the map $\cF \rightarrow {i_2}_\star i^\star_2 \cF$ restricted to the respective stalks of these sheaves at $p$. The map $\phi_2$ is the map $\alpha \mapsto \alpha \oplus \alpha$, and $\phi_1$ is the dual specialisation map (dual to the tuple of inclusions of the two dimensional tori $\xi^{-1}(p_r)$ into the three dimensional torus $\xi^{-1}(p)$). After a short diagram chase we see that the rank of the kernel of $\phi_3$ is equal to
\[
\dim (\Ima(\theta) \cap \Ima(\phi_2)) - 3.
\]

Next we compute the image of $H^0(\cF)$ in $H^0({i_2}_\star i^\star_2 \cF)$. To do this we first describe the latter group. Clearly $\cF' := {i_2}_\star i^\star_2 \cF$ is concentrated on the one dimensional strata of $\breve{X}_0(B)$, that is, on a union of projective lines.

Fixing a ray $r$ of $\Sigma$ let $s_1,\ldots, s_k$ denote the slabs meeting $r$. Given a point $q$ on the projective line corresponding to $r$ not contained in the singular locus, $\cF'_q$ is the cokernel of the specialization map
\[
H^1(\xi^{-1}(q),\QQ) \cong \QQ^2 \rightarrow \bigoplus_{1 \leq j \leq k} H^1(\xi^{-1}(q_j),\QQ) \cong \QQ^k,
\]
where $q_j \in s_j$ are points close to $q$ for each $j \in \{1,\ldots,k\}$. Suppose now that $q$ is the image under $\xi$ of the singular point of the fibre lying over a positive node of $B$, and let $j_1$, $j_2$, and $j_3$ in $\{1,\ldots,k\}$ be the indices of the (distinct) slabs whose interiors intersect the singular locus in any neighbourhood of the image of $q$ in $B$. Setting $I_q := \{1,\ldots, k\} \setminus \{j_1,j_2,j_3\}$, $\cF'_q$ is the cokernel of the specialization map
\[
H^1(\xi^{-1}(q),\QQ) \cong \{0\} \rightarrow \bigoplus_{j \in I_q} H^1(\xi^{-1}(p_{s_j}),\QQ) \cong \QQ^{k-3}
\]
where the second sum is over slabs meeting $r$ which do not meet singular locus near $q$. Suppose finally that $q$ lies over a general point of the singular locus of $B$ and $q$ is the image (under $\xi$) of the circle of singularities of this fibre. Then $\cF'_q$ is the cokernel of the specialization map
\[
H^1(\xi^{-1}(q),\QQ) \cong \QQ^1 \rightarrow \bigoplus_{j \in I_q} H^1(\xi^{-1}(p_{s_j}),\QQ) \cong \QQ^{k-2}
\]
where, again, $I_q$ indexes slabs $s_j$ for $j \in \{1,\ldots,k\}$ such that the interior of $s_j$ does not intersect the singular locus in some neighbourhood of the image of $q$ in $B$. To determine the global sections of $\cF'$ on this projective line we also need to compute the restriction maps of this sheaf. Let $q$ correspond to a singular point, and $q'$ a general nearby point on $\PP^1$. Then the restriction map is defined by the diagram
\begin{equation}
\label{eq:cFprime}
\xymatrix{
		H^1(\xi^{-1}(q),\QQ) \ar[r] \ar[d] & \bigoplus\limits_{j \in I_q} H^1(\xi^{-1}(q_j),\QQ) \ar[r] \ar[d] & \cF'_{q} \ar^{\alpha_{q}}[d] \\
	H^1(\xi^{-1}(q'),\QQ) \ar[r] & \bigoplus\limits_{1 \leq j \leq k} H^1(\xi^{-1}(q_j),\QQ) \ar[r] & \cF'_{q'}
}
\end{equation}
\medskip

Where the first vertical maps is the usual specialization maps and the second is the canonical inclusion of vector spaces. Thus for every singular point $q$ of $B$ contained in a ray $r$ of $\Sigma$ there is a map $\alpha_q \colon \QQ^{k-3} \rightarrow \QQ^{k-2}$ corresponding to the restriction of sections defined near $q$ to those defined near a general nearby point. We claim that $\alpha_q$ is injective. Indeed, consider the intersection $U$ of the images of $H^1(\xi^{-1}(q'),\QQ) \cong \QQ^2$ and $\bigoplus\limits_{j \in I_q} H^1(\xi^{-1}(q_j),\QQ)$ in $\bigoplus\limits_{1 \leq j \leq k}H^1(\xi^{-1}(q_j),\QQ)$; we have two cases:
\begin{enumerate}
	\item If the image of $q$ in $B$ is a positive node, the space $U$ is trivial; indeed any non-zero vector $v \in H^1(\xi^{-1}(q'),\QQ)$ has a non-zero image in $H^1(\xi^{-1}(q_j),\QQ)$ for some $j \in \{j_1,j_2,j_3\}$.
	\item If the image of $q$ in $B$ is not a trivalent point of $\Delta$, the space $U$ is isomorphic to $\QQ$, the image of the one dimensional vector subspace in $H^1(\xi^{-1}(q'),\QQ)$ whose image in $H^1(\xi^{-1}(q_j),\QQ)$ is trivial for $j \in \{j_1,j_2\}$. However, as the map $H^1(\xi^{-1}(q),\QQ) \to H^1(\xi^{-1}(q'),\QQ)$ is injective, $U$ is isomorphic to the image of the kernel of the projection
	\[
	 \bigoplus\limits_{j \in I_q} H^1(\xi^{-1}(q_j),\QQ) \to \cF'_{q}
	\] 
	in $\bigoplus\limits_{1 \leq j \leq k} H^1(\xi^{-1}(q_j),\QQ)$.
\end{enumerate}
However, a non-zero element in $\ker(\alpha_q)$ determines an element of $U$ which is not in the image of the composition 
\[
H^1(\xi^{-1}(q),\QQ)  \to H^1(\xi^{-1}(q'),\QQ) \to \bigoplus\limits_{1 \leq j \leq k} H^1(\xi^{-1}(q_j),\QQ),
\] 
and such elements of $U$ do not exist in either of the two cases described above.

The space of global sections on this $\PP^1$ (corresponding to a ray $r$ of $\Sigma$) is the intersection $\bar{V}_r$ of the images of the $\alpha_q$. Identifying the cohomology groups $H^1(\xi^{-1}(q_j),\QQ)$ and $H^1(\xi^{-1}(p_{s_j}),\QQ)$ for all $j \in \{1,\ldots, k\}$, let $V_r$ denote the pre-image of $\bar{V}_r$ in $\bigoplus_{1 \leq j \leq k}H^1(\xi^{-1}(p_{s_j}),\QQ)$ along the projection 
\[
\bigoplus_{1 \leq j \leq k}H^1(\xi^{-1}(q_j),\QQ) \to \cF_{q'}.
\]
where $q_j$, for $j \in \{1,\ldots,k\}$, and $q'$ are as defined above. We define 
\[
V := \bigoplus_{r \in \Sigma(1)} V_r \subset \bigoplus_{r,s} H^1(\xi^{-1}(p_s),\QQ)
\]
to be the sum of the subspaces $V_r$, where the sum in the second term is taken over pairs $(r,s)$ such that the ray $r$ is contained in $s \in \Sigma(2)$.
 
\begin{eg}
	Consider the case in which the singular locus of $B$ meets a ray $r$ in two transverse directions. This occurs, for example, if $r$ contains a vertex of $P^\circ$ dual to a square facet of $P$. In this case, the $\PP^1$ corresponding to this segment has $H^0(\cF',\QQ)$ equal to the intersection of two one-dimensional subspaces inside a two-dimensional space, that is, (as in the case of a single positive node familiar from~\cite{G01}) that $h^0(\cF',\QQ) = 0$.
\end{eg}

\begin{eg}
	Consider a ray $\rho$ of $\Sigma$ which meets a vertex of $P^\circ$ dual to a hexagonal facet of $P$. There are two choices for $J(\rho)$, corresponding to two smooth Minkowski decompositions of the hexagon shown in Figure~\ref{fig:dP6}.
	\begin{center}
		\begin{figure}
			\includegraphics[scale=0.6]{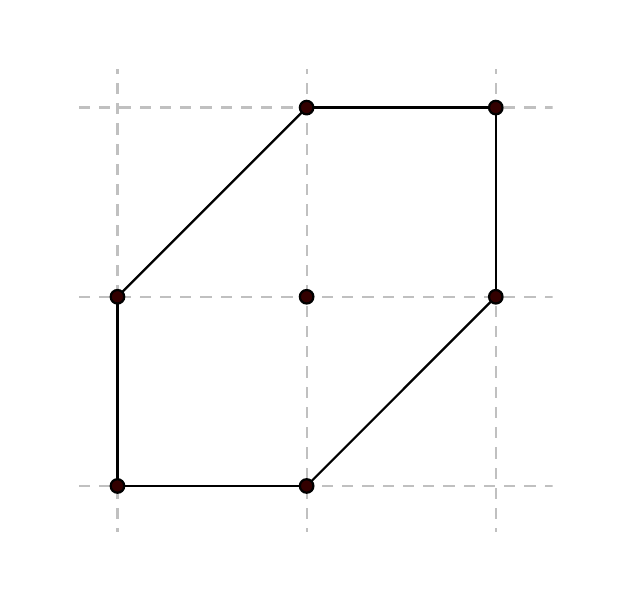}
			\includegraphics[scale=0.6]{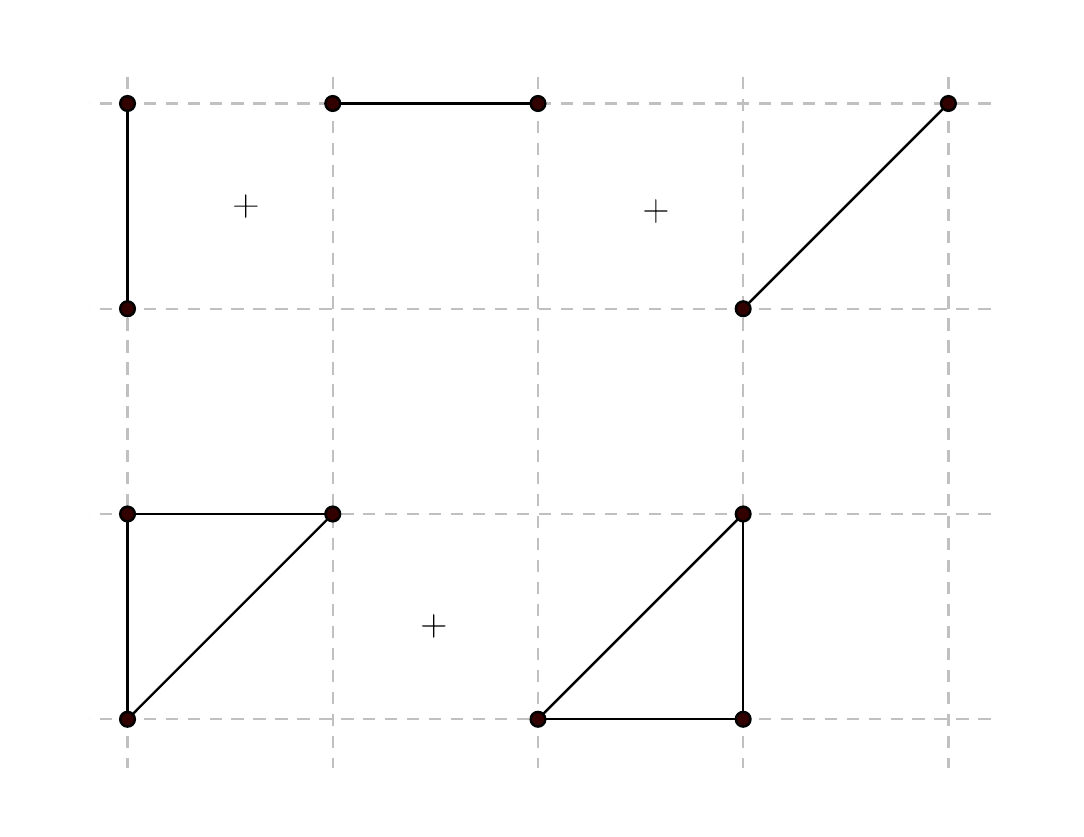}
			\caption{Minkowski decompositions of a hexagon}
			\label{fig:dP6}
		\end{figure}
	\end{center}
	For one of these choices (decomposing the facet of $P$ into a pair of triangles) there are two positive nodes lying on $\rho$, and the corresponding summand of $H^0(\cF',\QQ)$ is the intersection of two transverse three-dimensional subspaces of $\QQ^4$. In the other case there are no positive nodes lying on $\rho$, but three generic singular points. In this case the corresponding summand of $H^0(\cF',\QQ)$ is then the intersection of three transverse three-dimensional subspaces of $\QQ^4$.
\end{eg}

Having described the vector space $H^0(\cF') = H^0({i_2}_\star{i_2}^\star\cF)$ we interpret the image of the map $H^0(\cF) \to H^0({i_2}_\star{i_2}^\star\cF)$. In particular, we rephrase this as a `gluing condition' for sections of $H^0({i_2}_\star{i_2}^\star\cF)$ over $\cF_p$ (recalling that $p$ is the pre-image of the origin in $B$ in $\breve{X}_0(B)$). From diagram~\eqref{eq:cokernel_stalk}, we see that the global sections of $\cF$ are obtained by first taking the pre-image of the subspace $V$ in $\bigoplus_sH^1(\xi^{-1}(p_s),\QQ)$ along $\phi_2$, and taking the quotient by $H^1(\xi^{-1}(p),\QQ) \cong \QQ^3$.

To conclude the proof of Theorem~\ref{thm:picard_rank} we need to interpret $H^0(\cF)$ in terms of the vector space $\Gamma(\Sigma,J)$. To do this we need a basic observation from toric geometry.

\begin{lem}
	Given a cone $\sigma \in \Sigma(2)$, $T^\bot(\sigma)$ is canonically isomorphic to $H^1(\xi^{-1}(p),\ZZ)$, where $p$ is a general point in the toric stratum of $\breve{X}_0(B)$ which projects to a point $b \in \sigma \cap B_0$. 
\end{lem}
\begin{proof}
	The cotangent space of $b$ is canonically identified with $N_\RR$ which contains the lattice $N$, so $H_1(\pi^{-1}(b),\ZZ)$ is identified with $N$, and contains a distinguished one dimensional subspace, annihilating $\sigma$, canonically identified with elements of $H_1(\xi^{-1}(p),\ZZ)$. The dual subspace is then identified with the quotient of $M$ by the span of $\sigma$.
\end{proof}

Thus we have a canonical isomorphism
\[
\bigoplus_sH^1(\xi^{-1}(p_s),\ZZ) \cong \bigoplus_{\sigma \in \Sigma(2)}T^\bot(\sigma),
\]
and a subspace on each side, given by the pre-image of $V$ on the left and given by $\Gamma(\Sigma,J)$ on the right, see Remark~\ref{rem:global_sections}. Each element of $J$ defines a single linear condition on each side, and explicit computation shows that these are in fact identical conditions: both imply a gluing condition on the sections defined on the neighbouring two dimensional cones of $\Sigma$. There are two cases, depending on the dimension of the factor in $J$. In the case of a positive node (a two-dimensional factor in $J(\rho)$ for some ray $\rho$), the diagram~\eqref{eq:cFprime} becomes:
\[
\xymatrix{
	\{0\} \ar[r] \ar[d] & \QQ^{k-3} \ar[r] \ar[d] & \cF'_{q'} \ar^{\alpha_q}[d] \\
	\QQ^2 \ar[r] & \QQ^k \ar[r] & \cF'_q
}
\]
That is, the condition imposed on $\QQ^k$ by the element of $P_D \in J(\rho)$ is that values on the factors corresponding to $T^\bot(\sigma)$ such that $\hom(\sigma,P_D) \neq 0$ are sum to an element of $M/\langle\rho\rangle \cong  H^1(\xi^{-1}(p_r),\QQ)$. In the second case, that of a one-dimensional factor in $J(\rho)$, the diagram~\eqref{eq:cFprime} becomes:
\[
\xymatrix{
	\QQ \ar[r] \ar[d] & \QQ^{k-2} \ar[r] \ar[d] & \cF'_{q'} \ar^{\alpha_q}[d] \\
	\QQ^2 \ar[r] & \QQ^k \ar[r] & \cF'_q
}
\]
That is, the image of $\alpha_q$ is the image of the orbit of $\QQ^{k-2}$ by $\QQ^2$. In other words, elements of $\bigoplus_{\sigma \in \Sigma(2)}T^\bot(\sigma)$ such that the two components supporting $\Delta$ near $q$ sum to zero.

Since $h^1(\xi^{-1}(p),\QQ) = 3$ we have that $R+1 = \dim \Gamma(\Sigma,J)-3$, as expected, and we conclude the proof of Theorem~\ref{thm:picard_rank}. In fact, in many computations we can make use of a simpler directed system than $T^\bot$ to compute $\Gamma(\Sigma,J)$.

\begin{dfn}
Let $\bar{T}^\bot$ denote the functor $\Sigma[1,2] \rightarrow \Vect$ given by $\sigma \mapsto M_\QQ/\langle \sigma\rangle$. Recall that $\Sigma[1,2]$ denotes the poset of one and two dimensional cones of $\Sigma$.
\end{dfn}

Note that the diagram
\[
\xymatrix{
	\fC \ar^{T^\bot}[r] \ar[d] & \Vect \\
	\Sigma[1,2] \ar_{\bar{T}^\bot}[ur] & \\
}
\]
does not commute, since the value of $T^\bot$ generally depends on $J$.	

Observe that we can interpret $\bar{T}^\bot$ as a constructible sheaf on a graph obtained by projecting the cones in $\Sigma[1,2]$ to the unit sphere in $M_\RR$. In fact we observe that since the degeneration data we consider in this section uses the degeneration data described in \S\ref{sec:smooth_decompositions} this graph is nothing other than the one-skeleton of $P^\circ$, which we denote $P^\circ[1]$. The stalk of this sheaf over a point $p$ is then equal to $M_\QQ/\langle \sigma\rangle$ where $\sigma$ is the minimal cone of $\Sigma[1,2]$ projecting to $p$. It is often the case that $\Gamma(\Sigma,J)$ coincides with the global sections of $\bar{T}^\bot$.

\begin{lem}
	\label{lem:collapsing}
	If all the codimension one toric strata of $X_{P^\circ}$ (the toric variety with fan defined by the normal fan of $P$) belong to the set $\{\PP^2,\PP^1\times\PP^1,\FF_1,dP_7\}$ then
	\[
	\Gamma(\Sigma,J) \cong H^0(P^\circ[1],\bar{T}^\bot),
	\]
	where $P^\circ[1]$ denotes the one-skeleton of the polytope $P^\circ$, and we recall that $dP_7$ is the toric surface obtained by blowing up $\PP^1\times \PP^1$ in a reduced torus invariant point.
\end{lem}
\begin{proof}
	We recall that both vector spaces are canonically identified with subspaces of $\bigoplus_\sigma M_\QQ/\langle \sigma\rangle$. Considering a ray $\rho$ of $\Sigma$ if the corresponding 
	In all four cases enumerated the gluing conditions require that the elements of
	\[
	\{T^\bot(\sigma) : \sigma \in \Sigma(2)\}
	\]
	are obtained from an element of $\bar{T}(\rho)$. That is, a choice of sections defines an element of $\Gamma(\Sigma,J)$ if and only if it defines an element of $H^0(P^\circ[1],\bar{T}^\bot)$.
\end{proof}

\begin{rem}
	Note that in the sequel we will compute $\Gamma(\Sigma,J)$ by hand, using identifications of the spaces $M_\QQ/\langle \sigma \rangle$ for each $\sigma \in \Sigma(2)$ with $\QQ$. The choice of basis -- that is, the choice of orientation -- of each $M_\QQ/\langle \sigma\rangle$ affects the compatibility conditions around each ray, and in general some care is needed to express these correctly. In particular if normal directions are chosen around a positive node $x \in B$ compatibly with a cyclic ordering of the cones $\sigma_1$, $\sigma_2$, and $\sigma_3$ intersecting $\Delta$ near $x$, the relation on the elements $\alpha_i \in \QQ \cong M/\langle \sigma_i \rangle$ for each $i \in \{1,2,3\}$ becomes $\alpha_1 + \alpha_2 + \alpha_3=0$. 
\end{rem}

%% file: top_classification.tex
% !TEX root = fano_manifold_topology.tex

%----------------------------------------------------------------------
\section{Topological classification}
\label{sec:top_classification}
%----------------------------------------------------------------------

In this section we prove Theorem~\ref{thm:rank_one_models}, namely we prove that for all the models given in Appendix~\ref{sec:tables} such that $b_2(\breve{X}(B)) = 1$, the manifold $\breve{X}(B)$ is homeomorphic to the expected Fano threefold. This relies on computing a complete set of topological invariants for $\breve{X}(B)$ and applying the topological classification result of Jupp~\cite{J73}, generalising those of Wall for spin manifolds~\cite{Wall66}.

\begin{thm}[Jupp, \cite{J73}]
	\label{thm:top_classification}
	The assignment
	\[
	X \mapsto \left( b_3(X)/2,H^2(X,\ZZ),w_2(X),\tau(X),F_X,p_1(X) \right)
	\]
	induces a $1$-$1$ correspondence between oriented homeomorphism classes of $1$-connected $6$-dimensional topological manifolds with torsion free homology and equivalence classes of admissible systems of invariants. Moreover a topological manifold admits a smooth structure if and only if the class $\tau(X)$ vanishes.
\end{thm}

Remarking that we can always adjust the compactifications of torus fibrations we consider such that the total space is a smooth manifold, Theorem~\ref{thm:top_classification} implies the classification is complete once we have determined the following invariants of $\breve{X}(B)$ for a given affine manifold $B$, and shown that it has torsion free homology.

\begin{enumerate}
	\item The Betti numbers of $\breve{X}(B)$.
	\item The second Stiefel--Whitney class $w_2(\breve{X}(B)) \in H^2(X,\ZZ_2)$.
	\item The first Pontryagin class $p_1(\breve{X}(B)) \in H^4(X,\ZZ)$.
	\item The cubic form $F_{\breve{X}(B)}$ on $H^2(\breve{X}(B),\ZZ)$ induced by the cup product.
\end{enumerate}
  
Since we will only be concerned with models of rank one Fano threefolds up to homeomorphism, the cubic form is determined by the index of $[\pi^{-1}(\breve{X}(B))]$ and its triple intersection number. We first compute the index of the class $[\pi^{-1}(\partial B)] \in H_4(\breve{X}(B),\ZZ) \cong \ZZ$.

\subsection{Fano index}
\label{sec:fano_index}

The index of a Fano manifold $X$ is the maximal integer $n$ such that $K_X = nH$ for a class $H \in H^2(X,\ZZ)$. We show how to recover this invariant in the case $b_2(\breve{X}(B))=1$ and $B$ is a manifold obtained from a smooth Minkowski decomposition (as described in~\S\ref{sec:smooth_decompositions}).

\begin{dfn}
	Given smooth degeneration data $(\Sigma,C,J)$ let $B$ denote the affine manifold obtained via Construction~\ref{cons:slabs_to_affine_structure}, we define the \emph{index} $i(\breve{X}(B))$ of $\breve{X}(B)$ to be the index of $[\pi^{-1}(\partial B)] \in H_4(\breve{X}(B),\ZZ)$.
\end{dfn}

We make note of the following elementary lemma on the cohomology of projective toric cones for later use.
\begin{lem}
	Given a projective toric variety $Y$ together with a very ample line bundle $L$ embedding $Y$ into $\PP^n$, the projective closure of the affine cone $\bar{Y}$ of $Y$ in $\PP^{n+1}$ has $H^2(\bar{Y},\ZZ) \cong \ZZ$.
\end{lem}
\begin{proof}
	Recall that -- as $\bar{Y}$ is toric -- $H^2(\bar{Y},\ZZ) \cong \Pic(\bar{Y})$, which is itself isomorphic to the lattice of piecewise linear functions $\theta$ on the fan determined by $\bar{Y}$. Recall that a subset $S$ of the rays of the fan determined by $\bar{Y}$ is in canonical bijection with the rays of the fan determined by $Y$. The rays of $S$ span a cone, and hence -- up to adding a linear function -- we may assume that $\theta$ vanishes on every ray in $S$. Moreover, the complement of $S$ is a singleton set, and the value of $\theta$ on this ray defines a bijection $\Pic(\bar{Y}) \to \ZZ$.
\end{proof}

\begin{pro}
	\label{pro:fano_index}
	Let $B$ be a model for a rank one Fano threefold described in Appendix~\ref{sec:tables}, the class $[\pi^{-1}(\partial B)] \in H_4(\breve{X}(B),\ZZ)$ is Poincar\'{e} dual to a class in $H^2(\breve{X}(B),\ZZ)$ of the expected index.
\end{pro}
\begin{proof}
	In the case the affine manifold $B$ is constructed via the method given in \S\ref{sec:smooth_decompositions} we can follow the analysis of the Leray spectral sequence of $\xi$ in \S{\ref{sec:betti_numbers}}. If $b_2(\breve{X}(B)) = 1$ we have that $H^2(\breve{X}(B)) \cong H^2(\breve{X}_0(B)$. Clearly $\pi^{-1}(\partial B)$ defines a class in both groups, which are identified by this isomorphism. Thus, we only need to compute the index of the pre-image of $\partial B$ in the union of toric varieties $\breve{X}_0(B)$.
	
	Using the spectral sequence associated to the decomposition of $\breve{X}_0(B)$ into its constituent toric varieties, we see that $H^2(\breve{X}_0(B),\ZZ)$ is the kernel of $\bigoplus_\sigma H^2(X_\sigma,\ZZ) \rightarrow \bigoplus_\tau H^2(X_\tau,\ZZ)$ for maximal cells $\sigma$ in the decomposition $\Sigma \cap P^\circ$ and codimension one cells $\tau$ not contained in the boundary of $P^\circ$. The toric boundary of $X_P$ canonically determines an element of $\bigoplus_\sigma H^2(X_\sigma,\ZZ)$. Note that each factor in this direct sum is canonically isomorphic to $\ZZ$, and the kernel of the given map is a saturated sublattice (since $\bigoplus_\tau H^2(X_\tau,\ZZ)$ is torsion-free). Each $X_\sigma$ is the cone over a toric surface and the base of this cone is the element of $H^2(X_\sigma,\ZZ)$ determined by the toric boundary of $X_P$. Thus we only need to compute the greatest common divisor of the base of each cone $X_\sigma$.
		
	In fact, the only cases which we do not treat using this method are the models of $V_2$ and $B_1$. In fact, although we treat $V_2$ using the method given in \S\ref{sec:complete_intersections}, the only difference is that the polytope we consider is non-reflexive and this is no barrier to considering the same Leray spectral sequence. The only other example we consider is $B_1$. In this case we cannot apply the Leray spectral sequence, however we know that $\pi^{-1}(\partial B)$ consists of two components, since $[\pi^{-1}(\partial B)]^3 = 8$ and the cubic form on $H^2(\breve{X}(B))$ is integral, the Fano index must be equal to $2$.
\end{proof}

\subsection{Torsion freeness}
\label{sec:torsion_free}

We now show that $H^3(\breve{X}(B),\ZZ)$ is torsion free for each model $B$ of a rank one Fano threefold. In fact we can easily see now that there is no torsion in any cohomology group of $\breve{X}$.
	
\begin{pro}
 Given any Fano threefold $X$, the model $\breve{X}(B)$ of $X$ given in Appendix~\ref{sec:tables} has torsion free (co)homology.
\end{pro}
\begin{proof}
	The cohomology group $H^3(\breve{X}(B),\ZZ)$ may be computed using the Leray spectral sequence for $\xi$ by exactly the same method as used in \S\ref{sec:betti_numbers}. In fact the argument used in \cite{G01} to establish torsion freeness holds in this context, since this relies only on the topology of the complement of curves in weighted projective planes, and of points in $\PP^1$.
	
	The cohomology group $H^2(\breve{X}(B),\ZZ)$ is explicitly described in \S{\ref{sec:betti_numbers}} and, in the case $b_2(\breve{X}(B)) = 1$, is isomorphic to $\ZZ$. The fact that $H^4(\breve{X}(B),\ZZ)$ is torsion-free follows from the universal coefficient theorem and the torsion freeness of $H^3(\breve{X}(B),\ZZ)$.
	 
	The torsion freeness of $H^1(\breve{X}(B),\ZZ)$ and $H^5(\breve{X}(B),\ZZ)$ follow, for example, from simply connectedness. $H^0(\breve{X}(B),\ZZ)$ and $H^6(\breve{X}(B),\ZZ)$ are automatically torsion free.
\end{proof}

\subsection{Characteristic classes}
\label{sec:char_classes}

In order to conclude the proof of Theorem~\ref{thm:rank_one_models} we need to compute the classes $w_2(\breve{X}(B))$ and $p_1(\breve{X}(B))$. In fact, our task is made considerably simpler (significantly simpler than that of \cite{G01}), by the fact that $H^2(\breve{X}(B),\ZZ) \cong \ZZ$ and we have a canonically defined cycle giving a positive class given by $D = \pi^{-1}(\partial B)$. Moreover we know that $\pi^{-1}(\partial B)$ is diffeomorphic to a K$3$ surface and in \S\ref{sec:degree} we computed a cycle in the Euler class of the normal bundle of this embedded K$3$ surface.

\begin{pro}
	\label{pro:w2}
	Given an affine manifold $B$ determined by degeneration data associated to a collection of smooth Minkowski decompositions (see~\S\ref{sec:smooth_decompositions}) such that $b_2(\breve{X}(B)) = 1$ then $w_2(\breve{X}(B)) = \PD[\pi^{-1}(\partial B)]$, where $\PD$ denotes Poincar\'{e} duality. 
\end{pro}
\begin{proof}

Let $D := \pi^{-1}(\partial B)$. Observe that $\PD[D] \smile - $ is an isomorphism 
\[
H^2(\breve{X}(B),\ZZ) \rightarrow H^4(\breve{X}(B),i(\breve{X}(B))\ZZ).
\]
Letting  $\theta$ denote the inclusion $D \hookrightarrow \breve{X}(B)$, we first consider the case when the Fano index $i(\breve{X}(B))$ is not even. In this case $\theta_\star[D]$ reduces mod $2$ to a non-zero class in $H_4(\breve{X}(B),\ZZ_2)$ and the projection formula gives the equality
\[
\theta_\star\left(\theta^\star w_2(\breve{X}(B)) \frown [D]  \right) = w_2(\breve{X}(B)) \frown \theta_\star[D],
\]
taken with $\ZZ_2$ coefficients. We are able to compute the restriction of the second Steifel--Whitney class to $D$ as follows:  
\begin{align*}
\theta^\star w_2(\breve{X}(B)) &= w_2(T\breve{X}(B)|_D)\\
 &= w_2(TD \oplus \nu(D)) \\
 &= w_2(TD) + w_2(\nu(D))
\end{align*}
where $\nu(D)$ is the normal bundle of $D$. Moreover $w_2(TD) = 0$ since $D$ is diffeomorphic to a K$3$ surface and since $\nu(D)$ is a rank two bundle $w_2(\nu(D))$ is the mod $2$ reduction of its Euler class. Thus the left hand side of the projection formula reduces to the pushforward of the Poincar\'{e} dual to the Euler class of $D$ in $\breve{X}(\partial B)$. Since this is precisely the class $\PD(\theta_\star[D]) \frown \theta_\star[D]$, and $H^2(\breve{X}(B),\ZZ_2)$ is one-dimensional, this suffices to identify $w_2(\breve{X}(B))$ as the mod $2$ reduction of the Poincar\'{e} dual to $\theta_\star[D]$.

In fact, since the cohomology group $H^2(\breve{X}(B),\ZZ)$ is torsion free, the same argument works in the case of even Fano index after taking an integral lift of the class $w_2(\breve{X}(B))$. That is, in all such cases $w_2(\breve{X}(B)) = 0$.

\end{proof}

We can compute the first Pontryagin class in a similar way. First we compute the first Pontryagin class of a smooth, rank one, Fano threefold.

\begin{lem}
	Let $X$ be a smooth Fano threefold, then $p_1(X)c_1(X) = -K_X^3 - 48$.
\end{lem}
\begin{proof}
	By definition $p_1(X) := -c_2(TX \otimes_\RR \CC)$. By the Whitney sum formula for Chern classes we have that $p_1(X) = -2c_2(X) + c_1(X)^2$. Thus we have that
	\[
	p_1(X)c_1(X) = -2c_2(X)c_1(X) + c_1(X)^3.
	\]
	Since, by definition, $c_1(X)^3 = -K_X^3$ it suffices to compute $c_2(X)c_1(X)$. By Hirzebruch--Riemann--Roch and the fact that the holomorphic Euler characteristic of a Fano manifold is equal to one, we have that
	\[
	1 = \chi(X,\cO_X) = \langle td(X),[X] \rangle.
	\]
	The degree $6$ part of the Todd class is $c_1(X)c_2(X)/24$ and thus $p_1(X)c_1(X) = -K_X^3-48$.
\end{proof}

We can now prove the analogous statement to Proposition~\ref{pro:w2} for the first Pontryagin class of the manifold $\breve{X}(B)$.

\begin{pro}
	\label{pro:p1}
	Given an affine manifold $B$ model of a Fano threefold $X$ determined by degeneration data constructed using the method of \S\ref{sec:smooth_decompositions} such that $b_2(\breve{X}(B)) = 1$, we have that $p_1(\breve{X}(B))$ maps to $p_1(X)$ under the identification of $H^4(\breve{X}(B),\ZZ)$ with $H^4(X,\ZZ)$.
\end{pro}

\begin{proof}
	We use the same technique as in the computation of $w_2(\breve{X}(B))$, pulling back to $D := \pi^{-1}(\partial B)$, and splitting the tangent bundle. Although we expect $D$ to be in the class $c_1(\breve{X}(B))$ we do not use an almost complex structure on $\breve{X}(B)$; however by the computation of the index of $[\pi^{-1}(\partial B)]$ and its cube, the map
	\[
	 H^2(X,\ZZ) \rightarrow H^2(\breve{X}(B),\ZZ)
	\]
	defined by sending $[-K_X] \mapsto [\pi^{-1}(\partial B)]$ is a group isomorphism which identifies the respective cubic forms. Thus it suffices to prove that $p_1(\breve{X}(B)).[D] = [D]^3-48$. In fact, identifying $H_0(\breve{X}(B),\QQ)$ with $H^6(\breve{X}(B),\QQ)$, it suffices to compute $p_1(\breve{X}(B)) \frown \theta_\star[D]$. By the projection formula,
	\[
	p_1(\breve{X}(B)) \frown \theta_\star [D] = \theta_\star\left(\theta^\star(p_1(\breve{X}(B))) \frown [D]\right),
	\]
	and we have that $\theta^\star p_1(\breve{X}(B)) = p_1(D) + p_1(\nu(D))$. However, using the fact that $D$ is diffeomorphic to a K$3$ surface, $p_1(D) = -2c_2(D) + c_1(D)^2 = -2c_2(D) = -48$. Moreover $p_1(\nu(D))$ is the Euler class of $\nu(D) \oplus \nu(D)$, which is precisely $[D]^3$.
\end{proof} 

We are now in a position to apply Theorem~\ref{thm:top_classification}, and hence complete the proof of Theorem~\ref{thm:rank_one_models}.

%% file: examples.tex
% !TEX root = fano_manifold_topology.tex

%----------------------------------------------------------------------
\section{Examples}
\label{sec:examples}
%----------------------------------------------------------------------

In this section we present a number of sample calculations of the numerical invariants of Fano manifolds from degeneration data on a polytope.

\subsection{$V_{12}$}
\label{sec:V12}

The entry for the family of Fano manifolds $V_{12}$ in Appendix~\ref{sec:tables} suggests we consider degeneration data on a polytope $P$ with PALP ID $3874$, using the method described in~\S\ref{sec:smooth_decompositions}. That is, we consider smooth Minkowksi decompositions of each of the facets and take $\Sigma$ to be the normal fan of the Fano polytope $P$.

\begin{figure}
	\label{fig:V12}
	\includegraphics{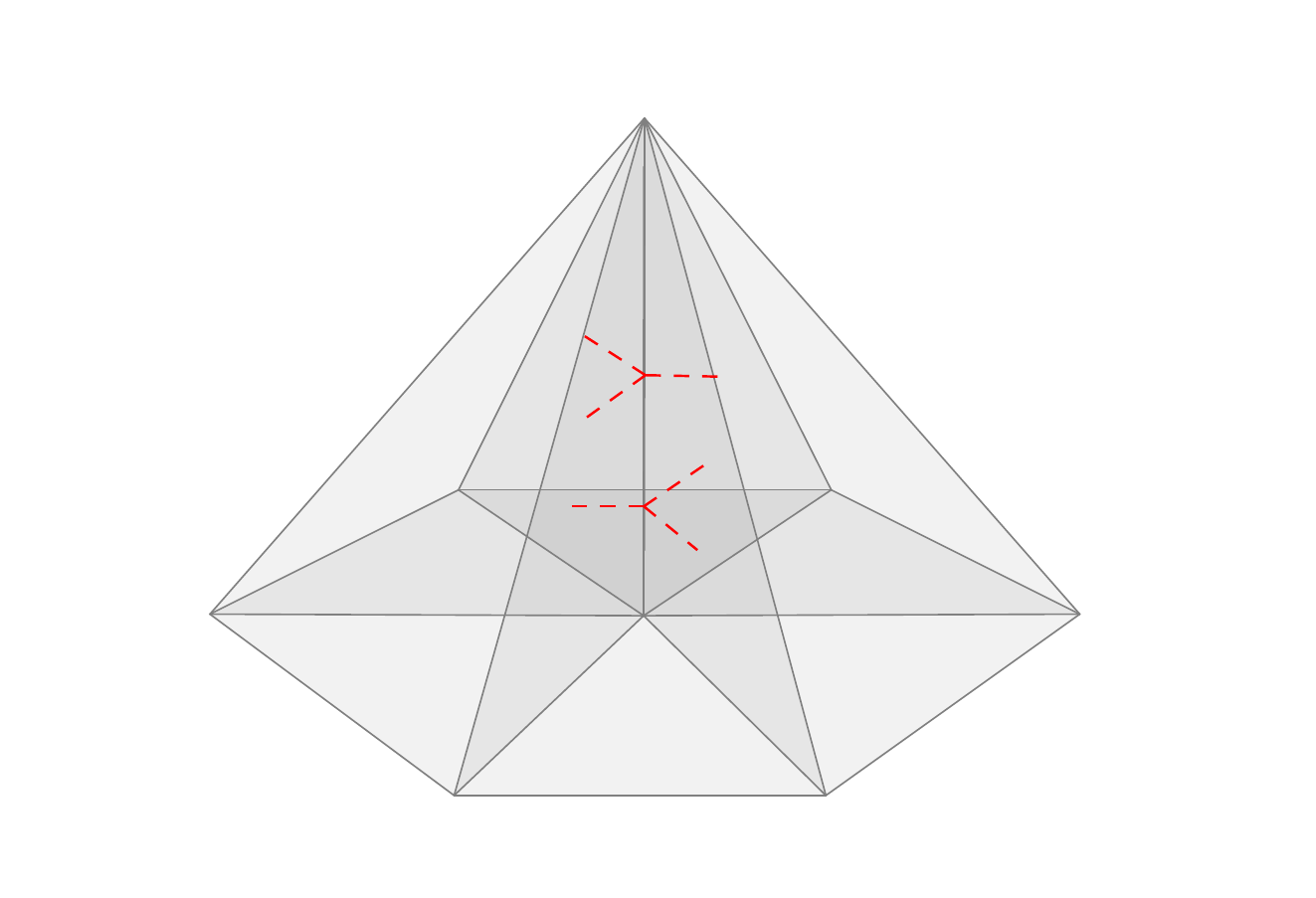}
	\caption{Part of the affine manifold $B^1_{12}$}
\end{figure}

In this case all facets are either rectangular (and hence have a unique smooth Minkowksi decomposition), or hexagonal, in which case there is a choice of Minkowski decomposition shown in Figure~\ref{fig:dP6}. The choice of Minkowksi decomposition changes the homotopy type of the total space of the associated torus fibration we obtain, and indeed the manifolds we obtain are models for different Fano manifolds. Following~\cite{CCGK}, and the data on \url{www.fanosearch.net}, we expect the following correspondence:

\begin{enumerate}
	\item Decomposing one hexagonal facet in each way models the Fano manifold $V_{12}$.
	\item Decomposing both hexagonal facets into line segments models the Fano manifold $\MM{2}{6}$.
	\item Decomposing both hexagonal facets into triangles models the Fano manifold $\MM{3}{1}$.
\end{enumerate}

Let $B^i_{12}$,~$i \in \{1,2,3\}$ be the affine manifolds constructed from these choices respectively. We will show that the manifolds $\breve{X}(B^i_{12})$ have $b_2(\breve{X}(B^i_{12})) = 1$,~$2$, and $3$ respectively. A part of $B^1_{12}$ is shown in Figure~\ref{fig:V12}, which shows the singular locus near a segment $\rho$ contained in the ray normal to a hexagonal face of $P$, in the case that $J(\rho)$ is the decomposition of the hexagon into a pair of triangles. Recall that each affine manifold $B^i_{12}$ is constructed from degeneration data $(\Sigma,C,J)$, where $\Sigma$ is the normal fan of $P$, $C$ maps each edge of $P^\circ$ to the length of the dual edge of $P$, and $J$ is determined by the choice of Minkowski decompositions.

We use Theorem~\ref{thm:picard_rank} to calculate $H^1(\breve{X}(B^i_{12}),R^1\xi_\star\ZZ)$, and hence $b_2(\breve{X}(B))$, in terms of the space $\Gamma(\Sigma,J)$. After choosing bases for the one-dimensional vector spaces $M/\langle \sigma \rangle$ an element of $\Gamma(\Sigma,J)$ is an element of $\QQ^{\Sigma(2)}$ meeting certain compatibility conditions along the rays of $\Sigma$. Let the section associated to each vector space be denoted $\alpha_i$,~$\beta_i$ and $\gamma_i$ for $i \in  \{1, \ldots ,6\}$ as shown in Figure~\ref{fig:V12a}.

\begin{figure}
	\includegraphics[scale=1.2]{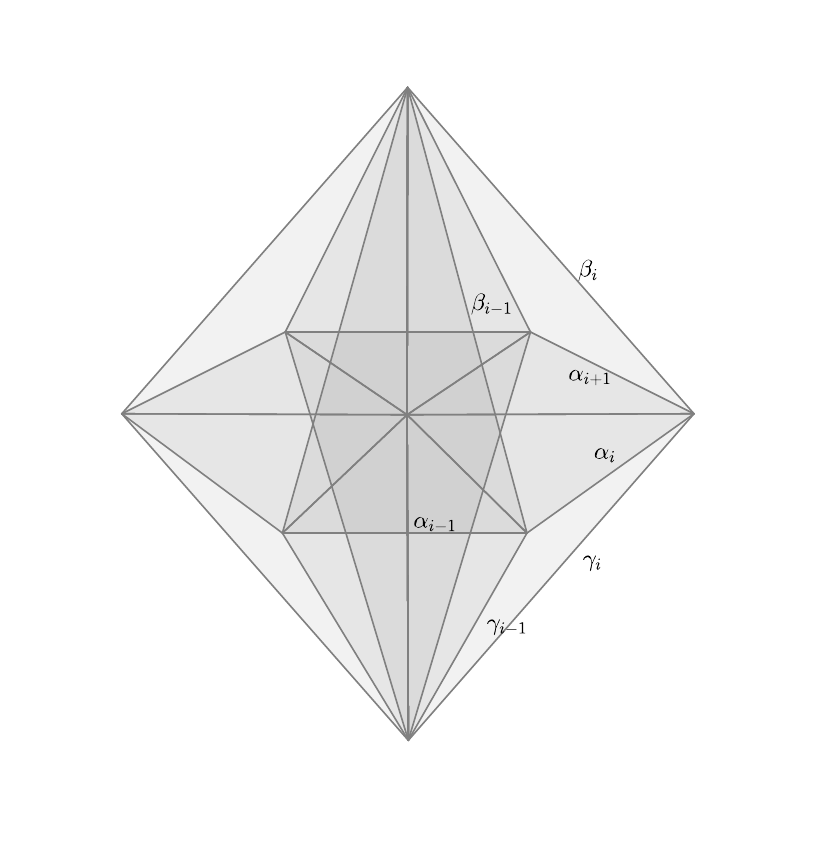}
	\caption{Labelling the one-skeleton of $P^\circ$}
	\label{fig:V12a}
\end{figure}

Following the proof of Theorem~\ref{thm:picard_rank} the condition that a tuple is contained in $\Gamma(\Sigma,J)$ imposes a linear condition along every ray of $\Sigma$ (the normal fan to $P$), depending on the choice of Minkowski decomposition.
\begin{itemize}
	\item The rays normal to the rectangular faces of $P$ give rise to conditions
	\begin{align*}
		\alpha_i + \alpha_{i+1} = 0 && \beta_i = -\gamma_i
	\end{align*}
	 for all $i \in \{1, \ldots , 6\}$, with indices regarded cyclically.
	\item The ray normal to a hexagonal face, without loss of generality we assume this to be the facet with (dual) edges labelled with $\beta_i$, decomposed into triangles gives conditions 
	\begin{align*}
		\beta_1+\beta_3+\beta_5 = 0 && \text{and} && \beta_2+\beta_4+\beta_6=0
	\end{align*}
	\item The ray normal to a hexagonal face, without loss of generality we assume this to be the facet with (dual) edges labelled with $\gamma_i$, decomposed into line segments gives conditions 
		\begin{align*}
		\gamma_1+\gamma_4 = 0 && \gamma_2+\gamma_5 = 0 && \gamma_3+\gamma_6=0
	\end{align*}
\end{itemize}

Imposing these conditions for $B_1$ (hexagonal facets decomposed in different ways) we eliminate the $\gamma_i$ using the $\beta_i$, and eliminate $\alpha_i$ for $i \neq 1$ using $\alpha_1$. Imposing the conditions from the facet decomposed into line segments we eliminate $\beta_4$,~$\beta_5$ and $\beta_6$ using $\beta_1$,~$\beta_2$ and $\beta_3$. Imposing the conditions from the facet decomposed into triangles we eliminate $\beta_2$, writing $\beta_2 = \beta_1+\beta_3$, given such a section all conditions are satisfied and we conclude that $\dim \Gamma(\Sigma,J) = 3$, that is, $b_2(\breve{X}(B_1))=1$.

Following a similar procedure the second Betti numbers are easy to compute in the other two cases. The key observation is that in the other two cases the facets impose the same conditions on the sections $\beta_i$ after eliminating the $\gamma_i$.

\begin{rem}
	\label{rem:normalise}
	It is always the case that $\dim \Gamma(\Sigma,J) \geq 3$, since sections coming from the first cohomology group of $T^3$ define linearly independent elements of $\Gamma(\Sigma,J)$. The preceding computation can therefore be simplified by normalising with respect to this $T^3$ action, allowing us to, for example, assume that $\alpha_1=\beta_1 = \beta_2 = 0$. Making these identifications we easily obtain spaces of solutions for the values of $\alpha_i$,~$\beta_i$ and $\gamma_i$ for $i \in \{1, \ldots, 6\}$ of dimensions $0$,~$1$ and $2$ respectively: the dimensions of $H^1(\breve{X}(B^i_{12}),R^1\xi_\star\ZZ)$, or equivalently, the numbers $b_2-1$. 
\end{rem}

\subsection{$V_{16}$}
\label{sec:V16}

Let $P$ be the reflexive polytope with PALP ID $3031$. The one-skeleton of $P^\circ$ is shown in Figure~\ref{fig:V16}.
\begin{figure}
	\includegraphics[scale = 0.8]{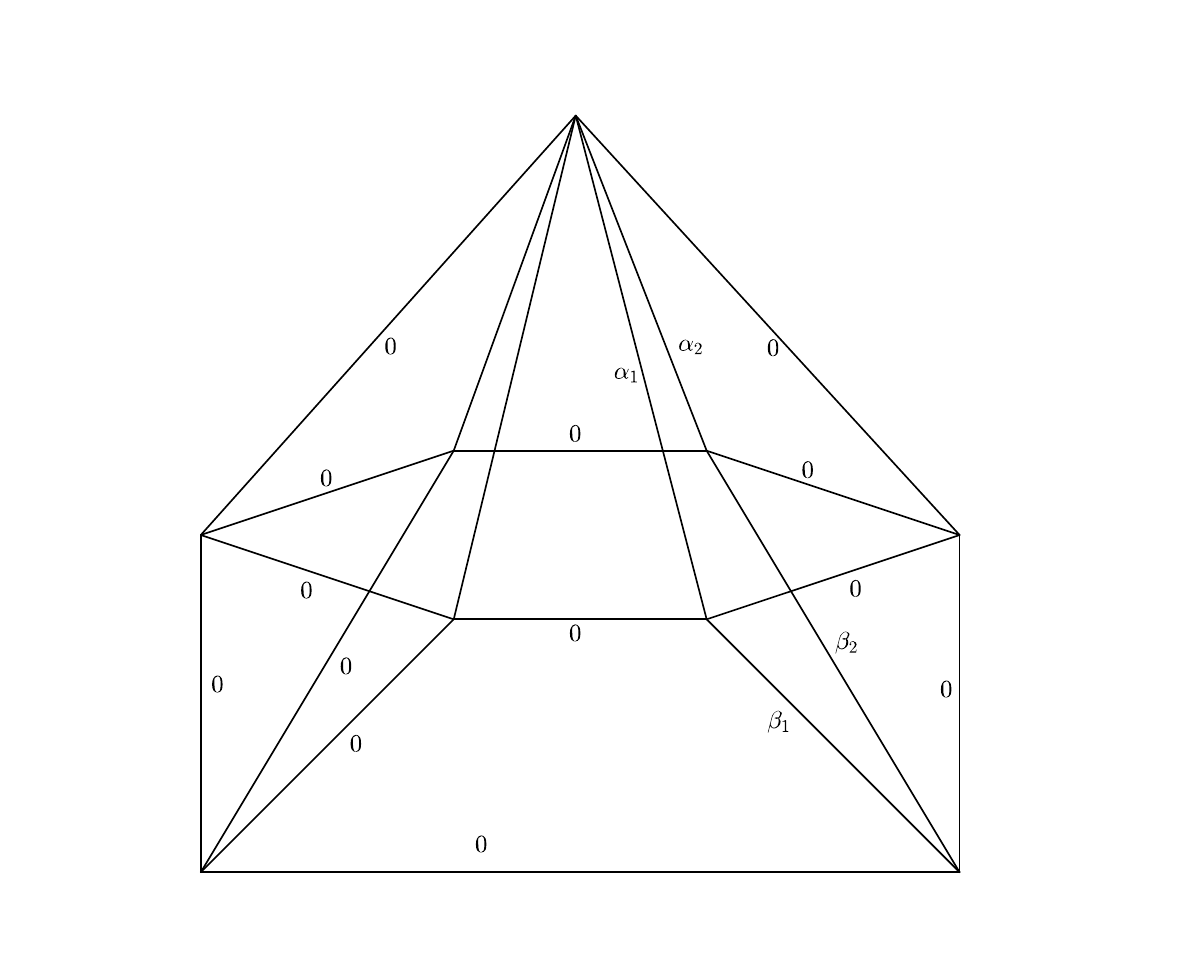}
	\caption{Computing the space $\Gamma(\Sigma,J)$ for $V_{16}$}
	\label{fig:V16}
\end{figure}
Again, there is a hexagonal face, which admits a pair of Minkowksi decompositions. Recalling from Remark~\ref{rem:normalise} that three independent variables can be fixed by choosing a suitable element of $H^1(T^3,\ZZ) \cong \ZZ^3$ we can reduce the possible remaining variables to those shown in Figure~\ref{fig:V16}.

In the case the hexagon is decomposed into a pair of triangles the only relation between $\alpha_1$ and $\alpha_2$ is that $\alpha_1+\alpha_2=0$ (choosing orientations appropriately), thus we obtain a one dimensional subspace in $\Gamma(\Sigma,J)$. The numerical invariants of this manifold coincide with those of $\MM{2}{10}$ (as predicted by \cite{CCGK}). In the case the hexagon is decomposed into three line segments we are forced to impose that $\alpha_1=\alpha_2=0$ and thus there are no non-trivial sections, that is, for this affine manifold $B$, $b_2(\breve{X}(B)) = 1$.

\subsection{$V_{22}$}
\label{sec:V22}

Considering the polytope $P$ with PALP ID $1886$ we see that each facet has a unique Minkowski decomposition and the hypotheses of Lemma~\ref{lem:collapsing} apply, that is, 
\[
\Gamma(\Sigma,J) \cong H^0(P^\circ[1],\bar{T}^\bot).
\]
This is a typical situation, and we include this example to show that even a rather complicated Fano threefold, such as $V_{22}$, can be easily (topologically) reproduced using these methods.

\begin{figure}
	\includegraphics[scale = 0.8]{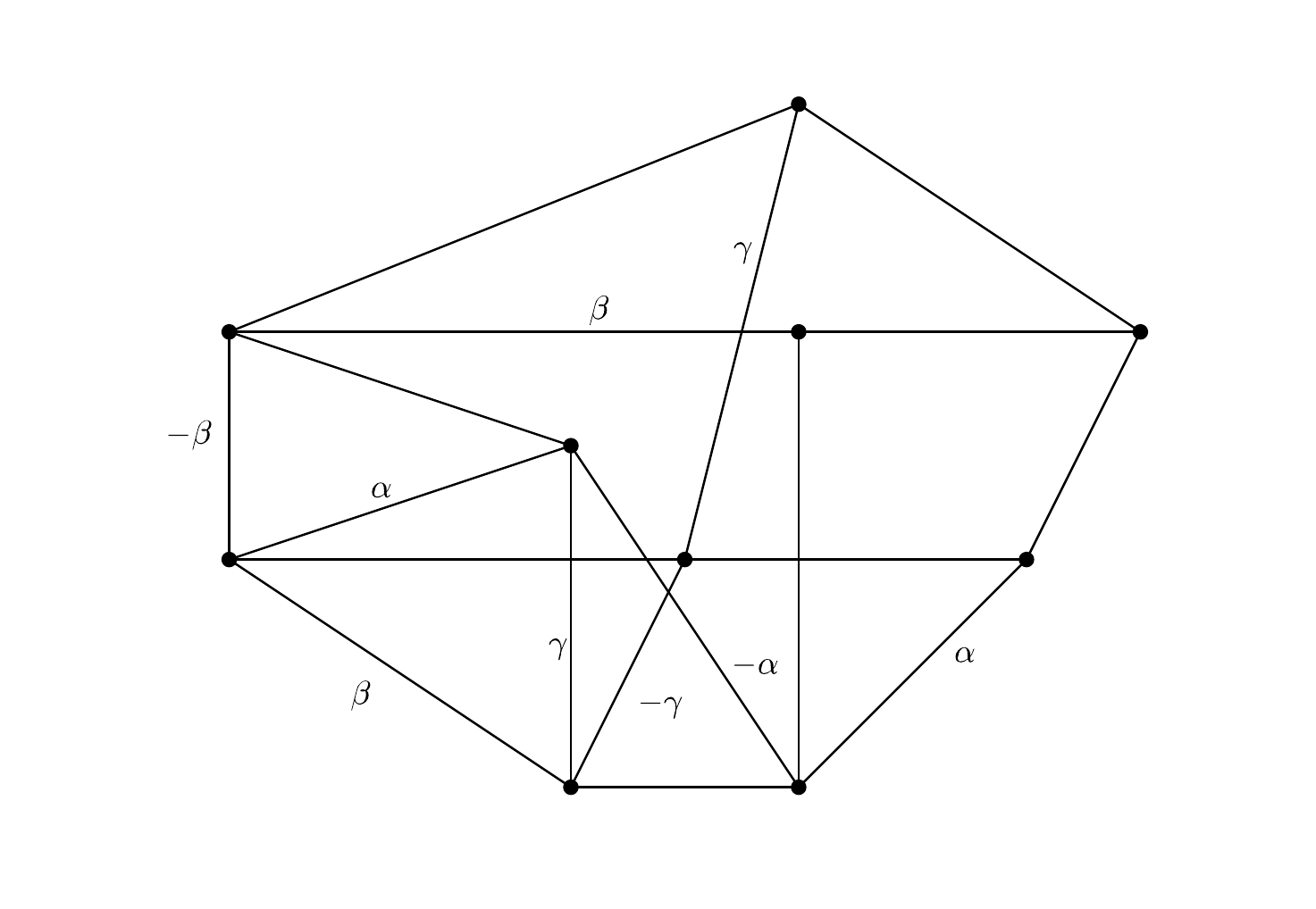}
	\caption{Computing $\Gamma(\Sigma,J)$ for $V_{22}$}
	\label{fig:V22}
\end{figure}

Figure~\ref{fig:V22} shows a one-dimensional representation of the one-skeleton of $P^\circ$. Relations of the form $\alpha_i + \alpha_j = 0$ for local sections $\alpha_i$,~$\alpha_j$ reduce the number of sections, some examples of which are shown on Figure~\ref{fig:V22}. Since we are free to identify $3$ independent variables to zero we set $\alpha = \beta = \gamma = 0$. For any three-valent vertex the corresponding relation is that the sum of the three neighbouring sections is zero. For any four-valent vertex the relations imply that if three sections of slabs neighbouring $\rho$ vanish, the other one must also vanish. These relations are enough to see that $\dim\Gamma(\Sigma,J)=3$, and thus that $b_2(\breve{X}(B)) = 1$ in this example.

\subsection{$\MM{2}{11}$}
\label{sec:MM211}

Let $P$ be the reflexive polytope with PALP ID $3008$. In Figure~\ref{fig:2-11} we show the one-skeleton of $P^\circ$ together with a one dimensional subspace of $\Gamma(\Sigma,J)$ which does not lie in the three dimensional space given by the first homology of the three dimensional torus. In fact it is easy to see that $b_2(\breve{X}(B)) = 2$ where $B$ is the affine manifold obtained by choosing the unique smooth Minkowski decompositions and applying the procedure described in~\S\ref{sec:smooth_decompositions}.

\begin{figure}
	\includegraphics[scale = 0.8]{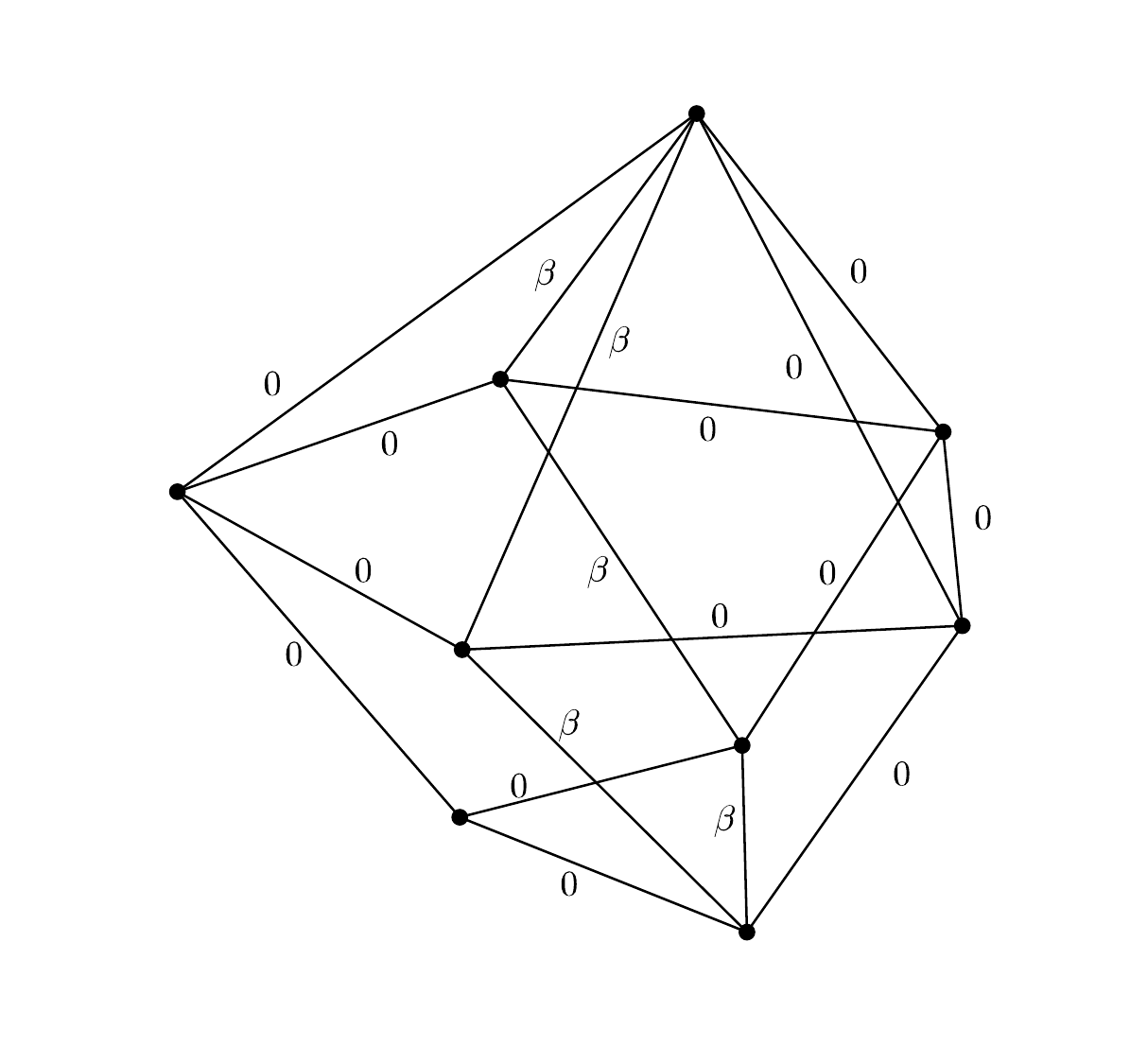}
	\caption{A one-dimensional subspace of $\Gamma(\Sigma,J)$ for $\MM{2}{11}$.}
	\label{fig:2-11}
\end{figure}

This example shows an important subtlety of the algorithm used to determine the second Betti number: In previous examples we have been able to choose orientations compatible with the cyclic ordering of the edges around a vertex. However in this case we have an odd length cycle of edges, each assigned the value $\beta$ (or $-\beta$). In this case we choose the orientations of these edges so that the signs of each $\beta$ is the same, and let the other orientations be arbitrary.

\subsection{$\MM{2}{32}$ and $\MM{3}{27}$}
\label{sec:MM232}

Let $P$ be the reflexive polytope with PALP ID $155$. The polytope $P$, as well as its polar $P^\circ$ is a cone over a hexagon. As usual there are two choices of smooth Minkowski decomposition of the hexagonal facet $F$ of $P$, which give models of varieties with different ranks (in this example). Figure~\ref{fig:2-32} shows an example of a one dimensional space of non-trivial sections in $\Gamma(\Sigma,J)$, in the case the Minkowski decomposition of the hexagon into three lines is chosen.

\begin{figure}
	\includegraphics{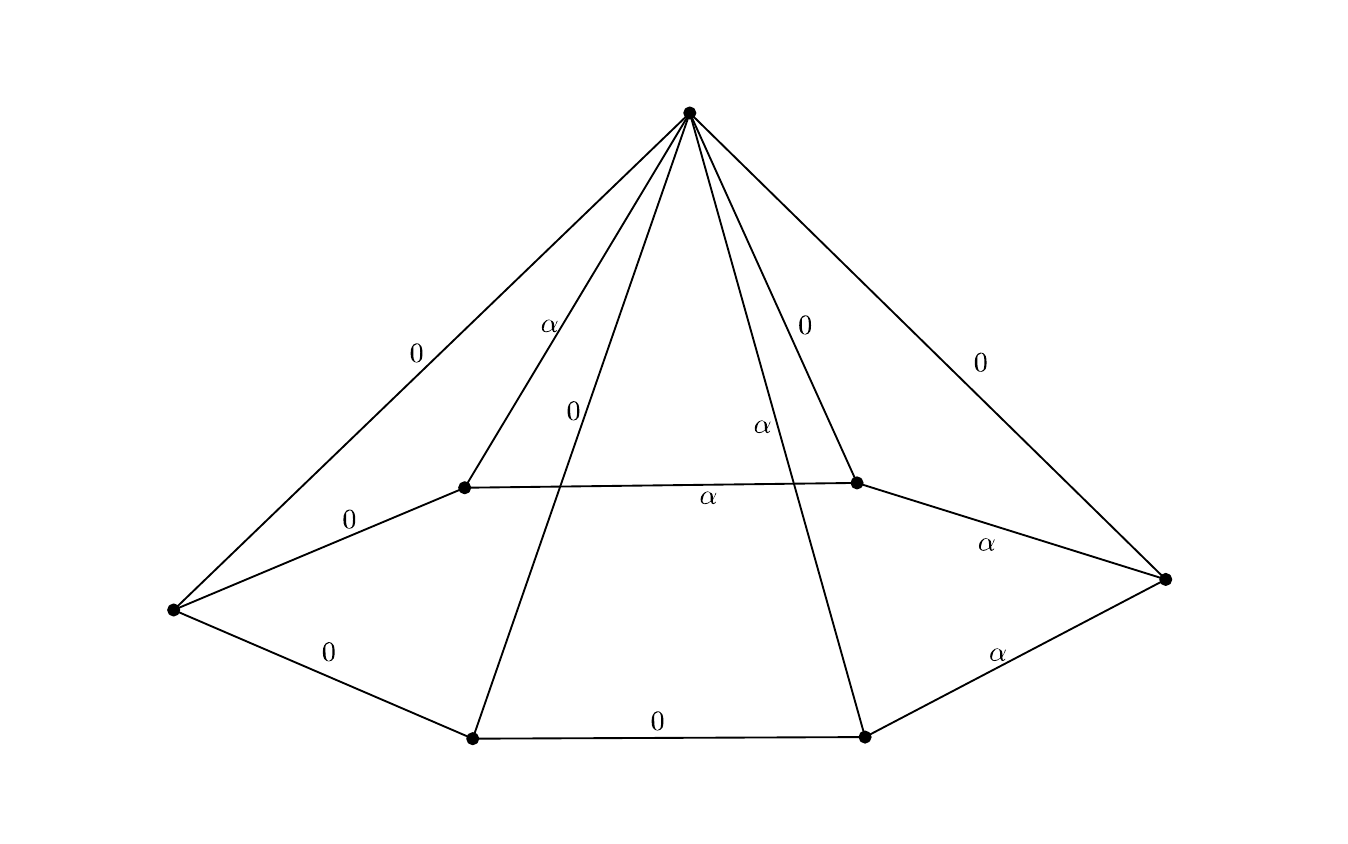}
	\caption{A one-dimensional subspace of $\Gamma(\Sigma,J)$ for $\MM{2}{32}$.}
	\label{fig:2-32}
\end{figure}

For either choice of Minkowksi decomposition we have $12$ slabs $\fs = (c,D)$ such that $X_c \cong \PP(1,1,2)$, and $\cO_{X_c}(D) = \cO_{\PP(1,1,2)}(2)$. Hence there are $24$ negative nodes in the integral affine manifold $B$ in each case. Moreover, there are $6$ positive nodes in $B$ if the Minkowski decomposition of $F$ into three line segments is chosen; and $8$ if $F$ is decomposed into a pair of triangles.

%% file: outstanding.tex
% !TEX root = fano_manifold_topology.tex

%----------------------------------------------------------------------
\section{Finding the outstanding invariants}
\label{sec:method_2_calculations}
%----------------------------------------------------------------------

We have now described how to compute invariants for compactified torus fibrations obtained from Construction~\ref{cons:smooth_decompositions}. We tabulate constructions of manifolds whose invariants match each of the $105$ families of Fano threefolds in Appendix~\ref{sec:tables}. We use Construction~\ref{cons:smooth_decompositions}, applied to the polytope specified in Appendix~\ref{sec:tables}, except in two cases:
\begin{enumerate}
	\item The eleven families (labelled with `Method $2$' in Appendix~\ref{sec:tables}) which we treat in this section.
	\item The five families Fano threefolds which are products of non-toric del~Pezzo surfaces with $\PP^1$.
\end{enumerate}

We treat the five product cases using Construction~\ref{cons:products}, and we do not explain these in more detail in this section. Note that we could also use Construction~\ref{cons:products} to find torus fibrations on the products of the smooth toric varieties with $\PP^1$, but this is unnecessary, since the polytopes corresponding to these smooth toric varieties are possible input to Construction~\ref{cons:smooth_decompositions}. We further note that the $89$ cases we can treat with Construction~\ref{cons:smooth_decompositions} correspond to families of Fano threefolds with very ample anti-canonical bundle; and three of the five products of non-toric del Pezzo surfaces with $\PP^1$ have very ample anti-canonical bundle.

Each of these constructions which appears in this section is based on the method described in \S\ref{sec:complete_intersections}, and while we describe the affine manifold in each case we do not describe how each Fano variety appears as a toric complete intersection. These complete intersection models are described in~\cite{CCGK}, and further details on the method of Laurent inversion can be found in~\cite{CKP17}.

\subsection{$V_2$}
\label{sec:V2}

In this case the method described in \S\ref{sec:complete_intersections} coincides with that described in~\S\ref{sec:smooth_decompositions} applied to a non-reflexive polytope, so we only present the degeneration data used to form $B$, and refer to the method used in \S{\ref{sec:betti_numbers}} to calculate the Betti numbers of $\breve{X}(B)$.

Consider the (non-reflexive) simplex
\[
P := \conv{(-1,-1,-1),(5,-1,-1),(-1,5,-1),(-1,-1,5)},
\]
the polar polytope $P^\circ$ is the convex hull of the standard basis elements $\{e_1,\ldots,e_3\}$ together with the point $\frac{1}{3}(-1,-1,-1)$. To define degeneration data for $P$, fix the following data:

\begin{enumerate}
	\item Let $\Sigma$ the the normal fan of $P$, that is the fan defining $\PP^3$.
	\item Let $C$ be determined by labelling edges of $P^\circ$ as follows,
	\begin{align*}
		[e_i,e_j] \mapsto 6, \text{ for, } i,j \in \{1,2,3\},  i \neq j \\
		[e_i,\frac{1}{3}(-1,-1,-1)] \mapsto 6 \text{ for, } i \in \{1,2,3\} \\
	\end{align*}
	We check that this defines a collection of nef line bundles on the slabs defined by intersecting $P^\circ$ with $\Sigma$.
	\item For facets dual to the vertices $e_i \in M$ of $P^\circ$, define $J$ for the corresponding ray of $\Sigma$ to be the usual factorization of the facets of $P$ into standard triangles. Define $J(\rho)$, for $\rho$ the remaining ray in $\Sigma$, generated by $v = \frac{1}{3}(-1,-1,-1)$, to be the factorization of the dual facet $(1/3 \cdot v^\star)$ into two standard triangles.
\end{enumerate}

The designation $[e_i,\frac{1}{3}(-1,-1,-1)] \mapsto 6$ may seem unexpected when compared with earlier examples, and we briefly explain it. The slabs $(c,D)$ containing the edges $E_i = [e_i,\frac{1}{3}(-1,-1,-1)]$ are associated to toric varieties isomorphic to $\PP(1,1,3)$. However, unlike the affine manifolds obtained from Construction~\ref{cons:smooth_decompositions}, the edge $E_i$ corresponds to a section of $\cO(1)$ (not $\cO(3)$) on $\PP(1,1,3)$. Such data is compatible with the ray data since, if $\rho$ is the ray of $\Sigma$ passing through $(-1,-1,-1)$, $v^\star \cong 3\cdot P_{L_\rho}$.

Since $P$ is not reflexive we cannot apply the arguments given in \S\ref{sec:degree} to compute $[\pi^{-1}(\partial B)]^3$ directly. However following Remark~\ref{rem:non_reflexive} we can dilate $P^\circ$ by a factor of $3$. Indeed, there are $11$ integral points on the boundary of $(3\cdot P^\circ)$, and hence its boundary has area $18 = 3^2\cdot 2$. That is, the toric variety $X_P$ has anti-canonical degree $2$, as required.

\subsection{$B_1$}
\label{sec:B1}

\begin{figure}
	\includegraphics[scale=1.2]{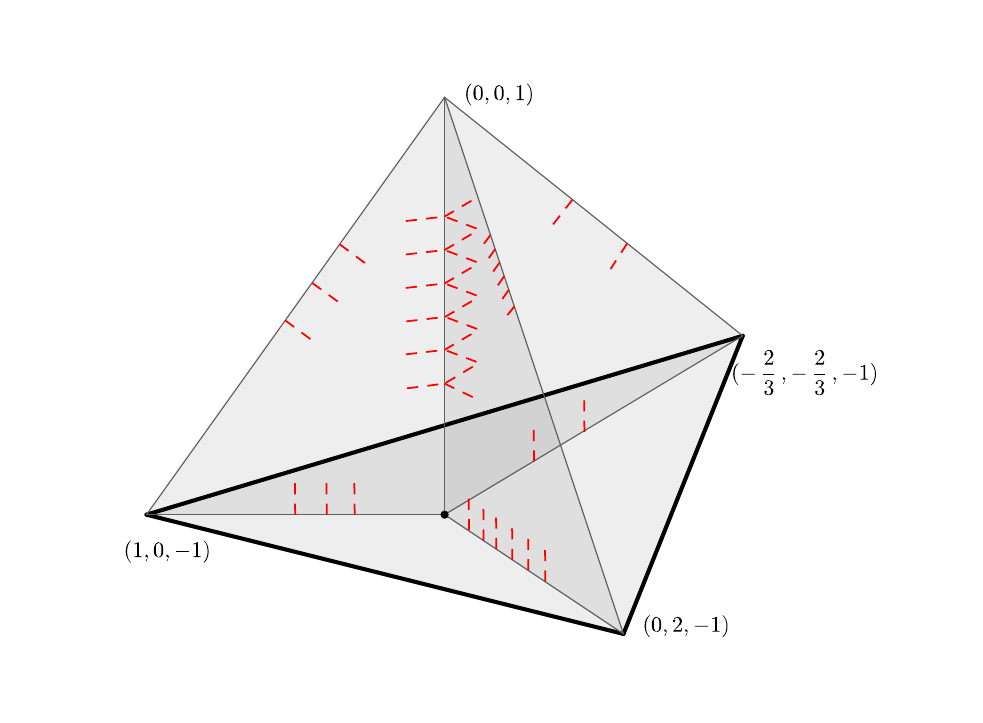}
	\caption{Affine manifold model of $B_1$}
	\label{fig:B1}
\end{figure}

Members of the family $B_1$ are sextics in $\PP(1,1,1,2,3)$, in particular such varieties do not have very ample anti-canonical bundle. Consider the polytope
\[
P := \conv{(0,0,1),(-1,-1,-1),(-1,5,-1),(5,-1,-1)}.
\]
Taking the toric variety associated to the spanning fan of $P$ we obtain the variety 
\[
\{x^6_0 = x_2yz\} \subset \PP(1,1,1,2,3)_{x_0,x_1,x_2,y,z}
\]
We construct an affine manifold $B$ illustrated in Figure~\ref{fig:B1}. We specify degeneration data; first fixing the fan $\Sigma$ with one-dimensional minimal cone $L$, generated by $(0,0,1)$ and three two-dimensional cones, generated by $L$ and $(1,0,0)$, $(0,1,0)$, or $(-1,-1,0)$ respectively. Recall that -- as in \S\ref{sec:smoothing_polytope} -- we do not assume that all cones in $\Sigma$ are strictly convex. We now fix degeneration data by specifying ray and edge data.
\begin{enumerate}
	\item Edge data $C$: Let $C$ be the torus invariant curve assigning the coefficient $6$ to each edge of $P^\circ$ meeting $(0,0,1)$ (and assigning zero to all other edges of $P^\circ$).
	\item Ray data $J$: Let $\rho$ be the ray generated by $v := (0,0,1)$. Set $J(\rho)$ to be the multiset of six standard triangles -- the six Minkowski summands of $v^\star$.
\end{enumerate}

Note that the boundary of $B$ is not a sphere with $24$ focus-focus singularities (that is, the base of a smooth K$3$), but has two components, each of which is a disc containing $11$ points. In other words, the boundary of $B$ is the base of a Lagrangian fibration on a pair of del~Pezzo surfaces of degree $1$, meeting along a genus one curve. The boundary circle of these two affine discs is $(\partial B)_1$, and is marked in bold on Figure~\ref{fig:B1}. In later examples we will continue to indicate $(\partial B)_1$ with bold edges. 

To compute $[\pi^{-1}(\breve{X}(B))]^3$ we observe that the cycle $\pi^{-1}(\breve{X}(B))$ is the sum of two submanifolds, $Y_1$ and $Y_2$. Moreover we can find homeomorphic tubular neighbourhoods of $Y_1$ and $Y_2$ mapping $Y_1$, and hence $[Y_1]^3 = [Y_2]^3$. Since $H_4(\breve{X}(B)) \cong \ZZ$ we must have that $Y_1$ and $Y_2$ are homologous submanifolds. Thus, $[\pi^{-1}(\breve{X}(B))]^3 = 8 \times [Y_1]^3$. However since $Y_1$ and $Y_2$ are homologous we see that $Y_1^2$ is represented by the genus one curve lying over the boundary of $\pi(Y_1)$. Since $\pi^{-1}(Y_1)$ is diffeomorphic to a del Pezzo surface of degree one, identifying $\pi^{-1}(\partial \pi(Y_1))$ with an anti-canonical section we have that $[Y_1]\cdot [Y_1]^2 = 1$. That is, $[\pi^{-1}(\breve{X}(B))]^3 = 8$.

We defer the computation of the Betti numbers to the next example -- the family $\MM{2}{1}$ -- which it essentially duplicates; noting that $H^2(\breve{X}_0(B))$ is isomorphic to $\ZZ$ in this example, and isomorphic to $\ZZ^2$ in the next (rank $2$) example.

\subsection{$\MM{2}{1}$}
\label{sec:MM21}

The Fano manifold $\MM{2}{1}$ is obtained by blowing up a threefold in the family $B_1$ in an elliptic curve which is the intersection of two elements of $|-\frac{1}{2}K_{B_1}|$. In \cite{CCGK} the authors observe that a threefold $\MM{2}{1}$ can be given as a divisor of bidegree $(1,1)$ in $\PP^1\times B_1$; since $B_1$ itself is given by a sextic in $\PP(1,1,1,2,3)$. Let $x_0$,$x_1$,$x_2$,$y$,$z$ denote the coordinates on $\PP(1,1,1,2,3)$, and $u_0$,$u_1$ denote those on $\PP^1$. We have a toric degeneration of a Fano manifold $X$ belonging to the family $\MM{2}{1}$ to the toric variety $X_P$ defined by the equations

\begin{align*}
x_2yz = tx^6_0 && \textrm{and} &&	x_1u_1 = x_0u_0,
\end{align*}

where $t$ is a complex parameter, in $\PP(1,1,1,2,3)\times \PP^1$. The degeneration in $t$ near $t = 0$ degenerates this toric variety into a union of three toric varieties, which we can use to define an affine structure. The affine manifold $B$ obtained by this process is shown in Figure~\ref{fig:MM21}. Note that, for clarity, we do not draw all the singular locus contained in each slab, but only the intersections with each of the edges in the decomposition of $P^\circ$. To construct $B$ carefully we first describe the slabs appearing in $\breve{X}_0(B)$. These are formed by the intersection of $P^\circ$ with two dimensional cones of $\Sigma$; the product of the fan determined by $\PP^2$ with $\RR$ (see Figure~\ref{fig:MM21}). The toric surfaces associated to these polygons are $S_1 \cong \FF_1$, $S_2 \cong \FF_2$, and $S_3 \cong \FF_3$. 
 
Note that two of the three vertical edges $E$ shown in Figure~\ref{fig:MM21} violate the assumption that $r(E^\star) = 1$; indeed one such edge determines a toric singularity with Gorenstein index $2$, the other with Gorenstein index $3$. While this changes the conditions required for the ray and edge data to be compatible and smooth, it does not fundamentally alter the construction, and we define ray and edge data in this setting as follows:
\begin{enumerate}
	\item Ray data $J$: there are two rays $\rho^+$, $\rho^-$ in $\Sigma^+(1)$, which we label such that $\rho^+$ contains a vertex $v$ of $P^\circ$; the facet $v^\star$ admits a Minkowski decomposition into $6$ standard triangles, and hence we take $J(\rho^+)$ to be a multiset containing $6$ copies of the $\cO_{\PP^2}(1)$. We set $J(\rho^-) := \{0\}$.
	\item Edge data $C$: we label edges $P^\circ$ contained in a two-dimensional cone of $\Sigma$ by setting $E \mapsto 6/r(E^\star)$ if $E$ is an edge contained in a two-dimensional cone of $\Sigma$.
\end{enumerate}

Slabs are defined as usual, and specifying the divisors on torus surfaces $S_i$ for $i \in \{1,2,3\}$ as before, we obtain divisors $D_i$ on $S_i$ which are vanishing loci of sections $\pi_i^\star \cO(6)$ where $\pi_i \colon \FF_i \to \PP(1,1,i)$ is the usual contraction. Note that these line bundles are all nef and we can define a singular locus $\Delta$ as in Construction~\ref{cons:slabs_to_affine_structure}. Note that, since no vertex of $P^\circ$ is contained in both a ray of $\Sigma$, and an edge $E$ such that $r(E^\star) > 1$, we can define ray data, and compatibility of ray and edge data as above.

Hence we may verify the usual compatibility between $C$ and $J$. Note that there is a unique ray $\rho$ such that $J(\rho)$ is non-trivial. The toric variety $X_\rho$ is isomorphic to $\PP^2$, and $L_\rho$ is $-2K_{X_\rho}$. We verify that the pullback to any boundary line has degree $6$, and hence the ray and edge data are compatible and $J$ is smooth. The line bundle defined by the edge data on each slab is equal to $\pi^\star\cO_{\PP(1,1,i)}(6)$, where $\pi\colon \FF_i \to \PP(1,1,i)$ is the usual contraction.  Convexity is satisfied since -- considering the vertex $v \in \V{P^\circ}$ contained in $\rho^+$ -- $v^\star = r(v^\star) P_{L_\rho} = P_{L_\rho}$, up to an integral affine transformation.

The induced affine structure on $P^\circ$ as $(\partial B)_0 = \varnothing$, while $(\partial B)_1$ is equal to a pair of disjoint circles (consisting of the `horizontal' edges in Figure~\ref{fig:MM21}).

\begin{rem}
	The horizontal triangle (which is not part of the decomposition of $X_P$) is a homeomorphic to a disc and indicates a second possible degeneration of $\breve{X}(B)$ in which one component is a product of a del Pezzo surface with $\PP^1$. In fact we can see that the cylinder that forms the boundary of this `neck' is the base of a torus fibration on a $\PP^1$ bundle on a genus one curve, and contracting this we recover a topological version of the construction of $\MM{2}{1}$ as the blow up of $B_1$ with centre an elliptic curve. It would be interesting to realise other extremal contractions of Fano threefolds topologically in this way, following, for example, the constructions given in~\cite{AAK16}. In fact we remark that this observation already guarantees that $\breve{X}(B)$ is homeomorphic to $\MM{2}{1}$ and we thank Paul Hacking for this remark.
\end{rem}

In fact we can follow the argument of \S{\ref{sec:betti_numbers}} to compute the Betti numbers of $\breve{X}(B)$. The map $\xi$ defined in Appendix~\ref{sec:contraction} is defined for any affine structure, and we consider the Leray spectral sequence for $\xi$. The arguments used in \S{\ref{sec:betti_numbers}} show that

\begin{align*}
 H^0(\breve{X}_0(B),R^2\xi_\star\ZZ) = 0 && \textrm{and} &&	 H^0(\breve{X}_0(B),R^1\xi_\star\ZZ) = 0
\end{align*}

Moreover no fibre of $\xi$ is a three-dimensional torus and, defining $\cF$ as the cokernel of
\[
\cF := \coker(R^1\xi_\star\ZZ \rightarrow {i_1}_\star i_1^\star R^1\xi_\star\ZZ),
\]
we see that $ {i_2}_\star i_2^\star \cF = \cF$. Considering the stalks of $\cF$ along the projective line it is supported on we see that $\cF = \ZZ$ away from the six positive vertices. Hence $H^0(\breve{X}_0(B),\cF) = 0$, and $H^2(\breve{X}(B),\ZZ) = H^2(\breve{X}_0(B),\ZZ) = \ZZ^2$. Note that -- as in \S\ref{sec:betti_numbers} -- we have that $H^1(\breve{X}_0(B),{i_1}_\star i_1^\star R^1\xi_\star\QQ) = 0$, since $H^1_c$ on the complement of a curve in a (complex) projective surface vanishes. This depends on the fact no boundary component of a slab supporting a non-trivial discriminant locus is contained in $(\partial B)_1$. In later examples this fails to be the case, and we will require a more detailed analysis of ${i_1}_\star i_1^\star R^1\xi_\star\QQ$; see \S\ref{sec:MM32}.

\begin{figure}
	\includegraphics[scale=1.2]{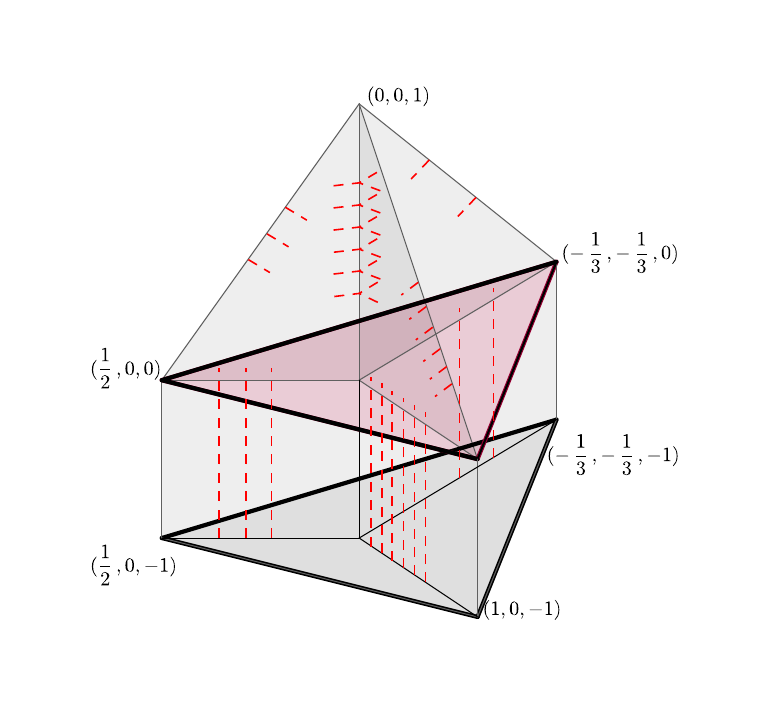}
	\caption{Affine manifold model of $\MM{2}{1}$}
	\label{fig:MM21}
\end{figure}

\subsection{$\MM{2}{2}$}
\label{sec:MM22}

Let $X$ be Fano manifold in the family $\MM{2}{2}$. We use the description of $X$ as toric hypersurface given in \cite{CCGK}. In particular $X$ is a divisor in a $\PP^1$ bundle over $\PP(1,1,1,2)$. The affine manifold $B$ obtained by this construction is shown in Figure~\ref{fig:2-2}. Computing the Euler number of $\breve{X}(B)$ we first note that the slab functions are sections of the following line bundles
\begin{itemize}
	\item A single $\PP^2$ slab, with line bundle $\cO(4)$.
	\item A pair of $\PP(1,1,2)$ slabs, with line bundles $\cO(2)$.
	\item A pair of $\PP^1\times \PP^1$ slabs, with line bundles $\cO(4,2)$.
	\item A single $\FF_1$ slab with line bundle pulled back from $\cO_{\PP^2}(1)$.
\end{itemize}

Summing the number of negative nodes $n$, we obtain $n = 16 + 2\times 8 + 2 \times 16 + 4 = 68$. The number of positive nodes $p$ is equal to $12$ and the total number of points in $\Delta \cap \partial B$ is $22$. Therefore the Euler number $e(\breve{X}(B)) = 22+12-68 = -34$, which is the Euler number of a threefold in the family $\MM{2}{2}$.

\begin{figure}
	\caption{Affine manifold model of $\MM{2}{2}$}
	\label{fig:2-2}
	\includegraphics[scale=1.2]{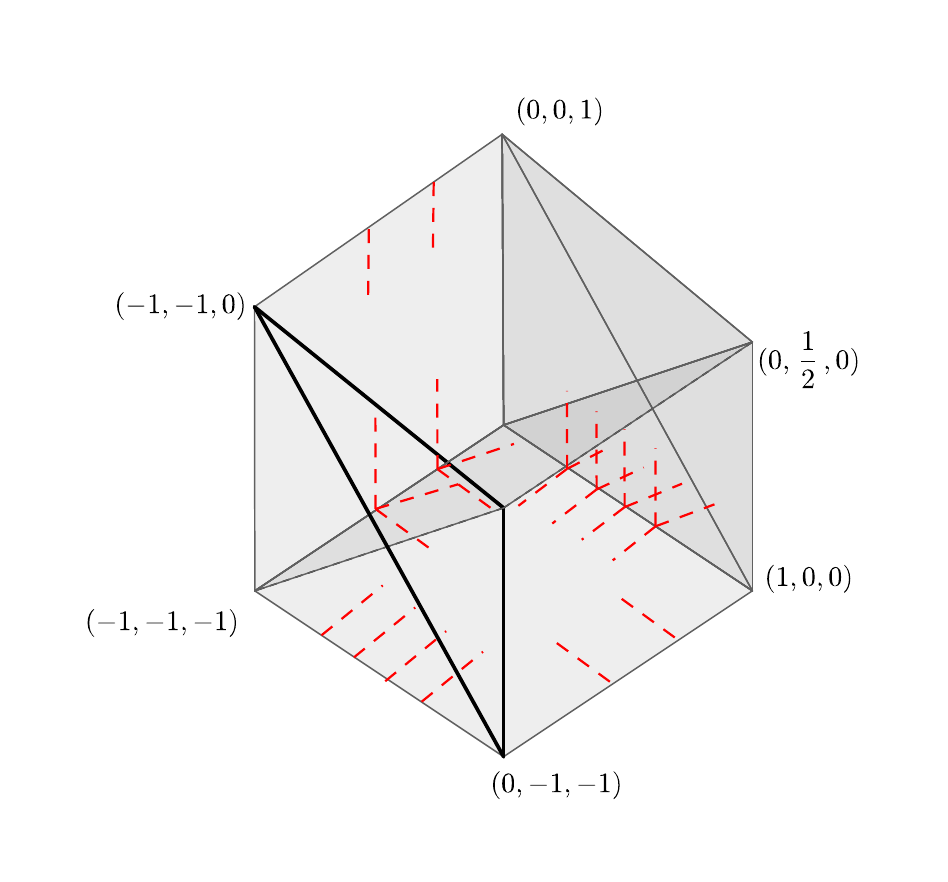}
\end{figure}

We compute the second Betti number using the Leray spectral sequence for the contraction map $\xi$ as usual. By the arguments used in \S\ref{sec:betti_numbers} (following \cite{G01}) we have that 

\begin{align*}
	H^0(\breve{X}_0(B),R^2\xi_\star\ZZ) = 0 && \textrm{and} &&	 H^0(\breve{X}_0(B),R^1\xi_\star\ZZ) = 0
\end{align*}

In fact, since the fan $\Sigma$ used to define the degeneration data is the fan for $\PP^3$ the argument to compute $H^1(\breve{X}_0(B),R^1\xi_\star\ZZ)$ is essentially the same as that used in \cite{G01}: the morphism $\cF \rightarrow {i_2}_\star i^\star_2 \cF$ is injective and ${i_2}_\star i^\star_2 \cF$ is equal to the constant sheaf $\ZZ$ away from a, non-zero and finite collection of points on $\PP^1$. Thus  $H^2(\breve{X}(B),\ZZ) = H^2(\breve{X}_0(B),\ZZ) = \ZZ^2$ by a straightforward computation.

\subsection{$\MM{2}{3}$}
\label{sec:MM23}

This example is very similar to that described in \S\ref{sec:MM21}. The Fano manifold $\MM{2}{3}$ is obtained by blowing up $B_2$ in an elliptic curve which is the intersection of two elements of $|\frac{-1}{2}K_{B_2}|$. By an identical analysis to that used in \S\ref{sec:MM21} we can construct an affine manifold $B$, shown in Figure~\ref{fig:MM23}, such that $b_2(\breve{X}(B)) = 2$. Note that there is a single edge $E$ of $P^\circ$ in this case such that $r(E^\star) > 1$ and -- as in \S\ref{sec:MM21} -- $E$ does not intersect any ray of $\Sigma$. Computing the Euler characteristic in this case we enumerate the special fibres of $\pi\colon \breve{X}(B) \rightarrow B$.

%where the singular locus of $B$ meets $\partial B$
\begin{itemize}
	\item There are $10\times 2 = 20$ points in $\Delta \cap \partial B$ (the focus-focus points on a pair of del~Pezzo surfaces of degree $2$).
	\item There are $4$ positive nodes.
	\item There are $2 \times 16 + 8 = 40$ negative nodes ($8$ induced by a section of $\cO(2)$ on $\PP(1,1,2)$, the other by a pair of section of $\cO(4)$ on $\PP^2$).
\end{itemize}

Thus we see that $e(\breve{X}(B)) = 20 + 4 - 40 = -16$ and $b_3$ is determined by the formula 
\[
2+2b_2(\breve{X}(B)) - b_3(\breve{X}(B)) = e(\breve{X}(B)),
\]
\noindent that is,
\[
\frac{1}{2}b_3(\breve{X}(B)) = 1 + 2 - \frac{1}{2}\times (-16) = 11,
\]
as expected. Similar analyses hold to compute the Euler numbers of the manifolds $\breve{X}(B)$ considered in \S\ref{sec:MM21} and \S\ref{sec:MM25}

\begin{figure}[h!]
	\caption{Affine manifold model of $\MM{2}{3}$}
	\label{fig:MM23}
	\includegraphics[scale=1.4]{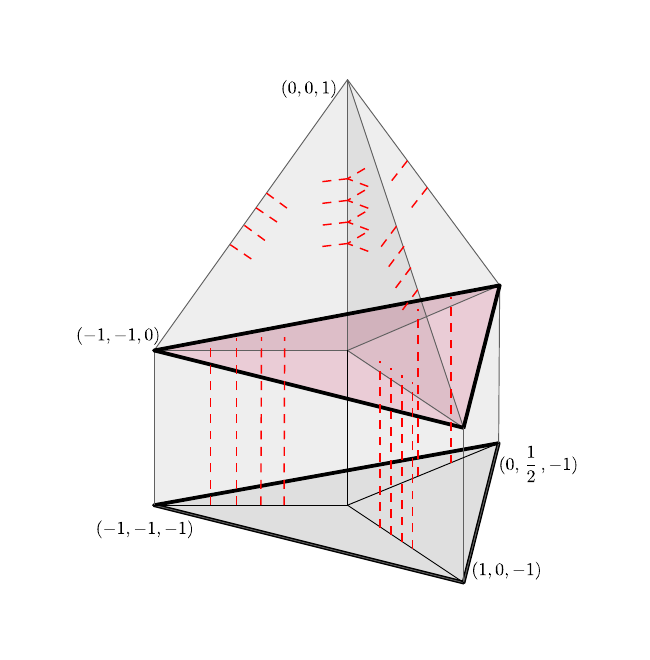}
\end{figure}

Note that as in \S\ref{sec:MM21} we can recover the blow up construction itself by collapsing a cylinder in the boundary. In fact using this observation we see directly that $\breve{X}(B)$ is homeomorphic to $\MM{2}{3}$.

\subsection{$\MM{2}{5}$}
\label{sec:MM25}

Consider a Fano manifold $X$ in the family $\MM{2}{5}$ is obtained by blowing up a plane cubic in $B_3$ (the cubic threefold). This example follows an essentially identical analysis to those of \S\ref{sec:MM21} and \S\ref{sec:MM25}. As such we do not recall the details of the computation of its Betti numbers here, but show, in Figure~\ref{fig:MM25}, the affine manifold $B$ constructed from the toric degeneration of $X$ obtained by considering $X$ as a divisor a toric variety, as described in \cite{CCGK}.

\begin{figure}[h!]
	\caption{Affine manifold model of $\MM{2}{5}$}
	\label{fig:MM25}
	\includegraphics[scale=1.4]{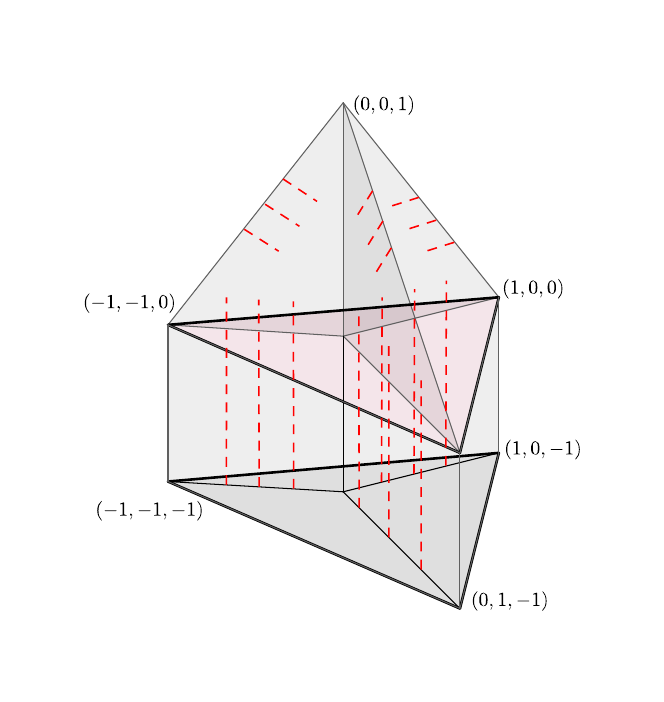}
\end{figure}

We provide the degree computation in this case, noting that essentially identical calculations apply to Examples~\ref{sec:MM21},\ref{sec:MM23}. The cycle $\pi^{-1}(\partial B)$ is the union of $3$ submanifolds of $\breve{X}(B)$. One of these $4$-manifolds is homeomorphic to $T^2 \times S^2$ and the other two are homeomorphic to cubic surfaces. Naming these cohomology classes $E$,~$D_1$ and $D_2$ respectively we see immediately that $D_1\cdot D_2 = 0$ and that $D_1\cdot E$ and $D_2 \cdot E$ are represented by $\pi^{-1}(\gamma_i)$, where $\gamma_i$ is the component of $(\partial B)_1$ meeting the images of $D_i$ for $i \in \{1,2\}$.

Following the argument made in \S\ref{sec:B1} and observing that $\breve{X}(B)$ is homeomorphic to the blow up of $B_3$, we have that $D_1 = D_2$ in $H^2(\breve{X}(B),\ZZ)$, and so $[\pi^{-1}(\partial B)]^3 = (2D_1 + E)^3$. Using the fact that $D_1^2 = 0$ the degree becomes $E^3 + 6D_1\cdot E^2$. It remains to compute $D_1\cdot E^2$, and $E^3$. These three may be computed from a topological push-off of $E$, and taking care over the orientations of each push-off.

\begin{rem}
	Note that $E$ is expected to be an exceptional divisor of the contraction of a threefold $\MM{2}{5}$ to a cubic threefold, and so the push-off used to compute the intersection number does not exist in the algebraic setting.
\end{rem}

\subsection{$\MM{3}{2}$}
\label{sec:MM32}

Let $X$ be a Fano manifold in the family $\MM{3}{2}$. Using the complete intersection model given in \cite{CCGK} we can construct a toric degeneration of $X$ and obtain an affine manifold as shown in Figure~\ref{fig:MM32}. The edge set $(\partial B)_1$ consists of precisely those edges of $P^\circ$ which do not intersect the singular locus, of which there are eight. The eight edges contained in $(\partial B)_1$ are marked in bold in Figure~\ref{fig:MM32}. The vertex set $(\partial B)_0$ consists of the four points $\{(1,0,-1),(0,1,-1),(0,1,0),(1,0,0)\}$.

\begin{figure}[h!]
	\caption{Affine manifold model of $\MM{3}{2}$}
	\label{fig:MM32}
	\includegraphics[scale=1.2]{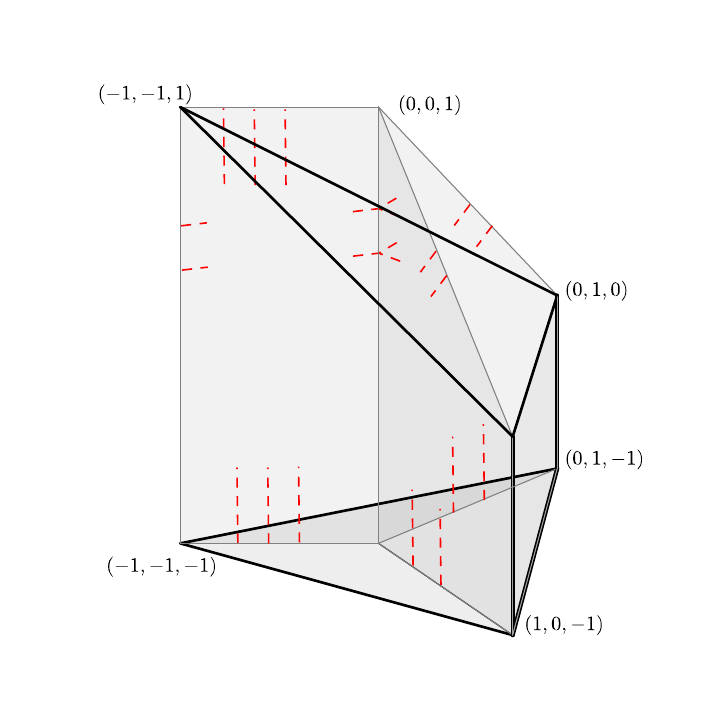}
\end{figure}

We compute the Leray spectral sequence of the map $\xi\colon \breve{X}(B) \rightarrow \breve{X}_0(B)$ using the techniques described in \S\ref{sec:betti_numbers}. First, using the spectral sequence determined by the stratification of $\breve{X}_0(B)$ we compute the dimensions of $H^i(\breve{X}_0(B),\QQ)$, the $E_1$ page of the corresponding spectral sequence is shown in \eqref{eq:stratification}. Alternatively -- taking small neighbourhoods of the strata -- we can regard \eqref{eq:stratification} as the $E_1$ page of a \u{C}ech-to-derived spectral sequence; in particular the terms which appear are groups of \u{C}ech cochains and the maps are \u{C}ech differentials.

\begin{equation}
\label{eq:stratification}
\xymatrix@R-2pc{
\QQ^6 \ar[r] & \QQ^6 \ar[r] & \QQ \\
0 \ar[r] & 0 \ar[r] & 0 \\
\QQ^3 \ar[r] & \QQ^3 \ar[r] & \QQ & \\
}
\end{equation}

Note that in this section all of our computations are over $\QQ$, since we only compute Betti numbers and do not study the possibility of torsion elements appearing in $H^3(\breve{X}(B),\ZZ)$. Let $V_i$, for $i \in \{1,2,3\}$, denote the three toric $6$-manifold pieces which form the maximal strata, let $V_{i,j}$ denote  the three toric surfaces obtained by intersecting these strata for $i,j \in\{1,2,3\}$ and $i \neq j$, and let $V_{1,2,3} := V_1\cap V_2\cap V_3$. Labelling the strata shown in Figure~\ref{fig:MM32}, we may assume that
\begin{enumerate}
	\item $V_1 \cong \PP_{\PP^1}(\cO^{\oplus 2}\oplus \cO(1))$, the blow up of $\PP^3$ in a line.
	\item $V_2 \cong \PP_{\PP^1}(\cO^{\oplus 2}\oplus \cO(1))$, and $V_2 \cong V_3$.
	\item $V_{1,2} \cong \FF_1$, and $V_{1,2} \cong V_{1,3}$.
	\item $V_{2,3} \cong \PP^1\times \PP^1$.
	\item $V_{1,2,3} \cong \PP^1$.
\end{enumerate}

The map $\QQ^3 \to \QQ$ in the bottom row is necessarily surjective; indeed the pullback map $\bigoplus_{i\neq j}H^0(V_{i,j},\QQ) \to H^0(\PP^1,\QQ)$ is non-zero on any factor. Similarly the map 
\[
\bigoplus_{i\neq j}H^2(V_{i,j},\QQ) \cong \QQ^6 \to \QQ \cong H^2(\PP^1,\QQ),
\]
is necessarily surjective. It remains to compute the map
\[
\varphi \colon \bigoplus^3_{i=1}H^2(V_i,\QQ) \cong \QQ^6 \to \QQ^6 \cong \bigoplus_{i\neq j}H^2(V_{i,j},\QQ).
\]
Observe that the pullback $\iota_{i,j}^\star \colon H^2(V_i,\QQ) \to H^2(V_{i,j},\QQ)$ is an isomorphism for any $i$ and $j \neq i$ in $\{1,2,3\}$. Consider the map $\ker(\varphi) \to H^2(V_{2,3})$ by projecting $\ker(\varphi) \to H^2(V_2,\QQ)$ and pulling back to $V_{2,3}$. This map is injective, as the maps $\iota_{i,j}^\star$ are injective. Thus any $\alpha \in \ker(\varphi)$ is determined by any of its three components. Moreover, it is straightforward to construct an embedding $H^2(V_{2,3},\QQ) \to \ker(\varphi)$, and hence $\dim\ker(\varphi) = 2$. Thus the $E_2$ page of the spectral sequence has the following form:
\[
\xymatrix@R-2pc{
\QQ^2 & \QQ & 0 \\
0 & 0 & 0 \\
\QQ & 0 & 0. \\
}
\]

Hence we have that $b^0(\breve{X}_0(B)) = 1$,~$b^2(\breve{X}_0(B)) = 2$,~$b^3(\breve{X}_0(B)) = 1$, and all other Betti numbers vanish. Note that we can interpret a generating element in $H^3(\breve{X}_0(B))$ geometrically: consider the subspace of $\breve{X}_0(B)$ corresponding to the three `top' (or `bottom') faces. This is isomorphic to the space formed by gluing three copies of $\PP^2$ cyclically along co-ordinate lines $L^i_0$ and $L^i_1$ for $i \in \{1,2,3\}$. Fixing a homotopy from $L^1_0$ to $L^1_1$ determines a singular chain with image homeomorphic to $S^2\times I$, where $I \subset \RR$ is an interval. Since $L^1_1$ is identified with $L^2_0$, we can choose a homotopy from $L^2_0$ to $L^2_1$. Continuing in this fashion we obtain a map from $S^2\times S^1 \to \breve{X}_0(B)$ which generates $H_3(\breve{X}_0(B))$.

We next observe that the groups $H^0(R^i\xi_\star\QQ)$ vanish for $i = 1$ or $2$, following the proof of Proposition~\ref{pro:Leray--Serre}. To compute $H^1(R^1\xi_\star\QQ)$ we use the short exact sequence
\[
\xymatrix{
0 \ar[r] & R^1\xi_\star\QQ \ar[r] & {i_1}_\star i_1^\star R^1\xi_\star\QQ \ar[r] & \cF \ar[r] & 0.
}
\]

The corresponding long exact sequence gives
\[
\xymatrix{
	0 \ar[r] & H^1(R^1\xi_\star\QQ) \ar[r] & H^1({i_1}_\star i_1^\star R^1\xi_\star\QQ) \ar[r] & H^1(\cF),
}
\]
and computing $H^1({i_1}_\star i_1^\star R^1\xi_\star\QQ)$ (and noting the departure of the calculation at this point from that appearing in \S\ref{sec:betti_numbers}) we observe that the sheaf ${i_1}_\star i_1^\star R^1\xi_\star\QQ$ is the sum of three sheaves $\cG_i$, $i \in \{1,2,3\}$, each supported on a different toric surface. The sheaf $\cG_1$ -- corresponding to the slab with associated toric variety  $\PP^1\times \PP^1$ -- is constant away from a curve defined by the singular locus. The sheaves $\cG_2$ and $\cG_3$ -- corresponding to the slabs with associated toric varieties $\FF_1$ -- are constant away from the union of a pullback of a conic in $\PP^2$ (determined by the singular locus) and the exceptional curve. Indeed, since the edges of $P^\circ$ corresponding to the exceptional divisors in each copy of $\FF_1$ lie in $(\partial B)_1$, fibres of $\xi$ over points in these divisors are singletons. Hence we have that while $H^1(\cG_1) = 0$, $H^1(\cG_i) = \ZZ$ for $i \in \{2,3\}$. Indeed,
\[
H^1(\cG_i) = H^1_c(\FF^1 \setminus (C\cup E)) \cong H_3(\FF^1 \setminus (C\cup E))
\]
by Poincar\'{e} duality, where $E$ is the exceptional curve of $p \colon \FF_1 \rightarrow \PP^1$. However $H_3(\FF^1 \setminus (C\cup E)) \cong H_3(\PP^2 \setminus(p(C) \cup \{pt\},\QQ)) \cong \QQ$, and is generated by a sphere containing the deleted point. Similarly, we can compute 
\[
H^1(\cF) \cong H^1_c(\PP^1\setminus \{\textrm{2 points}\},\QQ) \cong H_1(\CC^\star,\QQ) \cong  \QQ,
\]
and observe that the (horizontal) map
\[
\xymatrix{
H^1({i_1}_\star i_1^\star R^1\xi_\star\QQ) \ar^{\sim}[d] \ar[r] & H^1(\cF) \ar^{\sim}[d] \\
\QQ^2 \ar[r] & \QQ
}
\]
is zero. Thus $H^1(\cG_i) \cong \ZZ^2$ for $i \in \{2,3\}$. Consequently the $E_2$ page of the Leray spectral sequence associated to $\xi$ has the following form:
\[
\xymatrix@R-2pc{
\QQ & & & \\
0 & \star & & \\
0 & \QQ^2 \ar^{d_2}[rrd] & & \\
\QQ & 0 & \QQ^2 & \QQ \\
}
\]
We still need to determine the rank of the map $d_2$. Using the edge homomorphisms for the Leray spectral sequence we have that $d_2 \neq 0$ if and only if the map $\xi^\star \colon H^3(\breve{X}_0(B),\QQ) \rightarrow H^3(\breve{X}(B),\QQ)$ is zero. However note that by anti-commutativity the cup product on vanishes on  $H^3(\breve{X}_0(B),\QQ) \cong \QQ$ and thus, if $\alpha$ is a class generating $H^3(\breve{X}_0(B),\QQ)$ and $\beta \in H_3(\breve{X}(B),\QQ)$ is any class, $\xi_\star(\beta)\frown \alpha = 0$. Thus, using the projection formula,
\[
\xi^\star(\alpha) \frown \beta = \xi^\star(\alpha \frown \xi_\star(\beta) ) = \xi^\star(0) = 0.
\]
Since the cup product is non-degenerate on manifolds, $\xi^\star$ vanishes, the morphism $d_2$ has rank one, and $b^2(\breve{X}(B)) = 3$; as expected.

\subsection{$\MM{3}{4}$}
\label{sec:MM34}

The calculation of the second Betti number in the case $\MM{3}{4}$ is identical to that of $\MM{3}{2}$ and we do not repeat that calculation here. The affine manifold model of this Fano threefold is shown in Figure~\ref{fig:MM34}. We have that $(\partial B)_1$ consists of the edges which do not meet $\Delta \subset B$, with the exception of $[(1,0,-1),(1,0,0)]$ which is not contained in $(\partial B)_1$. The vertex set $(\partial B)_0$ is equal to $\{,(1,1,0),(1,1,-1),(-1,-1,0),(-1,-1,-1)\}$.

\begin{figure}[h!]
	\caption{Affine manifold model of $\MM{3}{4}$}
	\label{fig:MM34}
	\includegraphics[scale=1.2]{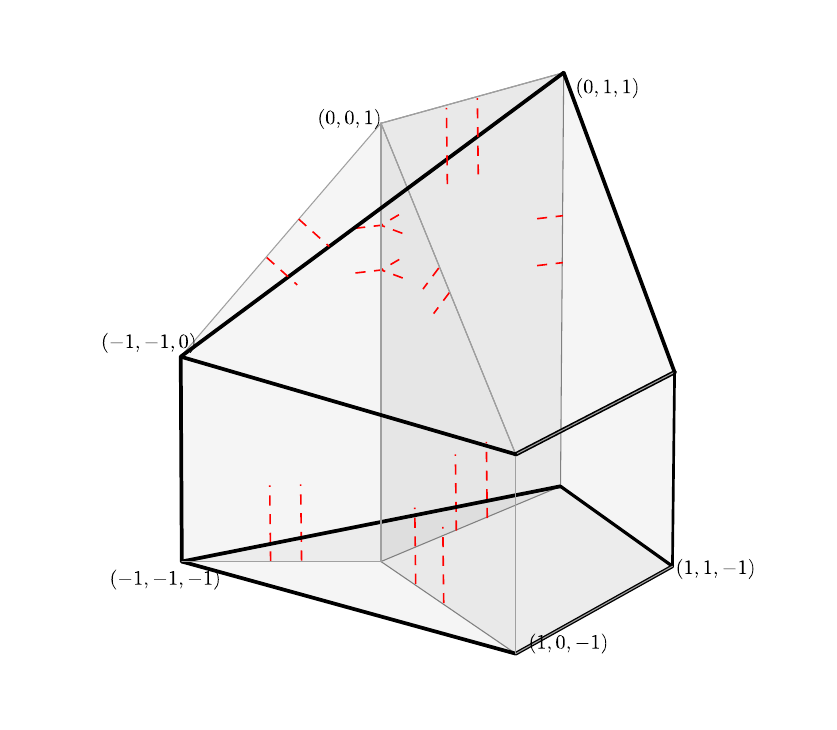}
\end{figure}

\subsection{$\MM{3}{5}$}
\label{sec:MM35}

Let $X$ be a Fano manifold in the family $\MM{3}{5}$. Using the complete intersection model given in \cite{CCGK} we can construct a toric degeneration of $X$ and obtain an affine manifold as shown in Figure~\ref{fig:MM35}.

\begin{figure}[h!]
	\caption{Affine manifold model of $\MM{3}{5}$}
	\label{fig:MM35}
	\includegraphics[scale=1.2]{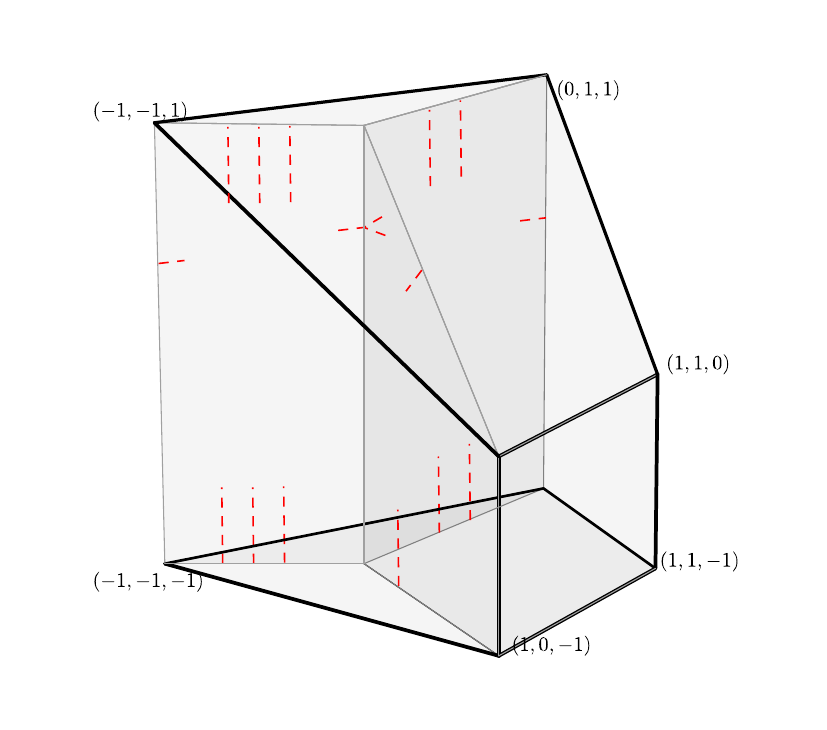}
\end{figure}

We calculate the Betti numbers using the same method as for Examples~\ref{sec:MM32} and \ref{sec:MM34}. Computing the Betti numbers of $\breve{X}_0(B)$ via the usual stratification we find the $E_1$ page:

\[
\xymatrix@R-2pc{
\QQ^7 \ar[r] & \QQ^6 \ar[r] & \QQ \\
0 \ar[r] & 0 \ar[r] & 0 \\
\QQ^3 \ar[r] & \QQ^3 \ar[r] & \QQ & \\
}
\]

Computing the differentials on this page we obtain the following $E_2$ page.

\[
\xymatrix@R-2pc{
	\QQ^2 & 0 & 0 \\
	0 & 0 & 0 \\
	\QQ & 0 & 0 & \\
}
\]
Note that now the calculation proceeds as in Example~\ref{sec:MM32}, except that $H^1(R^1\xi_\star\QQ) \cong \QQ$ and the map 
\[
\QQ \cong H^1(R^1\xi_\star\QQ) \rightarrow H^3(\xi_\star\QQ) \cong 0
\]
is necessarily trivial.

\subsection{$\MM{4}{2}$}
\label{sec:MM42}

Our model of a Fano manifold $X$ in the family $\MM{4}{2}$ is slightly different to the preceding examples, and shown in Figure~\ref{fig:MM42}. Indeed, to compute the Betti numbers of $\breve{X}(B)$ for $B$ shown in Figure~\ref{fig:MM42} we use a modified version of the map $\xi$. Rather than decompose $P^\circ$ along $\Sigma$, which would give $\breve{X}_0(B)$ four irreducible components, we divide $P^\circ$, indicated in Figure~\ref{fig:MM42}, containing all but one segment of the singular locus $\Delta \subset B$. Adapting the construction of $\xi$ there is a map $\xi' \colon \breve{X}_0(B) \rightarrow Y$ where $Y$ has two irreducible components, one (manifestly) toric (corresponding to the half of $P^\circ$ containing no singular locus), and one other, which is isomorphic to $\PP^1\times\PP^1\times\PP^1$. By now familiar arguments we see that $h^0(R^1\xi'_\star\QQ) = h^1(R^1\xi'_\star\QQ) = 0$ and $h^0(R^2\xi'_\star\QQ) = 0$, and thus $H^2(\breve{X}(B),\QQ) \cong H^2(Y,\QQ)$. However, filtering $Y$ by its irreducible components we obtain a spectral sequence with $E_1$ page:

\[
\xymatrix@R-2pc{
	\QQ^6 \ar[r] & \QQ^2 \ar[r] & 0 \\
	0 \ar[r] & 0 \ar[r] & 0 \\
	\QQ^2 \ar[r] & \QQ \ar[r] & 0 & \\
}
\]
From which, since the morphism $\QQ^6 \rightarrow \QQ^2$ must be surjective, we see that $H^2(Y,\QQ) \cong \QQ^4$.

\begin{figure}
	\caption{Affine manifold model of $\MM{4}{2}$}
	\label{fig:MM42}
	\includegraphics[scale=1.5]{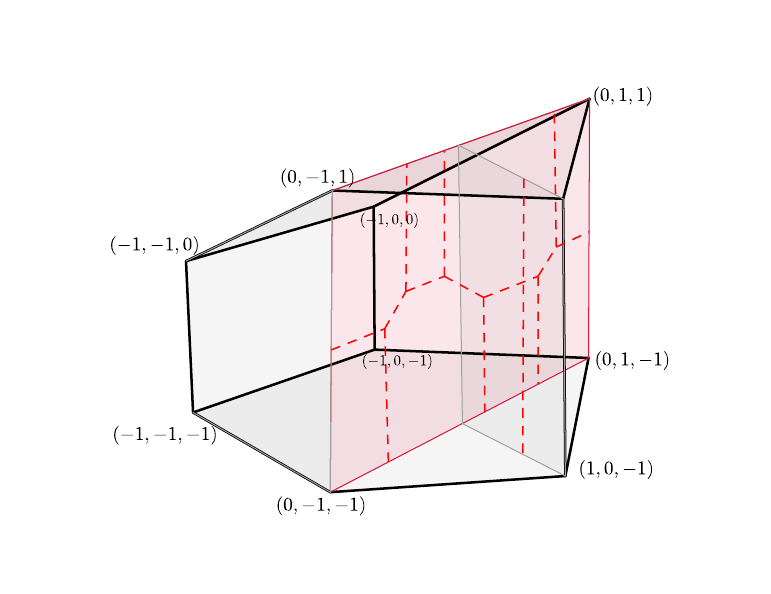}
\end{figure}

\subsection{$\MM{5}{1}$}
\label{sec:MM51}

Let $X$ be a Fano manifold in the family $\MM{5}{1}$. As in the examples above, we can use the complete intersection model given in \cite{CCGK} we can construct a toric degeneration of $X$ and obtain an affine manifold as shown in Figure~\ref{fig:MM51}.

\begin{figure}
	\caption{Affine manifold model of $\MM{5}{1}$}
	\label{fig:MM51}
	\includegraphics[scale=1.2]{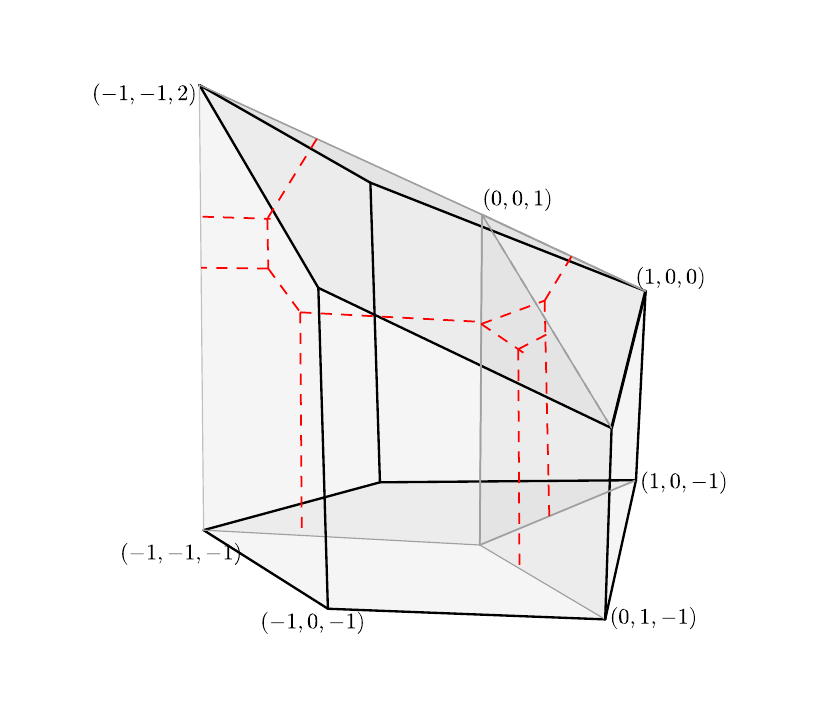}
\end{figure}

To compute the Betti numbers of this manifold we use the (usual) map $\xi \colon \breve{X}(B) \rightarrow \breve{X}_0(B)$. The computation then proceeds similarly to Example~\ref{sec:MM35}, Betti numbers of $\breve{X}_0(B)$ can be read off the $E_2$ page of the spectral sequence corresponding to its (toric) stratification.

\[
\xymatrix@R-2pc{
	\QQ^2 & 0 & 0 \\
	0 & 0 & 0 \\
	\QQ & 0 & 0 & \\
}
\]

Using this, and following the analysis in Example~\ref{sec:MM32} to compute $H^1(R^1\xi_\star\QQ)$ we obtain the $E_2$ page of the Leray spectral sequence

\[
\xymatrix@R-2pc{
	\QQ & & & \\
	0 & \star & & \\
	0 & \QQ^2 & & \\
	\QQ & 0 & \QQ^3 & 0 \\
}
\]

Since there are no non-trivial morphisms which can affect terms appearing in $H^2(\breve{X}(B),\QQ)$ we have that $b_2(\breve{X}(B)) = 5$, as expected.

%% file: appendices.tex
% !TEX root = fano_manifold_topology.tex

%----------------------------------------------------------------------
\section{Torus fibrations}
\label{sec:torus_fibrations}
%----------------------------------------------------------------------

In this section we recall the construction of a torus fibration over an affine manifold and the compactification of this fibration over the discriminant locus in dimension three. This construction is given in detail in \cite{G01,DBranes09}, see also \cite{CBM09}. Throughout this section we use the letters $p$ and $n$ to denote the numbers of positive and negative nodes respectively.

\subsection{Positive nodes}
\label{sec:positive_nodes}

The full construction of the torus fibration around a positive node appears in \cite{G01} -- where it is called a $(1,2)$-fibration -- as well as in \cite{CBM09}. We do not recall the full definition here, but describe the topology of the singular fibres.

Recall that a positive node in an an affine manifold (with boundary and singularities) is a point $p \in \Delta$ such that, given a point $b \in B$ not contained in the singular locus $\Delta$ the monodromy matrices, given a suitable basis of $T_bB$, are as follows:

\begin{align}
\label{eq:mono_matrices}	
	\begin{pmatrix}
		1 & 0 & 1 \\
		0 & 1 & 0 \\
		0 & 0 & 1
	\end{pmatrix}
	&&
	\begin{pmatrix}
		1 & 1 & 0 \\
		0 & 1 & 0 \\
		0 & 0 & 1
	\end{pmatrix}
	&&
	\begin{pmatrix}
		1 & -1 & -1 \\
		0 & 1 & 0 \\
		0 & 0 & 1
	\end{pmatrix}
\end{align}

We observe that these preserve a common one-dimensional subspace. Compactifying the fibration $\pi \colon T^\star B_0/\breve{\Lambda} \rightarrow B_0$ using the local model given in \cite{G01,CBM09}, the fibres $\pi^{-1}(q)$ for various points $q$ 
\begin{enumerate}
	\item $q$ generic: $\pi^{-1}(q)$ is $T^3$.
	\item $q$ generic in $\Delta$: $\pi^{-1}(q)$ is $S^1\times I_1$, where $I_1$ is the pinched torus.
	\item $q$ is the trivalent point: $\pi^{-1}(q)$ is homeomorphic to $(S^1\times T^2) /(\{x\} \times T^2)$, for some $x \in S^1$.
\end{enumerate}

It is then easy to compute the Euler characteristic of the fibration in a neighbourhood of a positive node of $B$. 

\begin{lem}
	\label{lem:euler_number_positive}
	The Euler number of the fibre of $\pi$ over a negative node is $+1$.
\end{lem}

We note that the monodromy matrices of $H_1$ of the fibres of a fibration are given (again in a suitable basis) by the inverse transpose of those appearing in \eqref{eq:mono_matrices}.

\subsection{Negative nodes}
\label{sec:negative_nodes}

Similarly to the construction of a compactification of the torus fibration near a positive node, the full construction of the torus fibration around a negative node appears in \cite{G01} -- where it is called a $(2,1)$-fibration -- as well as in \cite{CBM09}. Again, we do not recall the full definition here, but describe the topology of the singular fibres.

Recall that a negative node in an an affine manifold (with boundary and singularities) is a point $x \in \Delta$ such that, given a point $b \in B$ not contained in the singular locus $\Delta$ the monodromy matrices corresponding to loops around the branches of the singular locus meeting $x$, given a suitable basis of $T_bB$, are as follows:

\begin{align}
	\label{eq:mono_matrices_2}	
	\begin{pmatrix}
		1 & 0 & 0 \\
		0 & 1 & 0 \\
		1 & 0 & 1
	\end{pmatrix}
	&&
	\begin{pmatrix}
		1 & 0 & 0 \\
		1 & 1 & 0 \\
		0 & 0 & 1
	\end{pmatrix}
	&&
	\begin{pmatrix}
		1 & 0 & 0 \\
		-1 & 1 & 0 \\
		-1 & 0 & 1
	\end{pmatrix}
\end{align}

We enumerate the topology appearing as possible fibres of the compactification

\begin{enumerate}
	\item $q$ generic: $\pi^{-1}(q)$ is $T^3$.
	\item $q$ generic in $\Delta$: $\pi^{-1}(q)$ is $S^1\times I_1$, where $I_1$ is the pinched torus.
	\item $q$ is the trivalent point: $\pi^{-1}(q)$ is homeomorphic to $(S^1\times T^2) / \{S^1\times \Gamma\}$, where $\Gamma$ is the union of two circles which jointly form a basis of $H_1(T^2,\ZZ)$.
\end{enumerate}

It is easy to compute the homology groups of the fibre over the negative vertex, and consequently compute the Euler number of this fibre.

\begin{lem}
	\label{lem:euler_number_negative}
		The Euler number of the fibre of $\pi$ over a negative node is $-1$.
\end{lem}

\section{The contraction map}
\label{sec:contraction}

The analysis of the Betti numbers of $\breve{X}(B)$ for an affine manifold $B$ obtained from Construction~\ref{cons:slabs_to_affine_structure} relies heavily on a map
\[
\xi \colon \breve{X}(B) \rightarrow \breve{X}_0(B),
\]
analogous to the map $\xi$ appearing in \cite{G01}. In this section we define $\xi$ and describe its fibres over points of $B$.

\begin{rem}
	We remark that if we carefully define the map induced by a polyhedral degeneration the map $\xi$ is the usual contraction mapping from the general fibre to the special fibre. However, rather than using this as the definition of $\xi$ we use a similar treatment to that given in~\cite{G01}.
\end{rem}

Given a point $b \in B_0$ (possibly in $\partial B$), such that the minimal stratum $\sigma$ of the decomposition of $P^\circ$ given by $\Sigma$ containing $b$ has dimension $d$, the fibre $\pi^{-1}(b) := T^\star_b B/\breve{\Lambda}$, and there is a canonical inclusion $T_b\sigma \rightarrow T_bB$, giving a projection $T_b^\star B \rightarrow T^\star_b\sigma$. This projection descends to the fibre of $\pi$ and maps $\pi^{-1}(b)$ to a possibly lower dimensional torus, obtained as a quotient of $T^\star_b\sigma$ by the restriction of $\breve{\Lambda}$. Thus we have defined a map
\[
\xi_0 \colon \pi^{-1}(B_0) \rightarrow \breve{X}_0(B)
\]
which we now compactify over $\Delta$. In fact, given a point $b' \in \Delta$, every vanishing cycle of the fibre $\pi^{-1}(b')$ is contained in the kernel of the projection  $T_b^\star B \rightarrow T^\star_b\sigma$, where $b$ is a general point of $B_0$ close to $b'$. Thus we can extend $\xi_0$ over $\Delta$: in fact this can be realized explicitly by defining $T^n$ actions on the fibres of $\pi$, following~\cite{G01}.

\begin{dfn}
	Define $\tilde{\Delta}$ to be the image under $\xi$ of the singular set of $\pi^{-1}(\Delta)$. $\tilde{\Delta}$ consists of a collection of topological surfaces, meeting the one-dimensional strata of $\breve{X}_0(B)$ in points or circles, depending on the Minkowski factorisation $J$. 
\end{dfn}

The definition of the map $\xi$ is in fact less useful in practice than the knowledge of its fibres over the various strata of $\breve{X}_0(B)$, and hence we also enumerate these in this section. In each of the following cases $x$ is a point in $\breve{X}_0(B)$ such that $\pi(x) \notin \partial B$; where we refer to the codimension of the smallest stratum containing $x$ as the codimension of $x$.

\begin{center}
	{\renewcommand{\arraystretch}{2}
		\begin{tabular}{c | c | c}
			 codimension of $\pi(x)$ & $x \in \tilde{\Delta}$ & $\xi^{-1}(x)$ \\ \hline \hline
			$0$ & no & point \\ \hline
			$1$ & no & $S^1$ \\ \hline
			$1$ & yes & point \\ \hline
			$2$ & no & $T^2$ \\ \hline
			$2$ & yes & point or $S^1$ \\ \hline
			$3$ & no & $T^3$ \\ \hline
		\end{tabular}
	}
\end{center}
\bigskip 

\noindent The fibre $\xi^{-1}(x)$ for a point in $\tilde{\Delta}$ mapping to a point of codimension two is an point if $\pi(x)$ is a positive node of $B$ and a copy of $S^1$ otherwise. The analogous list of those points which map to the boundary of $B$ is as follows.

\begin{center}
	{\renewcommand{\arraystretch}{2}
		\begin{tabular}{c | c| c}
			codimension of $\pi(x)$ & $x \in \tilde{\Delta}$ & $\xi^{-1}(x)$ \\ \hline \hline
			$1$ & no & point \\ \hline
			$2$ & no & $S^1$ or point \\ \hline
			$2$ & yes & point \\ \hline
			$3$ & no & $T^2$ \\ \hline
		\end{tabular}
	}
\end{center}

The fibre $\xi^{-1}(x)$, for $p$ such that $\pi(x) \in (\partial B)_1$, is a point, and $\xi^{-1}(x)$ is homeomorphic to $S^1$ otherwise.

%% file: tables_of_invariants.tex
\begin{center}
%\footnotesize
\begin{longtable}{lcccccc}%{>{\hspace{0.8ex}}lr<{\hspace{2ex}}p{9.5cm}p{6.3cm}r<{\hspace{5ex}}}
\caption{Expected torus fibrations for 3-dimensional Fano manifolds ($b_2 = 1$).}\\

\toprule 
\multicolumn{1}{c}{{Name}}&
\multicolumn{1}{c}{{PALP ID}}&
\multicolumn{1}{c}{{Degree}} & 
\multicolumn{1}{c}{{$p$}} &
\multicolumn{1}{c}{{$n$}} & 
\multicolumn{1}{c}{{$\chi$}} & 
\multicolumn{1}{c}{{Notes}} 
\\ \midrule \endfirsthead

\multicolumn{7}{c}{{\tablename\ \thetable{}: Topological torus fibrations -- continued from previous page}} \\ \addlinespace[1.7ex] \midrule
\multicolumn{1}{c}{{Name}}&
\multicolumn{1}{c}{{PALP ID}}&
\multicolumn{1}{c}{{Degree}} &
\multicolumn{1}{c}{{$p$}} &
\multicolumn{1}{c}{{$n$}} & 
\multicolumn{1}{c}{{$\chi$}} & 
\multicolumn{1}{c}{{Notes}}
\\ \midrule\endhead

\midrule \multicolumn{7}{c}{{Continued on next page}} \endfoot

\bottomrule \endlastfoot

$V_2$ & n/a & 2 & $20$ & $144$ & $-100$ & Method $2$, see \S{\ref{sec:V2}}\\ \addlinespace[1.3ex] 
\rowcolor[gray]{0.95}
$V_4$ & $4311$ & $4$ & $16$ & $96$ & $-56$ & \\ \addlinespace[1.3ex]
$V_6$ & $4286$ & $6$ & $6$ & $66$ & $-36$ & \\ \addlinespace[1.3ex] 
\rowcolor[gray]{0.95}
$V_8$ & $4250$ & $8$ & $0$ & $48$ & $-24$ & \\ \addlinespace[1.3ex]
$B_1$ & n/a & $8$ & $6$ & $66$ & $-38$ & Method $2$, see \S{\ref{sec:B1}}\\ \addlinespace[1.3ex]
\rowcolor[gray]{0.95}
$V_{10}$ & $3964$ & $10$ & $8$ & $48$ & $-16$ & \\ \addlinespace[1.3ex]
$V_{12}$ & $3874$ & $12$ & $2$ & $36$ & $-10$ & see \S{\ref{sec:V12}}\\ \addlinespace[1.3ex]
\rowcolor[gray]{0.95}
$V_{14}$ & $3218$ & $14$ & $10$ & $40$ & $-6$ &\\ \addlinespace[1.3ex]
$V_{16}$ & $3031$ & $16$ & $6$ & $32$ & $-2$ & see \S{\ref{sec:V16}} \\ \addlinespace[1.3ex] 
\rowcolor[gray]{0.95}
$B_2$ & $427$ & $16$ & $8$ & $48$ & $-16$ & \\ \addlinespace[1.3ex]
$V_{18}$ & $2702$ & $18$ & $4$ & $28$ & $0$ & \\ \addlinespace[1.3ex] 
\rowcolor[gray]{0.95}
$V_{22}$ & $1886$ & $22$ & $10$ & $30$ & $4$ &  see \S{\ref{sec:V22}} \\ \addlinespace[1.3ex]
$B_3$ & $231$ & $24$ & $6$ & $36$ & $-6$ &  \\ \addlinespace[1.3ex] 
\rowcolor[gray]{0.95}
$B_4$ & $197$ & $32$ & $0$ & $24$ & $0$ &  \\ \addlinespace[1.3ex]
$B_5$ & $67$ & $40$ & $4$ & $24$ & $4$ & \\ \addlinespace[1.3ex] 
\rowcolor[gray]{0.95}
$Q_3$ & $3$ & $54$ & $4$ & $24$ & $4$ & \\ \addlinespace[1.3ex]
$\PP^3$ & $0$ & $64$ & $4$ & $24$ & $4$ & smooth toric 
\end{longtable}

\bigskip
\begin{longtable}{lcccccc}%{>{\hspace{0.8ex}}lr<{\hspace{2ex}}p{9.5cm}p{6.3cm}r<{\hspace{5ex}}}
	\caption{Topological torus fibrations ($b_2 = 2$).}\\
	
	\toprule 
	\multicolumn{1}{c}{{Name}}&
	\multicolumn{1}{c}{{PALP ID}}&
	\multicolumn{1}{c}{{Degree}} & 
	\multicolumn{1}{c}{{$p$}} &
	\multicolumn{1}{c}{{$n$}} & 
	\multicolumn{1}{c}{{$\chi$}} & 	
	\multicolumn{1}{c}{{Notes}} 
	\\ \midrule \endfirsthead
	
	\multicolumn{7}{c}{{\tablename\ \thetable{}: Expected torus fibrations for 3-dimensional Fano manifolds -- continued from previous page}} \\ \addlinespace[1.7ex] \midrule
	\multicolumn{1}{c}{{Name}}&
	\multicolumn{1}{c}{{PALP ID}}&
	\multicolumn{1}{c}{{Degree}} & 
	\multicolumn{1}{c}{{$p$}} &
	\multicolumn{1}{c}{{$n$}} & 
	\multicolumn{1}{c}{{$\chi$}} & 	
	\multicolumn{1}{c}{{Notes}}
	\\ \midrule\endhead
	
	\midrule \multicolumn{7}{c}{{Continued on next page}} \endfoot
	
	\bottomrule \endlastfoot
	
	$\MM{2}{1}$ & n/a & $4$ & $6$ & $66$ & $-38$ & Method $2$, see \S{\ref{sec:MM21}} \\ \addlinespace[1.3ex] 
	\rowcolor[gray]{0.95}
	$\MM{2}{2}$ & n/a & $6$ & $12$ & $68$ & $-34$ & Method $2$, see \S{\ref{sec:MM22}} \\ \addlinespace[1.3ex]
	$\MM{2}{3}$ & n/a & $8$ & $4$ & $40$ & $-16$ & Method $2$, see \S{\ref{sec:MM23}} \\ \addlinespace[1.3ex]
	\rowcolor[gray]{0.95}
	$\MM{2}{4}$ & $3963$ & $10$ & $10$ & $48$ & $-14$ & \\ \addlinespace[1.3ex]
	$\MM{2}{5}$ & $3776$ & $12$ & $3$ & $27$ & $-6$ & Method $2$, see \S{\ref{sec:MM25}} \\ \addlinespace[1.3ex]
	\rowcolor[gray]{0.95}
	$\MM{2}{6}$ & $3348$ & $12$ & $12$ & $48$ & $-12$ & see also \S{\ref{sec:V12}} \\ \addlinespace[1.3ex]

	$\MM{2}{7}$ & $3238$ & $14$ & $12$ & $40$ & $-4$ & \\ \addlinespace[1.3ex]
	\rowcolor[gray]{0.95}
	$\MM{2}{8}$ & $1968$ & $14$ & $12$ & $48$ & $-12$ & \\ \addlinespace[1.3ex]
	$\MM{2}{9}$ & $2605$ & $16$ & $8$ & $36$  & $-4$ & \\ \addlinespace[1.3ex]
	\rowcolor[gray]{0.95}
	$\MM{2}{10}$ & $3035$ & $16$ & $8$ & $32$ & $0$ & \\ \addlinespace[1.3ex]

	$\MM{2}{11}$ & $3008$ & $18$ & $6$ & $34$ & $-4$ & see \S{\ref{sec:MM211}} \\ \addlinespace[1.3ex] 
	\rowcolor[gray]{0.95}
	$\MM{2}{12}$ & $2355$ & $20$ & $0$ & $24$ & $0$ & \\ \addlinespace[1.3ex]	
	$\MM{2}{13}$ & $2353$ & $20$ & $4$ & $26$ & $2$ & \\ \addlinespace[1.3ex] 
	\rowcolor[gray]{0.95}
	$\MM{2}{14}$ & $2352$ & $20$ & $8$ & $28$ & $4$ & \\ \addlinespace[1.3ex]	
	$\MM{2}{15}$ & $910$ & $22$ & $10$ & $36$ & $-2$ & see also $1385$,$1598$\\ \addlinespace[1.3ex] 
	\rowcolor[gray]{0.95}
	$\MM{2}{16}$ & $1519$ & $22$ & $6$ & $28$ & $2$ & see also $1484$, $1903$\\ \addlinespace[1.3ex]	
	$\MM{2}{17}$ & $1096$ & $24$ & $8$ & $28$ & $4$ & \\ \addlinespace[1.3ex] 
	\rowcolor[gray]{0.95}
	$\MM{2}{18}$ & $1032$ & $24$ & $8$ & $30$ & $2$ & \\ \addlinespace[1.3ex]	
	$\MM{2}{19}$ & $1108$ & $26$ & $2$ & $24$ & $2$ & see also $690$ \\ \addlinespace[1.3ex] 
	\rowcolor[gray]{0.95}
	$\MM{2}{20}$ & $1109$ & $26$ & $6$ & $24$ & $6$ & see also $1098$ \\ \addlinespace[1.3ex]
	$\MM{2}{21}$ & $730$ & $28$ & $6$ & $24$  & $6$ & \\ \addlinespace[1.3ex] 
	\rowcolor[gray]{0.95}
	$\MM{2}{22}$ & $413$ & $30$ & $6$ & $24$  & $6$ & \\ \addlinespace[1.3ex]
	$\MM{2}{23}$ & $410$ & $30$ & $4$ & $24$  & $4$ & \\ \addlinespace[1.3ex] 
	\rowcolor[gray]{0.95}
	$\MM{2}{24}$ & $411$ & $30$ & $6$ & $24$ & $6$ & \\ \addlinespace[1.3ex]
	$\MM{2}{25}$ & $198$ & $32$ & $4$ & $24$ & $4$ & \\ \addlinespace[1.3ex] 
	\rowcolor[gray]{0.95}
	$\MM{2}{26}$ & $201$ & $34$ & $6$ & $24$ & $6$ & see also polytope $412$\\ \addlinespace[1.3ex]
	$\MM{2}{27}$ & $70$ & $38$ & $6$ & $24$ & $6$ & \\ \addlinespace[1.3ex] 
	\rowcolor[gray]{0.95}
	$\MM{2}{28}$ & $68$ & $40$ & $4$ & $24$ & $4$ & \\ \addlinespace[1.3ex]
	$\MM{2}{29}$ & $71$ & $40$ & $6$ & $24$ & $6$ & \\ \addlinespace[1.3ex] 
	\rowcolor[gray]{0.95}
	$\MM{2}{30}$ & $22$ & $46$ & $6$ & $24$ & $6$ & \\ \addlinespace[1.3ex]
	$\MM{2}{31}$ & $20$ & $46$ & $6$ & $24$ & $6$ & see also polytope $69$\\ \addlinespace[1.3ex] 
	\rowcolor[gray]{0.95}
	$\MM{2}{32}$ & $155$ & $48$ & $6$ & $24$ & $6$ & see \S{\ref{sec:MM232}} (see also polytope $21$)\\ \addlinespace[1.3ex]
	$\MM{2}{33}$ & $6$ & $54$ & $6$ & $24$ & $6$ & smooth toric \\ \addlinespace[1.3ex]
	\rowcolor[gray]{0.95}
	$\MM{2}{34}$ & $4$ & $54$ & $6$ & $24$ & $6$ &  $\PP^2 \times \PP^1$ \\ \addlinespace[1.3ex]
	$\MM{2}{35}$ & $5$ & $56$ & $6$ & $24$ & $6$ & smooth toric \\ \addlinespace[1.3ex]
	\rowcolor[gray]{0.95}
	$\MM{2}{36}$ & $7$ & $62$ & $6$ & $24$ & $6$ &  smooth toric \\ \addlinespace[1.3ex] 
\end{longtable}

\bigskip
\begin{longtable}{lcccccc}%{>{\hspace{0.8ex}}lr<{\hspace{2ex}}p{9.5cm}p{6.3cm}r<{\hspace{5ex}}}
	\caption{Topological torus fibrations ($b_2 = 3$).}\\
	
	\toprule 
	\multicolumn{1}{c}{{Name}}&
	\multicolumn{1}{c}{{PALP ID}}&
	\multicolumn{1}{c}{{Degree}} & 
	\multicolumn{1}{c}{{$p$}} &
	\multicolumn{1}{c}{{$n$}} & 
	\multicolumn{1}{c}{{$\chi$}} & 	
	\multicolumn{1}{c}{{Notes}} 
	\\ \midrule \endfirsthead
	
	\multicolumn{7}{c}{{\tablename\ \thetable{}: Expected torus fibrations for 3-dimensional Fano manifolds -- continued from previous page}} \\ \addlinespace[1.7ex] \midrule
	\multicolumn{1}{c}{{Name}}&
	\multicolumn{1}{c}{{PALP ID}}&
	\multicolumn{1}{c}{{Degree}} & 
	\multicolumn{1}{c}{{$p$}} &
	\multicolumn{1}{c}{{$n$}} &
	\multicolumn{1}{c}{{$\chi$}} & 	 
	\multicolumn{1}{c}{{Notes}}
	\\ \midrule\endhead
	
	\midrule \multicolumn{7}{c}{{Continued on next page}} \endfoot

	\bottomrule \endlastfoot
	$\MM{3}{1}$ & $3349$ & $12$ & $16$ & $48$ & $-8$ & see also \S{\ref{sec:V12}} \\ \addlinespace[1.3ex] 
	\rowcolor[gray]{0.95}
	$\MM{3}{2}$ & $2790$ & $14$ & $2$ & $20$ & $2$ & Method $2$, see \S{\ref{sec:MM32}} \\ \addlinespace[1.3ex]
	$\MM{3}{3}$ & $2677$ & $18$ & $12$ & $34$ & $2$ & \\ \addlinespace[1.3ex] 
	\rowcolor[gray]{0.95}
	$\MM{3}{4}$ & $2543$ & $18$ & $2$ & $16$ & $4$ & Method $2$, see \S{\ref{sec:MM34}} \\ \addlinespace[1.3ex]
	$\MM{3}{5}$ & $1366$ & $20$ & $1$ & $11$ & $8$ & Method $2$, see \S{\ref{sec:MM35}} \\ \addlinespace[1.3ex] 
	\rowcolor[gray]{0.95}
	$\MM{3}{6}$ & $1937$ & $22$ & $10$ & $28$ & $6$ & \\ \addlinespace[1.3ex]
	$\MM{3}{7}$ & $1932$ & $24$ & $8$ & $26$ & $6$ & \\ \addlinespace[1.3ex] 
	\rowcolor[gray]{0.95}
	$\MM{3}{8}$ & $1932$ & $24$ & $10$ & $26$ & $8$ & \\ \addlinespace[1.3ex]
	$\MM{3}{9}$ & $373$ & $26$ & $8$ & $30$ & $2$ &  \\ \addlinespace[1.3ex] 
	\rowcolor[gray]{0.95}
	$\MM{3}{10}$ & $1112$ & $26$ & $8$ & $24$ & $8$ & \\ \addlinespace[1.3ex]
	$\MM{3}{11}$ & $729$ & $28$ & $6$ & $24$ & $6$ &  see also $731$,$723$\\ \addlinespace[1.3ex] 
	\rowcolor[gray]{0.95}
	$\MM{3}{12}$ & $737$ & $28$ & $8$ & $24$ & $8$ & \\ \addlinespace[1.3ex]
	$\MM{3}{13}$ & $420$ & $30$ & $8$ & $24$ & $8$ & \\ \addlinespace[1.3ex] 
	\rowcolor[gray]{0.95}
	$\MM{3}{14}$ & $202$ & $32$ & $6$ & $24$ & $6$ & \\ \addlinespace[1.3ex]
	$\MM{3}{15}$ & $419$ & $32$ & $8$ & $24$ & $8$ & \\ \addlinespace[1.3ex] 
	\rowcolor[gray]{0.95}
	$\MM{3}{16}$ & $212$ & $34$ & $8$ & $24$ & $8$ & \\ \addlinespace[1.3ex]	
	$\MM{3}{17}$ & $208$ & $36$ & $8$ & $24$ & $8$ & \\ \addlinespace[1.3ex] 
	\rowcolor[gray]{0.95}
	$\MM{3}{18}$ & $211$ & $36$ & $8$ & $24$ & $8$ & \\ \addlinespace[1.3ex]
	$\MM{3}{19}$ & $74$ & $38$ & $8$ & $24$ & $8$ &  \\ \addlinespace[1.3ex] 
	\rowcolor[gray]{0.95}
	$\MM{3}{20}$ & $79$ & $38$ & $8$ & $24$ & $8$ & \\ \addlinespace[1.3ex]
	$\MM{3}{21}$ & $213$ & $38$ & $8$ & $24$ & $8$ & \\ \addlinespace[1.3ex] 
	\rowcolor[gray]{0.95}
	$\MM{3}{22}$ & $75$ & $40$ & $8$ & $24$ & $8$ & \\ \addlinespace[1.3ex]
	$\MM{3}{23}$ & $76$ & $42$ & $8$ & $24$ & $8$ &  \\ \addlinespace[1.3ex] 
	\rowcolor[gray]{0.95}
	$\MM{3}{24}$ & $77$ & $42$ & $8$ & $24$ & $8$ & \\ \addlinespace[1.3ex]
	$\MM{3}{25}$ & $24$ & $44$ & $8$ & $24$ & $8$ &  smooth toric \\ \addlinespace[1.3ex] 
	\rowcolor[gray]{0.95}
	$\MM{3}{26}$ & $25$ & $46$ & $8$ & $24$ & $8$ & smooth toric\\ \addlinespace[1.3ex]
	$\MM{3}{27}$ & $30$ & $48$ & $8$ & $24$ & $8$ &  $\PP^1 \times \PP^1 \times \PP^1$, see \S{\ref{sec:MM232}} \\ \addlinespace[1.3ex] 
	\rowcolor[gray]{0.95}
	$\MM{3}{28}$ & $29$ & $48$ & $8$ & $24$ & $8$ & $\FF_1\times \PP^1$\\ \addlinespace[1.3ex]
	$\MM{3}{29}$ & $26$ & $50$ & $8$ & $24$ & $8$ &  smooth toric, see $176$ \\ \addlinespace[1.3ex] 
	\rowcolor[gray]{0.95}
	$\MM{3}{30}$ & $28$ & $50$ & $8$ & $24$ & $8$ & smooth toric, see $167$ \\ \addlinespace[1.3ex]
	$\MM{3}{31}$ & $27$ & $52$ & $8$ & $24$ & $8$ &  smooth toric \\ \addlinespace[1.3ex] 
\end{longtable}

\bigskip
\begin{longtable}{lcccccc}%{>{\hspace{0.8ex}}lr<{\hspace{2ex}}p{9.5cm}p{6.3cm}r<{\hspace{5ex}}}
	\caption{Topological torus fibrations ($b_2 = 4$).}\\
	
	\toprule 
	\multicolumn{1}{c}{{Name}}&
	\multicolumn{1}{c}{{PALP ID}}&
	\multicolumn{1}{c}{{Degree}} & 
	\multicolumn{1}{c}{{$p$}} &
	\multicolumn{1}{c}{{$n$}} & 
	\multicolumn{1}{c}{{$\chi$}} & 	 
	\multicolumn{1}{c}{{Notes}} 
	\\ \midrule \endfirsthead
	
	\multicolumn{7}{c}{{\tablename\ \thetable{}: Expected torus fibrations for 3-dimensional Fano manifolds -- continued from previous page}} \\ \addlinespace[1.7ex] \midrule
	\multicolumn{1}{c}{{Name}}&
	\multicolumn{1}{c}{{PALP ID}}&
	\multicolumn{1}{c}{{Degree}} & 
	\multicolumn{1}{c}{{$p$}} &
	\multicolumn{1}{c}{{$n$}} & 
	\multicolumn{1}{c}{{$\chi$}} & 	 
	\multicolumn{1}{c}{{Notes}}
	\\ \midrule\endhead
	
	%\midrule \multicolumn{5}{c}{{Continued on next page}} \endfoot
	
	%\bottomrule \endlastfoot
		
	$\MM{4}{1}$ & $1529$ & $24$ & $8$ & $24$ & $8$ & \\ \addlinespace[1.3ex]
	\rowcolor[gray]{0.95}
	$\MM{4}{2}$ & $667$ & $26$ & $0$ & $6$ & $10$ & Method $2$, see \S{\ref{sec:MM42}} \\ \addlinespace[1.3ex]
	$\MM{4}{3}$ & $734$ & $28$ & $8$ & $24$ & $8$ &  \\ \addlinespace[1.3ex]
	\rowcolor[gray]{0.95}
	$\MM{4}{4}$ & $740$ & $30$ & $10$ & $24$ & $10$ &  \\ \addlinespace[1.3ex]
	$\MM{4}{5}$ & $426$ & $32$ & $10$ & $24$ & $10$ &  \\ \addlinespace[1.3ex] 
	\rowcolor[gray]{0.95}
	$\MM{4}{6}$ & $425$ & $32$ & $10$ & $24$ & $10$ & \\ \addlinespace[1.3ex]
	$\MM{4}{7}$ & $423$ & $34$ & $10$ & $24$ & $10$ &\\ \addlinespace[1.3ex] 
	\rowcolor[gray]{0.95}
	$\MM{4}{8}$ & $424$ & $36$ & $10$ & $24$ & $10$ & polytopes $215,217$ give identical entries  \\ \addlinespace[1.3ex]
	$\MM{4}{9}$ & $216$ & $38$ & $10$ & $24$ & $10$ & \\ \addlinespace[1.3ex]
	\rowcolor[gray]{0.95}
	$\MM{4}{10}$ & $81$ & $40$ & $10$ & $24$ & $10$ & polytopes $214,402$ give identical entries  \\ \addlinespace[1.3ex]
	$\MM{4}{11}$ & $84$ & $42$ & $10$ & $24$ & $10$ & smooth toric \\ \addlinespace[1.3ex] 
	\rowcolor[gray]{0.95}
	$\MM{4}{12}$ & $82$ & $44$ & $10$ & $24$ & $10$ & smooth toric  \\ \addlinespace[1.3ex]
	$\MM{4}{13}$ & $83$ & $46$ & $10$ & $24$ & $10$ & smooth toric \\ \addlinespace[1.3ex] 
\end{longtable}

\bigskip
\begin{longtable}{lcccccc}%{>{\hspace{0.8ex}}lr<{\hspace{2ex}}p{9.5cm}p{6.3cm}r<{\hspace{5ex}}}
	\caption{Topological torus fibrations ($b_2 \geq 5$).}\\
	
	\toprule 
	\multicolumn{1}{c}{{Name}}&
	\multicolumn{1}{c}{{PALP ID}}&
	\multicolumn{1}{c}{{Degree}} & 
	\multicolumn{1}{c}{{$p$}} &
	\multicolumn{1}{c}{{$n$}} &  
	\multicolumn{1}{c}{{$\chi$}} & 	 
	\multicolumn{1}{c}{{Notes}} 
	\\ \midrule \endfirsthead
	
	\multicolumn{7}{c}{{\tablename\ \thetable{}: Expected torus fibrations for 3-dimensional Fano manifolds -- continued from previous page}} \\ \addlinespace[1.7ex] \midrule
	\multicolumn{1}{c}{{Name}}&
	\multicolumn{1}{c}{{PALP ID}}&
	\multicolumn{1}{c}{{Degree}} & 
	\multicolumn{1}{c}{{$p$}} &
	\multicolumn{1}{c}{{$n$}} &
	\multicolumn{1}{c}{{$\chi$}} & 
	\multicolumn{1}{c}{{Notes}}
	\\ \midrule\endhead
	
	\midrule \multicolumn{7}{c}{{Continued on next page}} \endfoot
	
	\bottomrule \endlastfoot
	$\MM{5}{1}$ & $2268$ & $28$ & $1$ & $5$ & $12$ & Method $2$, see \S{\ref{sec:MM51}} \\ \addlinespace[1.3ex]
	\rowcolor[gray]{0.95}
	$\MM{5}{2}$ & $219$ & $36$ & $12$ & $24$ & $0$ &  \\ \addlinespace[1.3ex]
	$\MM{5}{3}$ & $218$ & $36$ & $0$ & $0$ & $12$ & $\PP^1 \times dP_6$ \\ \addlinespace[1.3ex]
	\rowcolor[gray]{0.95}
	$\MM{6}{1}$ & $356$ & $30$ & $0$ & $0$ & $14$ & $\PP^1 \times dP_5$ \\ \addlinespace[1.3ex] 
	$\MM{7}{1}$ & $505$ & $24$ & $0$ & $0$ & $16$ & $\PP^1 \times dP_4$ \\ \addlinespace[1.3ex]
	\rowcolor[gray]{0.95}
	$\MM{8}{1}$ & $768$ & $18$ & $0$ & $0$ & $18$ & $\PP^1 \times dP_3$ \\ \addlinespace[1.3ex] 
	$\MM{9}{1}$ & n/a & $12$ & $0$ & $0$ & $20$ & $\PP^1 \times dP_2$  \\ \addlinespace[1.3ex]
	\rowcolor[gray]{0.95}
	$\MM{10}{1}$ & n/a & $6$ & $0$ & $0$ & $22$ & $\PP^1 \times dP_1$ \\ \addlinespace[1.3ex]
\end{longtable}

\end{center}